\documentclass[11pt,reqno]{amsart}
\usepackage{amssymb,amsfonts,amsthm,amscd,stmaryrd,dsfont,esint,upgreek,constants,todonotes}
\usepackage[mathcal]{euscript}

\usepackage[top=1.25in,bottom=1.25in, left=1in,right=1in]{geometry}
\theoremstyle{plain}
\newtheorem{thm}{Theorem}
\newtheorem{cor}{Corollary}
\newtheorem{lem}[cor]{Lemma}
\newtheorem{prop}[cor]{Proposition}

\theoremstyle{definition}
\newtheorem{definition}[cor]{Definition}
\newtheorem{remark}[cor]{Remark}
\newtheorem{example}[cor]{Example}

\numberwithin{cor}{section}
\numberwithin{equation}{section}

\newconstantfamily{C}{symbol=C}
\newconstantfamily{c}{symbol=c}
\newconstantfamily{O}{symbol=\Omega}

\newcommand{\R}{\mathbb{R}}
\newcommand{\Q}{\mathbb{Q}}
\newcommand{\Z}{\mathbb{Z}}
\newcommand{\N}{\mathbb{N}}
\newcommand{\B}{\mathcal{B}}

\renewcommand{\d}{d}
\newcommand{\Rd}{\mathbb{R}^\d}
\newcommand{\Zd}{\mathbb{Z}^\d}
\newcommand{\ep}{\varepsilon}

\newcommand{\E}{\mathbb{E}}
\newcommand{\Prob}{\mathbb{P}}

\newcommand{\indc}{\mathds{1}}
\newcommand{\Lip}{{\mathcal{L}}}
\renewcommand{\bar}{\overline}

\renewcommand{\hat}{\widehat}

\renewcommand{\Re}{\mathcal{R}}

\DeclareMathOperator*{\osc}{osc}

\DeclareMathOperator{\intr}{int}

\DeclareMathOperator{\USC}{USC}
\DeclareMathOperator{\LSC}{LSC}
\DeclareMathOperator{\BUC}{BUC}
\DeclareMathOperator{\Var}{var}
\DeclareMathOperator{\diam}{diam}
\DeclareMathOperator{\dist}{dist}
\DeclareMathOperator{\conv}{conv}
\DeclareMathOperator*{\esssup}{ess\,sup}
\DeclareMathOperator*{\essinf}{ess\,inf}

\begin{document}

\title[Error estimates and convergence rates for stochastic homogenization]{Error estimates and convergence rates for the\\ stochastic homogenization of Hamilton-Jacobi equations}

\author[S. N. Armstrong]{Scott N. Armstrong}
\address{Department of Mathematics\\ University of Wisconsin, Madison\\ 480 Lincoln Drive\\
Madison, Wisconsin 53706.}
\email{armstron@math.wisc.edu}
\author[P. Cardaliaguet]{Pierre Cardaliaguet}
\address{Ceremade (UMR CNRS 7534) \\ Universit\'e Paris-Dauphine \\ Place du Mar\'echal De Lattre De Tassigny \\ 75775 Paris CEDEX 16, France}
\email{cardaliaguet@ceremade.dauphine.fr}
\author[P. E. Souganidis]{Panagiotis E. Souganidis}
\address{Department of Mathematics\\ The University of Chicago\\ 5734 S. University Avenue
Chicago, Illinois 60637.}
\email{souganidis@math.uchicago.edu}
\date{\today}

\begin{abstract}
We present exponential error estimates and demonstrate an algebraic convergence rate for the homogenization of level-set convex Hamilton-Jacobi equations in i.i.d.~random environments, the first quantitative homogenization results for these equations in the stochastic setting. By taking advantage of a connection between the metric approach to homogenization and the theory of first-passage percolation, we obtain estimates on the fluctuations of the solutions to the approximate cell problem in the ballistic regime (away from the flat spot of the effective Hamiltonian). In the sub-ballistic regime (on the flat spot), we show that the fluctuations are governed by an entirely different mechanism and the homogenization may proceed, without further assumptions, at an arbitrarily slow rate. We identify a necessary and sufficient condition on the law of the Hamiltonian for an algebraic rate of convergence to hold in the sub-ballistic regime and show, under this hypothesis, that the two rates may be merged to yield comprehensive error estimates and an algebraic rate of convergence for homogenization.

Our methods are novel and quite different from the techniques employed in the periodic setting, although we benefit from previous works in both first-passage percolation and homogenization. The link between the rate of homogenization and the flat spot of the effective Hamiltonian, which is related to the nonexistence of correctors, is a purely random phenomenon observed here for the first time. 
\end{abstract}

\maketitle

\section{Introduction} \label{I}

We consider the Hamilton-Jacobi equation
\begin{equation} \label{HJq}
u^\ep_t + H\left(Du^\ep,\frac x\ep, \omega\right) = 0 \quad \mbox{in} \ \Rd \times (0,\infty),
\end{equation}
where the Hamiltonian $H=H(p,y,\omega)$ is level-set convex and coercive in $p$ and depends on an element $\omega$ of an underlying probability space $(\Omega,\mathcal F, \Prob)$. If the action of translation on $\Rd$ is stationary and ergodic with respect to the law of $H$, then, as $\ep \to 0$, the solutions $u^\ep=u^\ep(x,t,\omega)$ of~\eqref{HJq}, subject to appropriate initial conditions, converge $\Prob$-almost surely to the solution $u$ of the deterministic equation 
\begin{equation}\label{HJh}
u_t + \overline H(Du) = 0 \quad \mbox{in} \ \Rd \times (0,\infty)
\end{equation}
with the same initial conditions, where the effective Hamiltonian $\overline H$ is level-set convex, continuous and coercive. This fundamental theorem concerning the qualitative theory of stochastic homogenization of Hamilton-Jacobi equations was proved for convex Hamiltonians by one of the authors~\cite{S} (see also Rezakhanlou and Tarver~\cite{RT}) and, more recently, by two of the authors~\cite{ASo3} in the generality discussed here. 

In this paper, we present the first quantitative homogenization results for Hamilton-Jacobi equations in the stochastic setting. Throughout the paper we assume that the Hamiltonian $H=H(p,y,\omega)$ satisfies a finite range dependence hypothesis (a continuum analogue of ``i.i.d.") in its spatial dependence. This essentially means that, for some fixed distance $D> 0$, the values of $H(p,y,\cdot)$ for $y\in E$ are independent of those for $y\in F$ provided that $\dist(E,F)>D$. (Obviously we lose no generality by taking $D=1$.)

By a novel integration of probabilistic and pde techniques, we (i) obtain explicit estimates showing the probability of $|u^\ep(x,t,\omega) - u(x,t)| > \lambda$ decays exponentially in $\lambda^2$, and (ii) identify a necessary and sufficient condition for the almost sure, local uniform convergence $u^\ep \to u$ to proceed at an algebraic rate~$O(\ep^\alpha)$. The main results, including the precise assumptions, are stated in the next section. They essentially give the error estimates 
\begin{equation}\label{}
-O\!\left( \ep^{\frac18-\delta} \right) \leq u^\ep - u \leq O\!\left(\ep^{\frac15-\delta}\right) \quad \mbox{for all $\delta > 0$}
\end{equation}
for the homogenization of~\eqref{HJq}, where the first inequality depends on a supplemental assumption on the law of $H$. 

The difficulty in obtaining estimates on the fluctuations of $u^\ep(x,t,\cdot)-u(x,t)$ is due in part to the fact that the dependence of $u^\ep$ on $H$ is \emph{highly singular}. Understanding how the solutions depend on the random environment is very challenging. Difficulties of a similar nature occur, for example, in the theory of first passage percolation (see Kesten~\cite{K2} and Alexander~\cite{A}) and in the study of the fluctuations of the Lyapunov exponents for Brownian motion in Poissonian potentials (see Sznitman~\cite{SzEE} and W{\"u}thrich~\cite{W}). As far as we know, the only previous result on the oscillations of solutions of Hamilton-Jacobi equations in random media is found in the work of Rezakhanlou~\cite{R}, who gave structural conditions on $H$ in dimension $\d=1$ in which a central limit theorem holds. Such phenomena are not expected to appear in any dimension $\d \geq 2$ (see Remark~\ref{notsharp}).

Our arguments rely crucially on adaptations of some of the probabilistic techniques of~\cite{K2,A}, which are based on Azuma's inequality and the martingale method of bounded differences. This connection between first-passage percolation and the stochastic homogenization of~\eqref{HJq}, made explicit for the first time in this paper (as far as we know), arises naturally from an analogy between the passage time in percolation and solutions of the \emph{metric problem} (see~Remark~\ref{contform} below). Using arguments inspired from~\cite{K2,A}, we prove exponential error estimate and obtain rates of convergence for the homogenization of the metric problem. Then, by quantifying the new proof of homogenization recently introduced by two of the authors~\cite{ASo3}, we transform the estimates for the metric problem into error estimates for the approximate cell problem (see~\eqref{acp} below).

The rate of convergence of \emph{periodic} homogenization of Hamilton-Jacobi has been understood for some time and goes back to the work of Capuzzo-Dolcetta and Ishii~\cite{ICD}, who proved that $u^\ep$ and $u$ differ by at most $O(\ep^{\frac13})$. The periodic setting is much simpler to understand due to the fact that the cell problem has periodic solutions; that is, exact correctors exist. A quantitative version of the classical perturbed test function proof of homogenization due to Evans~\cite{E,E2} then yields the convergence rate. Our main results stated in Section~\ref{MR} do not encompass the periodic or almost periodic settings, since obviously an almost periodic function cannot be embedded into the random setting in such a way that it satisfies a finite range of dependence condition. However, as we show, our arguments yield a uniform rate of convergence in the almost periodic setting (see Section~\ref{APs}). 

In the stochastic environment, the situation is not only much more complicated but also qualitatively different from the periodic setting. It is, therefore, necessary to devise a new strategy, since the usual proof of periodic homogenization, which is based on exact correctors and can be quantified to yield a rate, does not generalize to random enviroments. Indeed, as Lions and Souganidis~\cite{LS1} demonstrated with an explicit example, exact correctors do not exist, in general, for stochastic Hamiltonians. The only known proofs of the qualitative homogenization of Hamilton-Jacobi equations in the stationary ergodic setting are based on an application of the subadditive ergodic theorem to certain subadditive quantities (e.g., the $m_\mu$'s below) and then showing that these quantities control, in an appropriate way, the solutions of~\eqref{HJq}. In order to obtain a convergence rate for the homogenization, one is therefore left with the twofold task of quantifying both the limits given by the subadditive ergodic theorem as well as the precise way in which the subadditive quantities control the solutions of~\eqref{HJq}. The former must necessarily be handled by probabilistic methods and the latter by pde methods. 

These are not mere technical difficulties. It turns out that, in the stochastic setting, the magnitude of the fluctuations of the solutions~$u^\ep$ about their limit~$u$ must be separated into two distinct regimes, which we refer to as \emph{ballistic} and \emph{sub-ballistic}, respectively (this terminology is borrowed from the probability literature, see, for example, Sznitman~\cite{Szb}). Intuitively, in the \emph{ballistic} regime, the solutions are able to ``feel" the random environment sufficiently quickly as $\ep \to 0$. Then the mixing of the medium dominates, which results in an algebraic convergence rate. In the \emph{sub-ballistic} regime, the dependence of the solutions on $H$ is highly localized in the vicinity of points $y$ in which $H(p,y,\omega)$ is close to its essential supremum. It is therefore the law of $H(p,0,\cdot)$ near its essential supremum that principally governs the rate at which homogenization occurs, and this rate may be arbitrarily slow without a further assumption on the law. 

To give a more detailed overview of the approach, we start from the \emph{approximate cell problem}
\begin{equation}\label{acp}
\delta v^\delta + H(p+Dv^\delta,y,\omega) = 0 \quad \mbox{in} \ \Rd,
\end{equation}
which, for each fixed $p\in \Rd$ and $\delta > 0$, admits a unique bounded, uniformly continuous solution $v^\delta=v^\delta(y,\omega\,;p)$. The introduction of~\eqref{acp} in the context of the homogenization of Hamilton-Jacobi equations goes back to the original proof of periodic homogenization due to Lions, Papanicolaou and Varadhan~\cite{LPV}. It is well-known by now (see, e.g.,~\cite{ASo1}) that the homogenization of~\eqref{HJq} to~\eqref{HJh} is equivalent to (and the effective Hamiltonian $\overline H$ can be identified by) the limit
\begin{equation}\label{dvdlimit}
\lim_{\delta \to 0} - \delta v^\delta(0,\omega\,;p) = \overline H(p) \qquad \Prob\mbox{-a.s.}
\end{equation}
Moreover, this equivalence is easy to quantify in the sense that an error estimate or convergence rate for~\eqref{dvdlimit} can be transformed into one for homogenization. We are therefore left with the task of quantifying the limit in~\eqref{dvdlimit}.

The intuitive reason for the difficulty of arguing directly for the limit~\eqref{dvdlimit} in the random case is the complicated dependence of the $v^\delta$'s on $H$, which is both \emph{singular} (information propagates only along characteristics and does not spread out) and \emph{global} (information may travel far away in space, which is compounded by the lack of compactness). This problem is overcome in the stochastic setting by (i) imposing some kind of convexity assumption on $H$ and (ii) using the subadditive structure of the metric problem (or its time-dependent analogue) to obtain an almost sure limit via the subadditive ergodic theorem. A comparison argument (introduced in~\cite{ASo1,ASo3}) then yields that, in the ballistic regime ($p$'s satisfying $\overline H(p) > \min \overline H$), the metric problem controls the limiting behavior of the $\delta v^\delta(0,\omega\,;p)$ as $\delta \to 0$. In the sub-ballistic regime, i.e., for $p$'s belonging to the ``flat spot" $\{\overline H(\cdot) = \min \overline H\}$, the limiting behavior of the $\delta v^\delta$ is driven primarily by the law of $H(p,0,\cdot)$ near its essential supremum, as mentioned above, and it turns out to be the thickness of the tail of this distribution which governs the rate of homogenization.

We continue by introducing the \emph{metric problem}: for each fixed $\mu$ larger than a certain constant, which turns out to be $\min \overline H$, $x\in \Rd$ and $\omega\in \Omega$, there exists a unique nonnegative continuous solution $m_\mu = m_\mu(\cdot,x,\omega) : \Rd \to \R$ of
\begin{equation}\label{Imp}
H(Dm_\mu,y,\omega) = \mu \quad \mbox{in} \ \Rd \setminus \{ x \} \quad \mbox{and} \quad m_\mu(x,x,\omega) = 0.
\end{equation}
In terms of control theory, the quantity $m_\mu(y,x,\omega)$ corresponds to a ``cost" of transporting a particle from $x$ to $y$ in the medium $\omega$. It thus has the properties of a metric and is analogous to the time constant in first-passage percolation (see Remark~\ref{contform}).

The $m_\mu(\cdot,x,\omega)$'s are the maximal subsolutions of $H(Dw,y,\omega) \leq \mu$ in $\Rd$ subject to $w(x) = 0$, and this implies a subadditivity property. An easy application of the subadditive ergodic theorem (see~\cite{ASo3}) then yields the existence of $\overline m_\mu \in C(\Rd)$ such that
\begin{equation}\label{Implim}
\lim_{t\to\infty} t^{-1} m_\mu(ty,0,\omega) = \overline m_\mu(y) \qquad \Prob\mbox{-a.s.} 
\end{equation}
The deterministic $\overline m_\mu$ can be identified as the  unique nonnegative solution of
\begin{equation*}\label{}
\overline H(D\overline m_\mu) = \mu \quad \mbox{in} \ \Rd \setminus \{ 0 \} \quad \mbox{and} \quad \overline m_\mu(0) = 0,
\end{equation*}
and this provides another way of identifying $\overline H$, one level set at a time. (See Subsection~\ref{Hbarheur} for more.)

The existence and  some basic properties of the $m_\mu(\cdot,x,\omega)$'s have been known for some time (see, for example, Lions~\cite{Li}). More recently a simple comparison argument was introduced in~\cite{ASo3} which demonstrated that the $m_\mu(\cdot,x,\omega)$'s control the $\delta v^\delta(\cdot,\omega\,;p)$'s from below for every $p\in \Rd$, and from above for $p$'s in the ballistic regime. It follows from this analysis that the limit~\eqref{Implim} implies the homogenization of~\eqref{HJq}. As we show here, this argument is constructive in the sense that a quantitative rate for the convergence of~\eqref{Implim} implies a rate for~\eqref{dvdlimit}. 

The main advantage of the metric problem is that it is \emph{localized}. Indeed, while changes in the medium may influence the value of $\delta v^\delta$ at far away points, the quantity $m_\mu(y,x,\omega)$ depends only on the values of $H(p,z,\omega)$ for $z$'s satisfying $m_\mu(z,x,\omega) \leq m_\mu(y,x,\omega)$, which is a bounded set with diameter proportional to $|y-x|$. This localization (see Lemma~\ref{localize} and~\eqref{e.gotbunk} below) permits us to use the independence of the medium by way of the martingale method of bounded differences and an application of Azuma's concentration inequality in a similar manner as in first-passage percolation~\cite{K2}. 

In addition to quantifying the limit in~\eqref{Implim} and hence in~\eqref{dvdlimit} for $p$'s in the ballistic regime (and obtaining an almost sure, algebraic rate of convergence), we identify a necessary and sufficient condition (see~\eqref{plushyp} below) for an algebraic rate of convergence to hold for $p$'s in the sub-ballistic regime. This follows from a direct analysis of the $v^\delta$'s using explicit comparison arguments. Merging the results for the ballistic and sub-ballistic regimes then yields, under assumption~\eqref{plushyp}, an algebraic rate for~\eqref{dvdlimit} for all $p$'s.

We remark that we do not expect our arguments to yield sharp error estimates or convergence rates for homogenization. Indeed, as we explain in Remark~\ref{notsharp}, this is related to outstanding conjectures on the fluctuations of the time constant in first passage percolation. It is likely that the exponent in the rate of convergence convergence improves in higher dimensions, as is expected in first-passage percolation, although proving a rigorous statement to this effect seems out of reach. 

As mentioned above, the periodic homogenization of Hamilton-Jacobi equations was proved in~\cite{LPV}. This was simplified in~\cite{E} and subsequently extended to almost periodic media by Ishii~\cite{I}. The stochastic homogenization of convex first-order Hamilton-Jacobi equations was first proved in~\cite{S,RT} and, for viscous convex Hamilton-Jacobi equations, by Lions and Souganidis \cite{LS2} and Kosygina, Rezakhanlou, and Varadhan \cite{KRV}. Lions and Souganidis~\cite{LS1} obtained results on the existence and nonexistence of correctors in the random setting and introduced in~\cite{LS3} a more direct proof of homogenization in probability. Later, more direct proofs of almost sure homogenization, based on the metric problem, were given in~\cite{ASo1,ASo3}. 
 
The metric problem has been used by Davini and Siconolfi~\cite{DS1,DS2} to study some connections between the stochastic homogenization of Hamilton-Jacobi equations and weak KAM theory and, for periodic $H$'s with special structure, by Oberman, Takei and Vladimirsky~\cite{OTV} and Luo, Yu and Zhao~\cite{LYZ} in order to implement efficient numerical schemes for computing $\overline H$.

\subsection*{Outline of the paper}
In the next section we give the precise assumptions and the statement of the main results. In Section~\ref{Pre} we review some preliminary results needed in our arguments. Controlling the fluctuations of the metric problem is the topic of Section~\ref{Met1} and in Section~\ref{Met2} we control its statistical bias. These estimates are combined with comparison arguments in Section~\ref{EEdvd} to obtain corresponding bounds for the approximate cell problem in the ballistic regime. The sub-ballistic regime is studied in the second part of Section~\ref{EEdvd}, where we produce error estimates under an auxiliary hypothesis on the law of $H$ as well as examples demonstrating that, without such a hypothesis, the rate may be arbitrarily slow. We complete the proof of the error estimates in Section~\ref{EEue} and give convergence rate for the homogenization of the time-dependent problem~\eqref{HJq}. Finally, in Section~\ref{APs} we discuss the convergence rates of the homogenization of~\eqref{acp} and~\eqref{HJq} in almost periodic media. In the appendices we summarize the fundamentals of the metric and approximate cell problems.

\subsection*{Notation and conventions}

The symbols $C$ and $c$ denote positive constants which may vary from line to line and, unless otherwise indicated, depend only on the assumptions for $H$ and other appropriate parameters (often an upper bound for $|p|$ or $\mu$). For $s,t\in \R$, we write $s\wedge t : = \min\{ s,t\}$ and $s \vee t := \max\{s,t\}$.  We denote the $\d$-dimensional Euclidean space by $\Rd$,  $\Q^\d$ is the set of elements of $\Rd$ with rational coordinates, $\N$ is the set of natural numbers and $\N^* := \N \setminus \{ 0 \}$. For each $y \in \Rd$, $|y|$ denotes the Euclidean length of $y$. If $E \subseteq\Rd$, then $|E|$ is the Lebesgue measure of $E$, $\intr E$  the interior of $E$, $\overline E$  the closure of $E$ and $\conv E$ the closure of the convex hull of $E$. For $r>0$, we set $B(y,r): = \{ x\in \Rd : |x-y| < r\}$ and $B_r : = B(0,r)$. The distance between two subsets $U,V\subseteq \Rd$ is $\dist(U,V) = \inf\{ |x-y|: x\in U, \, y\in V\}$. If $f:E \to \R$ then we denote $\osc_E f : = \sup_E f - \inf_E f$. If~$K$ is a finite set, then $|K|$ is the number of elements of~$K$. The set of Lipschitz functions on a set~$U\subseteq\Rd$ is written~$\mathrm{Lip}(U) = C^{0,1}(U)$ and we set~$\Lip := \mathrm{Lip}(\Rd)$. The set of bounded and uniformly continuous real-valued functions on a metric space~$Y$ is denoted $\BUC(Y)$, and $\USC(Y)$ and $\LSC(Y)$ are respectively the sets of real-valued upper and lower semicontinuous functions on~$Y$. The Borel~$\sigma$-field on~$\Rd$ is~$\mathcal{B}$. If~$\mathcal G_1$ and $\mathcal G_2$ are~$\sigma$-fields on sets~$X_1$ and~$X_2$, respectively, then~$\mathcal G_1 \otimes \mathcal G_2$ denotes the $\sigma$-field on $X_1\times X_2$ generated by $\mathcal G_1 \times \mathcal G_2$. For a probability space $(\Omega,\mathcal F, \Prob)$, we say that an event $A\in \mathcal F$ is of \emph{full probability} if $\Prob[A]=1$. We denote the indicator random variable of $A\in \mathcal F$ by $\indc_A$. If $X$ is a random variable and $\mathcal G\subseteq \mathcal F$ is a $\sigma$-field, then $\E\left[ X  \vert  \mathcal G\right]$ denotes the conditional expectation of $X$ with respect to $\mathcal G$.

Throughout the paper, all differential inequalities are taken to hold in the viscosity sense. Readers not familiar with the fundamentals of the theory of viscosity solutions may consult standard references such as~\cite{CIL,Ba}.

\section{The assumptions and the statement of the main results} \label{MR}

We introduce our hypotheses and state the main results of the paper. 

\subsection{The hypotheses}
Let $(\Omega, \mathcal F, \mathds P)$ be a probability space endowed with a group $(\tau_y)_{y\in \Rd}$ of $\mathcal F$-measurable, measure-preserving transformations $\tau_y:\Omega\to \Omega$. That is, we assume that, for every $x,y\in\Rd$ and $A\in \mathcal F$,
\begin{equation}\label{pres}
\Prob[\tau_y(A)] = \Prob[A] \quad \mbox{and} \quad \tau_{x+y} = \tau_{x} \circ \tau_y.
\end{equation}
The Hamiltonian $H:\Rd\times\Rd\times \Omega\to \R$ is assumed to be measurable with respect to~$\mathcal B \otimes \mathcal B \otimes \mathcal F$. We write $H = H(p,y,\omega)$ and require that $H$ be \emph{stationary} in its dependence on $(y,\omega)$ with respect to the translation group $( \tau_y )_{y\in\Rd}$, that is, we assume that, for every 
$p,y,z\in\Rd$ and $\omega \in \Omega$,
\begin{equation} \label{stnary}
H(p,y,\tau_z \omega) = H(p,y+z,\omega).
\end{equation}
In order to state the finite range of dependence assumption, which means roughly that $H$ is ``i.i.d." in its spatial dependence, we define, for each $V\in \B$, the following $\sigma$--algebra on $\Omega$:
\begin{equation*} \label{}
\mathcal G(V):= \mbox{$\sigma$--algebra generated by the random variables $\omega \mapsto H(p,x,\omega)$, with $p\in \Rd$, $x\in V$.}
\end{equation*}
We may also suppose without loss of generality that $\mathcal F = \mathcal G(\Rd)$. The finite range dependence hypothesis is then the requirement that, for every $V,W\in \mathcal B$,
\begin{equation}\label{indy}
\dist(V,W) \geq 1 \quad \mbox{implies that} \quad \mathcal{G}(V) \ \mbox{and} \ \mathcal{G}(W) \ \mbox{are independent.}
\end{equation}
Of course, this implies that the group $\left( \tau_y \right)_{y\in \Rd}$ is ergodic, but is much stronger.

We continue with other structural hypotheses on $H$. We assume, for each $R> 0$, that the family
\begin{equation}\label{regpx}
\{ H(\cdot,\cdot,\omega) : \omega\in\Omega \} \quad \mbox{is  precompact in} \ C( B_R\times \Rd)
\end{equation}
and
\begin{equation}\label{reg}
\{ H(\cdot,x,\omega) : \omega\in\Omega, \ x\in \Rd \} \quad \mbox{is bounded in} \ C^{0,1}( B_R).
\end{equation}
We also require that $H$ is \emph{uniformly coercive} in $p$, that is,\begin{equation}\label{coer}
\lim_{|p|\to\infty} \essinf_{\omega \in \Omega} H(p,0,\omega) = +\infty.
\end{equation}
We assume that $H$ is slightly more than \emph{level-set convex} in $p$. Precisely, we assume that there exists $\Lambda:\R\times\R \to \R$, which is nondecreasing in each variable, such that, for all $\mu,\nu\in \R$,
\begin{equation}\label{Lambda}
\Lambda(\mu,\nu) \leq \mu \vee \nu \quad \mbox{and} \quad \Lambda(\mu,\nu) < \mu \vee \nu \quad \mbox{if} \ \nu\neq \mu,
\end{equation}
and that $H$ satisfies, for all $p,q,y \in\Rd$ and $\omega\in \Omega$,
\begin{equation}\label{sqc}
H\left(\tfrac12 (p + q),y,\omega\right) \leq \Lambda\big( H(p,y,\omega), H(q,y,\omega)\big).
\end{equation}
Of course, $H$ is convex if and only if \eqref{sqc} holds with $\Lambda (\mu,\nu) = \frac12( \mu + \nu)$.

We also make the following assumptions regarding the shape of the level sets of $H$: for every $p,y\in\Rd$ and $\omega\in \Omega$,
\begin{equation}\label{cntl}
H(p,y,\omega) \geq H(0,y,\omega) \qquad \mbox{and} \qquad  \esssup_{\omega\in\Omega} H(0,0,\omega) = 0. 
\end{equation}
From the point of view of optimal control theory, the fact that there is a common $p_0$ for all $\omega$ at which $H(\cdot,0,\omega)$ attains its minimum provides some ``controllability", i.e., upper and lower bounds on the length of optimal paths. We loose no generality by assuming $p_0=0$ and $\esssup_{\omega\in\Omega} H(0,0,\omega) = 0$. From our point of view, ~\eqref{cntl} controls the growth of the $m_\mu$'s (see~\eqref{control2} below). 

With the exception of Section~\ref{APs}, the hypotheses~\eqref{pres}--\eqref{cntl} described above are in force throughout the paper. For ease of reference, we write 
\begin{equation}\label{assum}
\mbox{{\eqref{pres}, \eqref{stnary}, \eqref{indy}, \eqref{regpx}, \eqref{reg}, \eqref{coer}, \eqref{Lambda}, \eqref{sqc} and \eqref{cntl} hold.}}
\end{equation}

Some of our results are proved under an extra assumption on the distribution of $H(0,0,\cdot)$ near its maximum. Precisely, this extra hypothesis is that there exist $\theta \geq 0$ and $c >0$ such that, for every $0 < \lambda \leq c$,
\begin{equation}\label{plushyp}
\Prob\left[ H(0,0,\cdot) > - \lambda \right] \geq c \lambda^\theta.
\end{equation}
In light of~\eqref{cntl}, we see that, roughly speaking,~\eqref{plushyp} is a requirement that the event that $H(0,0,\cdot)$ is near its maximum is not \emph{too unlikely}. For example, if $H(0,0,\cdot)$ attains its maximum on a set of positive probability, then of course~\eqref{plushyp} holds for $\theta = 0$.

Throughout the paper, the quantification ``for every $\omega\in \Omega$" is used exclusively for deterministic statements. For assertions which holds $\Prob$-almost surely (abbreviated as $\Prob$-a.s.~) we may write, for example, ``for every $\omega \in \Omega_1$" where $\Omega_1 \in \mathcal F$ is a specified event of full probability, i.e., $\Prob [\Omega_1] =1$. 

\subsection{The main results}

Our first main result consists of error estimates for the limit~\eqref{Implim}, which measure the likelihood that the quantity $|m_\mu(y,0,\omega) - \overline m_\mu(y)|$ is large relative to $|y|$. The definition and basic properties of the metric problem~\eqref{Imp} and its solutions $m_\mu$ and $\overline m_\mu$ are reviewed in the next section. The proof of Theorem~\ref{mpEE} is completed in Section~\ref{Met2}.

\begin{thm}[Error estimates for the metric problem] \label{mpEE}
Assume~\eqref{assum} and fix $K>0$. Then there exists $C > 0$, depending only on $K$ and~$H$, such that, for every $0< \mu \leq K$, $\lambda > 0$ and $|y| > 1$,
\begin{equation}\label{easy}
\Prob\Big[\, m_\mu(y,0,\cdot) - \overline m_\mu(y) \leq - \lambda \Big] \leq \exp\left(-\frac{\mu \lambda^2}{C |y|} \right)\;,
\end{equation}
and, if 
\begin{equation}\label{lamb-cond}
\lambda \geq C\left( \frac{|y|^{\frac12}}{\mu^{\frac32}} + \frac{|y|^\frac23}{\mu} \right) \left( \log\left(1+\frac{|y|}{\mu}\right)\right)^\frac12,
\end{equation}
then
\begin{equation}\label{hard}
\Prob\Big[\, m_\mu(y,0,\cdot) - \overline m_\mu(y) \geq \lambda \Big] \leq \exp\left(-\frac{\mu \lambda^2}{ C |y|} \right).
\end{equation}
\end{thm}

The error estimates for the metric problem and a careful quantification of the comparison arguments introduced in~\cite{ASo3}, together with an analysis of the convergence on the flat spot under the additional assumption~\eqref{plushyp}, yield the following error estimates for the limit~\eqref{dvdlimit}. The basic properties of the solutions $v^\delta$ of the approximate cell problem~\eqref{acp} are outlined in the next section.

\begin{thm}[Error estimates for the approximate cell problem] \label{acpEE}
Assume~\eqref{assum} and fix $K> 0$. There exists $C > 0$, depending only on $K$ and $H$, such that, for every $|p| \leq K$, we have:
\begin{enumerate}

\item[(i)] For every $0 < \delta \leq \lambda \leq 1$,
\begin{equation}\label{dvdEEabove}
\Prob\left[ -\delta v^\delta(0,\cdot\,;p) \geq \overline H(p) + \lambda \right] \leq C \delta^{-3\d} \exp\left( -\frac{\lambda^3}{C\delta} \right).
\end{equation}

\item[(ii)] If $\overline H(p)> 0$ and $0 < \delta \leq \lambda \leq 1$ satisfy	
\begin{equation}\label{lamb-condout2}
\lambda \geq C \left(  \overline H(p)^{-\frac32}\delta^{\frac12}  + \overline H(p)^{-1}\delta^{\frac13}\right) \left( 1 + | \log \delta| + |\log \overline H(p)| \right)^\frac12,
\end{equation}
then
\begin{equation}\label{dvdEEbelow}
\Prob\left[ -\delta v^\delta(0,\cdot\,;p) \leq \overline H(p) - \lambda \right] \leq C \delta^{-3\d} \exp\left( -\frac{ \overline H(p)\lambda^2}{C\delta} \right).
\end{equation}

\item[(iii)] Assume also~\eqref{plushyp}. There exists $c > 0$, depending on $K$ and $H$, such that if $0 < \delta \leq \lambda \leq c$ satisfy
\begin{equation} \label{lambang}
\lambda \geq C \delta^{\frac16} | \log\delta|^{\frac1{4}},
\end{equation}
then
\begin{equation}\label{EEFq}
\Prob\left[ -\delta v^\delta (0,\cdot\, ; p) \leq \overline H(p) - \lambda \right] \leq C \delta^{-3\d} \exp\left( - \frac{1}{C} \left(  \frac{\lambda^3}{\delta} \wedge \frac{\lambda^{\d+\theta}}{\delta^\d}\right) \right).
\end{equation}
\end{enumerate}
\end{thm}

By a covering argument and an application of the Borel-Cantelli lemma, the error estimates contained in Theorem~\ref{acpEE} yield $\Prob$-almost sure, local uniform rates of convergence for the limit~\eqref{dvdlimit}.

\begin{thm}[A convergence rate for the approximate cell problem]  \label{acpCR}
Assume~\eqref{assum} and fix $K> 0$. Then there exists an event $\Cl[O]{O-acpCR}\in \mathcal F$ of full probability and a constant $C>0$, depending on~$K$ and~$H$, such that, for every $|p| \leq K$ and $\omega\in \Cr{O-acpCR}$, the following hold:
\begin{enumerate}
\item[(i)] For every $R> 0$,
\begin{equation}\label{aboveRas}
\limsup_{\delta\to 0} \sup_{y\in B_{R/\delta}} \frac{-\delta v^\delta(y,\omega\,;p) - \overline H(p)}{C\delta^{\frac13}|\log \delta|^{\frac13} } \leq  1.
\end{equation}
\item[(ii)] If $\overline H(p) > 0$, then, for every $R> 0$,
\begin{equation}\label{belowRasF}
\liminf_{\delta\to 0} \inf_{y\in B_{R/\delta}} \frac{-\delta v^\delta(y,\omega\,;p) - \overline H(p)}{C \overline H(p)^{-1} \delta^{\frac13} |\log \delta|^{\frac12}} \geq  -1.
\end{equation}

\item[(iii)] If~\eqref{plushyp} holds and we set
\begin{equation}\label{alphbeta}
\alpha:= \frac16 \wedge \frac{\d}{\d+\theta}\qquad \mbox{and} \qquad \beta:= \frac14 \;,
\end{equation}
then, for every $R> 0$, 
\begin{equation}\label{belowRasF2}
\liminf_{\delta\to 0} \inf_{y\in B_{R/\delta}} \frac{-\delta v^\delta(y,\omega\,;p) - \overline H(p)}{C \delta^{\alpha} |\log \delta|^{\beta} } \geq  -1.
\end{equation}
\end{enumerate}
\end{thm}

The previous two results are proved in Section~\ref{EEdvd}, where we also give a converse to Theorem~\ref{acpCR}(iii), which states that the extra assumption~\eqref{plushyp} is actually necessary for an algebraic rate of convergence to hold at $p=0$. Indeed, keeping in mind that our assumptions imply that $\overline H(0) = 0$, we prove in Proposition~\ref{plushyp-roi} roughly that, if~\eqref{plushyp} is false, then for every exponent $\eta > 0$,
\begin{equation}\label{noalgrate}
\liminf_{\delta\to 0} \frac{-\delta v^\delta(0,\omega\,;0)}{\delta^{\eta} } = -\infty \quad \Prob-\mbox{a.s.}
\end{equation}
Furthermore, for any modulus function $\rho$, we construct examples of $H$'s satisfying \eqref{assum} for which
\begin{equation}\label{norate}
\liminf_{\delta\to 0} \frac{-\delta v^\delta(0,\omega\,;0)}{ \rho(\delta) } \leq -1   \quad \Prob-\mbox{a.s.}
\end{equation}
It is therefore necessary to impose, in addition to~\eqref{assum},  some assumption on the distribution of $H(0,0,\cdot)$ near its essential supremum in order to obtain a rate for the limit~\eqref{dvdlimit} at $p=0$. 

We next present our main quantitative results for the homogenization of~\eqref{HJq}. Here~$u^\ep$ and~$u$ denote, respectively, the unique solutions of~\eqref{HJq} and~\eqref{HJh} subject to the initial condition $u^\ep(\cdot,0) = u(\cdot,0) = u_0\in C^{0,1}(\Rd)$, which are bounded and Lipschitz continuous on $\Rd \times [0,T]$ for each $T> 0$. We begin with exponential estimates for the probability that $|u^\ep(x,t) - u(x,t)|$ is large.

\begin{thm}[Error estimates for homogenization] \label{EEH}

Assume~\eqref{assum} and fix $K>0$. Then there exists a constant $C > 0$, depending on $K$ and $H$ such that, for every $u_0 \in C^{0,1}(\Rd)$ satisfying $\| u_0 \|_{C^{0,1}(\Rd)} \leq K$ and~$T\geq 1$, the following hold:

\begin{enumerate}

\item[(i)] For every $0 < \ep \leq 1$ and $\lambda \geq C \ep^{\frac13}$, 
\begin{equation}\label{EEa}
 \Prob\left[ \inf_{x\in B_T} \inf_{0\leq t \leq T} \left( u^\ep(x,t,\cdot) - u(x,t) \right)  \leq - \lambda T \right] \leq C T^{6\d} \lambda^{9\d} \ep^{-6\d} \exp\left( -\frac{T\lambda^5 }{C\ep} \right). 
\end{equation}

\item[(ii)] If~\eqref{plushyp} holds, then, for every $0 < \ep \leq 1$ and 
\begin{equation}\label{alpbangp}
\lambda \geq C \ep^{\frac18} |\log \ep|^{\frac3{16}},
\end{equation}
\begin{multline}\label{EEb}
\quad \Prob\left[ \sup_{x\in B_T} \sup_{0\leq t \leq T} \left( u^\ep(x,t,\cdot) - u(x,t) \right)  \geq \lambda T \right]  \\ \leq C T^{6\d} \lambda^{9\d} \ep^{-6\d} \exp\left( - \frac1{C} \left(  \frac{T\lambda^5}{\ep} \wedge \frac{T^\d\lambda^{3\d+\theta}}{\ep^\d}\right) \right).
\end{multline}

\end{enumerate}
\end{thm}

Our final main result is an almost sure, locally uniform, algebraic rate of convergence for the homogenization of~\eqref{HJq}.

\begin{thm}[Convergence rate for homogenization] \label{CRH}
Assume~\eqref{assum} and fix $K> 0$. Then there exists an event $\Cl[O]{O-CRH}\in \mathcal F$ of full probability and a constant $C>0$, depending on $K$ and $H$, such that, for every $\omega\in \Cr{O-CRH}$ and $u_0\in C^{0,1}(\Rd)$ with $\| u_0 \|_{C^{0,1}(\Rd)} \leq K$, the following hold:

\begin{enumerate}

\item[(i)] For every $T \geq 1$,
\begin{equation}\label{CRHup}
\liminf_{\ep \to 0} \inf_{x\in B_T} \inf_{0 < t \leq T} \frac{u^\ep(x,t,\omega) - u(x,t)}{\ep^{\frac15}|\log\ep|^{\frac15}} \geq - CT.
\end{equation}

\item[(ii)] If~\eqref{plushyp} holds, $\alpha$ and $\beta$ are as in~\eqref{alphbeta} and we set
\begin{equation}\label{alphbeta2}
\bar a:= \frac{\alpha}{1+2\alpha} = \frac18 \wedge \frac{\d}{3\d+\theta} \qquad \mbox{and} \qquad \bar b:= \frac{\beta}{1+2\alpha} = \frac3{16} \vee \frac{\d+\theta}{4(3\d+\theta)},
\end{equation}
then, for every $T \geq 1$,
\begin{equation}\label{CRHdn}
\limsup_{\ep \to 0} \sup_{x\in B_T} \sup_{0<t\leq T} \frac{u^\ep(x,t,\omega) - u(x,t)}{\ep^{\bar a}|\log\ep|^{\bar b}} \leq CT.
\end{equation}
\end{enumerate}
\end{thm}

\begin{remark} We discuss later the sharpness of the exponent $\bar a$ and $\bar b$. Let us point out for the moment that, 
in the special case that $H$ is positively homogeneous of order one in $p$, i.e., for every $t\geq 0$, $p,y\in \Rd$ and $\omega\in \Omega$,
\begin{equation}\label{poshom}
H(tp,y,\omega) = tH(p,y,\omega),
\end{equation}
condition~\eqref{plushyp} is clearly satisfied for $\theta =0$, and thus Theorem~\ref{CRH} gives a rate of $O\big(\ep^{\frac18}| \log \ep|^{\frac3{16}}\big)$ for homogenization. Moreover, this rate can be improved since~\eqref{poshom} implies that $\overline H$ is also positively homogeneous of order one, or equivalently that $\mu \mapsto \overline m_\mu(y)$ is positively homogeneous of order one, that is, for all $\mu > 0$, $x,y\in\Rd$ and $\omega\in \Omega$,
\begin{equation*}\label{}
m_\mu(y,x,\omega) = \mu m_1(y,x,\omega).
\end{equation*}
Thus the fluctuations of $m_\mu-\overline m_\mu$ are proportional to $\mu$, and this prevents~\eqref{hard} from degenerating as $\mu \to 0$. Indeed, we find that~\eqref{hard} holds for every $\lambda > 0$ satisfying
\begin{equation*}\label{}
\lambda \geq C_2 \mu |y|^{\frac23} \left( \log(1+|y|) \right)^{\frac12}
\end{equation*}
instead of the more restrictive~\eqref{lamb-cond}. This improvement may be propagated through the rest of the paper to find that~\eqref{belowRasF2} holds for $\alpha=\frac13$ and $\beta = \frac12$, and~\eqref{CRHdn} for $\bar a = \frac15$ and $\bar b=\frac3{10}$. A similar observation holds for Hamiltonians which are positively homogeneous of any positive order and we expect that other such improvements are possible for $H$'s with special structure.
\end{remark}

\subsection{Explicit examples}
We illustrate the assumptions with two simple but typical classes of Hamilton-Jacobi equations:
$$
H_1(p,y,\omega)= \frac12 |p|^2 -V(y,\omega) \qquad \mbox{and} \qquad H_2(p,y,\omega)= a(y,\omega)|p|.
$$
The former arises in problems in the calculus of variations and geometric optics, for example, and the latter in front propagation. To ensure that~\eqref{assum} is satisfied, we require $a,V:\R^d\times \Omega\to \R$ to be measurable, stationary with respect to the action of the translation group, satisfy a finite range of dependence hypothesis, be uniformly continuous and bounded in the first variable (uniformly in the second variable) and nonnegative. We also require  
\begin{equation}\label{CondSurV1}
\essinf_{\omega \in \Omega} V(0,\omega)=0
\end{equation}
and that $a(\cdot,\omega)$ is Lipschitz uniformly in $\omega$ and bounded below by a positive constant. Observe that the more restrictive condition \eqref{plushyp} is satisfied by  $H_2$ and for $H_1$ is equivalent to the existence of constants $\theta>0$ and $c>0$ such that, for every $0 < \lambda \leq c$,  
\begin{equation}\label{CondSurV2}
\Prob\left[ V(0,\cdot) < \lambda \right] \geq c \lambda^\theta.
\end{equation}
It is relatively easy to construct random potentials which do \emph{not} satisfy~\eqref{CondSurV2}: see Subsection~\ref{doucement}.

We remark that the following Hamiltonian is \emph{not} covered by our assumptions:
\begin{equation*} \label{}
H_1'(p,y,\omega)= \frac12 |p|^2 - p\cdot b(y,\omega).
\end{equation*}
Here $b$ is a random vector field satisfying appropriate conditions, and the assumption not satisfied is~\eqref{cntl}. We believe it would be very interesting to develop an error analysis for stochastic homogenization for Hamiltonians like $H_1'$ not satisfying~\eqref{cntl}. The difficulty from the point of view of our approach is that we lose control on the rate of growth of the sublevel sets of $m_\mu$.

There are many ways of constructing of random functions like $a$, $V$ and $b$, above. For example, one may consider a Poissonian point cloud, attach a deterministic bump function to every point and sum. Such a random function satisfies a finite range of dependence if the bump function has compact support, and is called a \emph{Poissonian potential}. There are other possibilities such as ``random checkerboards" and so on, but we do not discuss these here.

\section{Preliminaries} \label{Pre}

In this section we recall the basic properties of the metric problem and the approximate cell problem and their connections to the effective Hamiltonian. We conclude by giving the statement of some results needed in the sequel.

\subsection{The metric problem: basic properties}

We summarize some elementary facts concerning the functions $m_\mu$, which play a central role in the rest of the paper. They are defined for each $x,y\in \Rd$, $\mu \geq 0$, and $\omega\in\Omega$ by 
\begin{equation}\label{defmmu}
m_\mu(z,x,\omega) : = \sup \big\{ w(z) -w(x) \, : \, w\in \Lip \ \mbox{and} \ H(Dw,y,\omega) \leq \mu \ \mbox{in} \ \Rd \big\}.
\end{equation}
Note that, due to~\eqref{cntl}, the zero function belongs to the admissible class, which is therefore nonempty. Lemma~\ref{convtrick} and~\eqref{coer} yield that $m_\mu(y,x,\omega)$ is finite and, in fact, nonnegative and bounded from above by $C|y-x|$, for some $C> 0$ depending on an upper bound for $\mu$.

It is immediate from~\eqref{defmmu} that $m_\mu$ is measurable with respect to $\mathcal B \otimes\mathcal B\otimes \mathcal F$, since the expression on the right of \eqref{defmmu} is. Moreover, from~\eqref{defmmu} and~\eqref{stnary} we see that $m_\mu$ is jointly stationary in its first two variables, i.e., for every $x,y,z\in\Rd$ and $\omega\in\Omega$,
\begin{equation}\label{mmustat}
m_\mu(y,x,\tau_z\omega) = m_\mu(y+z,x+z,\omega).
\end{equation}

Also immediate from~\eqref{defmmu} (and the fact that a supremum of a family of viscosity subsolutions is a viscosity subsolution, see~\cite{CIL,Ba}) that the $m_\mu(\cdot,x,\omega)$'s are global subsolutions of~\eqref{eikp}, i.e., for every $\mu \geq 0$, $x\in \Rd$ and $\omega\in \Omega$, 
\begin{equation}\label{globsub}
H(Dm_\mu(\cdot,x,\omega),y,\omega) \leq \mu \quad \mbox{in} \ \Rd.
\end{equation}
Further properties of the $m_\mu$'s are recorded in the next proposition. Detailed proofs of most of these facts can be found in~\cite{ASo3}. In Appendix~\ref{appmappp} we present sketches of the arguments.

\begin{prop}\label{existMP}
For every $\mu \geq 0$ and $\omega\in \Omega$, the following hold:
\begin{enumerate}
\item[(i)] For each fixed $x\in \R$, the function $m_\mu(\cdot,x,\omega)$ is a solution of 
\begin{equation}\label{mpagan}
H(Dm_\mu(\cdot,x,\omega),y,\omega) = \mu \quad \mbox{in} \ \Rd\setminus \{ x \} \quad \mbox{and} \quad m_\mu(x,x,\omega) = 0. 
\end{equation}
Moreover, if $\mu > 0$, then $m_\mu$ is the unique nonnegative solution of~\eqref{mpagan}. 
\item[(ii)] If $U \subseteq \Rd$ is open, $x\in \Rd\setminus U$ and $u\in \Lip$ is a subsolution of~\eqref{eikp} in $U$, then 
\begin{equation}\label{maxm}
u - m_\mu(\cdot,x,\omega) \leq \max_{y\in \partial U} \left( u(y) -  m_\mu(y,x,\omega) \right) \quad \mbox{in} \ U. 
\end{equation}

\item[(iii)] For $x,y,z\in \Rd$,
\begin{equation}\label{subadd}
m_\mu(y,x,\omega) \leq m_\mu(y,z,\omega) + m_\mu(z,x,\omega).
\end{equation}

\item[(iv)] There exist $l_\mu, L_\mu \geq 0$ satisfying, for some $C,c>0$ depending only on an upper bound for~$\mu$, 
\begin{equation}\label{lmuLmu}
 c\mu \leq l_\mu \leq L_\mu \leq C,
\end{equation}
such that
\begin{equation}\label{control2}
l_\mu |y-x| \leq m_\mu(y,x,\omega) \leq L_\mu|y-x|.
\end{equation}

\item[(v)] For every $x,y\in \Rd$,
\begin{equation}\label{lips}
|m_\mu(y,x,\omega)  - m_\mu(z,x,\omega)| \leq L_\mu|y-z|.
\end{equation}

\item[(vi)] For every open set $U \subseteq \Rd$, $x\in U$ and $y\in \Rd \setminus U$,
\begin{equation}\label{dynprog2}
m_\mu(y,x,\omega) = \min_{z\in \partial U} \big( m_\mu(y,z,\omega) + m_\mu(z,x,\omega) \big).
\end{equation}

\item[(vii)] For every $x\in \Rd$, 
\begin{equation}\label{flipx2}
H(-Dm_\mu(x,\cdot,\omega) , \cdot, \omega) = \mu \quad \mbox{in} \ \Rd \setminus \{ x \}
\end{equation}
and
\begin{equation}\label{flipx}
H(-Dm_\mu(x,\cdot,\omega) , \cdot, \omega) \leq \mu \quad \mbox{in} \ \Rd.
\end{equation}

\item[(viii)] There exists a constant $c>0$, depending on an upper bound for $\mu$, such that, for every $0\leq \widehat \mu \leq \mu$ and $x,y\in \Rd$,
\begin{equation}\label{strinc}
m_{\widehat \mu} (y,x,\omega) + c( \mu-\widehat \mu) |x-y| \leq m_\mu(y,x,\omega).
\end{equation}
\end{enumerate}
\end{prop}

\begin{remark} \label{contform}
The functions $m_\mu$ can be expressed by the following representation formula due to Lions~\cite{Li}, which provides the above facts with a control theoretic interpretation:
\begin{equation}\label{controlform}
m_\mu(y,x,\omega)= \inf\left\{ \int_0^1 J_\mu(\gamma'(s),\gamma(s),\omega) \, ds \, : \, \gamma \in \mathcal C(x,y) \right\},
\end{equation}
where $\mathcal C(x,y)$ is the set of Lipschitz curves $\gamma:[0,1] \to \Rd$ such that $\gamma(0) = x$ and $\gamma(1)=y$ and $J_\mu$ is the support function of the $\mu$-sublevel set of $H$, given by
\begin{equation*}\label{}
J_\mu(q,y,\omega):= \sup\left\{ p\cdot q\,:\, H(p,y,\omega) \leq \mu \right\}.
\end{equation*}
The expression~\eqref{controlform} provides us with an interpretation of $m_\mu(y,x,\omega)$ as measuring the ``cost" of moving from the point $x$ to the point $y$ in the medium $\omega$. We make no direct use of~\eqref{controlform} in this paper, preferring instead to work with the maximality property (Proposition~\ref{existMP}(ii)) which is equivalent to it. Nevertheless, our intuition is enriched from~\eqref{controlform} and it suggests an analogy between the metric problem and first-passage percolation.
\end{remark}

\begin{definition}
 For each $\mu,t>0$ and $x\in \Rd$, we define the \emph{reachable set to $x$ in time $t$} by
\begin{equation}\label{reach}
\Re_{\mu,t}(x) : = \left\{ (y,\omega) \in \Rd \times \Omega: m_\mu(y,x,\omega) \leq t \right\}.
\end{equation}
We also use the notation $\Re_{\mu,t} = \Re_{\mu,t}(0)$,  $\Re^\omega_{\mu,t}(x):= \{ y\in \Rd : (y,\omega) \in \Re_{\mu,t}(x) \}$ and $\Re^\omega_{\mu,t} = \Re^\omega_{\mu,t}(0)$. 
\end{definition}

We continue by examining some elementary properties of the reachable set. It is useful to note that, in the particular case that $x=0$ and $U:= \{ y \in\Rd \,:\, m_\mu(y,0,\omega) < t\}$, Proposition~\ref{existMP}(vi) asserts that, for every $t> 0$ and $y\in \Rd$ such that $m_\mu(y,0,\omega) \geq t$,
\begin{equation}\label{DPRT}
m_\mu(y,0,\omega) = t + \min_{z\in \Re_{\mu,t}^\omega} m_\mu(y,z,\omega).
\end{equation}
In view of~\eqref{control2} and~\eqref{DPRT}, we see that $m_\mu(y,x,\omega) < t$ for every $y$ in the interior of $\Re_{\mu,t}^\omega(x)$ and  $\partial \Re^\omega_{\mu,t} = \{ y\in\Rd: m_\mu(y,0,\omega)  = t\}$. 
In fact, \eqref{control2} and \eqref{DPRT}, give the following estimates for the growth rate of the reachable set: for every $0 < s < t$ and $\omega\in \Omega$,
\begin{equation}\label{capture2}
\left\{ x\in \Rd \, : \, \dist\left(x,\Re_{\mu,s}^\omega\right) \leq L_\mu^{-1} (t-s) \right\}  \subseteq \Re_{\mu,t}^\omega \subseteq \left\{ x\in \Rd \, : \, \dist\left(x,\Re_{\mu,s}^\omega\right) \leq l_\mu^{-1} (t-s) \right\}.
\end{equation}
We may think of $t\mapsto \Re_{\mu,t}^\omega$ as a ``growing front," and in this interpretation~\eqref{capture2} provides uniform positive lower and upper bounds on the speed of the front. In particular, for all $\mu,t>0$ and $\omega\in \Omega$,
\begin{equation}\label{capture}
B_{t/ L_\mu} \subseteq \Re_{\mu,t}^\omega \subseteq B_{t/l_\mu}.
\end{equation}

The maximality property (Proposition~\ref{existMP}(ii)) can be improved for the domain $U=\Re_{\mu,t}^\omega(x) \setminus \{ x \}$ by restricting the maximum over $\partial U$ to $\{ x \}$, as stated in the following lemma. This is the crucial fact that localizes the metric problem.

\begin{lem} \label{localize}
For every $\omega\in \Omega$, $\mu \geq 0$, $x\in\Rd$ and $w\in \mathrm{Lip}(\Re_{\mu,t}^\omega(x))$,
\begin{equation}\label{mprset}
H(Dw,y,\omega) \leq \mu \quad \mbox{in} \ \Re_{\mu,t}^\omega(x) \qquad \mbox{implies} \qquad w(\cdot) - w(x) \leq m_\mu(\cdot,x,\omega) \quad \mbox{in} \ \Re^\omega_{\mu,t}(x).
\end{equation}
\end{lem}
\begin{proof}
Let $w \in \mathrm{Lip}(\Re^\omega_{\mu,t}(\omega))$ satisfy the antecedent of \eqref{mprset} and assume with no loss of generality that $w(x) = 0$. Define $\widetilde w:= m_\mu(\cdot,x,\omega) \vee (  w(\cdot) \wedge (t-\ep))$ and, noticing that $\widetilde w = m_\mu(\cdot,x,\omega)$ near the boundary of $\Re^\omega_{\mu,t}(\omega)$, extend $\widetilde w$ to be defined on $\Rd$ by taking $\widetilde w = m_\mu(\cdot,x,\omega)$ in the complement of $\Re^\omega_{\mu,t}(\omega)$. Observe that, in light of Lemma~\ref{convtrick} and \eqref{cntl}, $\widetilde w$ is a subsolution of $H(D\widetilde w,y,\omega)\leq\mu$ in $\Rd$. We deduce from the maximality property that $\widetilde w \leq m_\mu(\cdot,x,\omega)$ in $\Rd$ and in particular $ w \wedge (t-\ep) \leq m_\mu(\cdot,x,\omega)$ in $\Re_{\mu,t}^\omega(x)$. Sending $\ep \to 0$ and using that $m_\mu(y,x,\omega) < t$ in the interior of $\Re_{\mu,t}^\omega(x)$, we obtain that $w \leq m_\mu(\cdot,x,\omega)$ in $\Re_{\mu,t}^\omega(x)$, as claimed. 
\end{proof}

It follows from~Lemma~\ref{localize} that the representation formula~\eqref{defmmu} may be restricted to the reachable set. Precisely, for every $\omega\in \Omega$, $\mu \geq 0$, $x\in\Rd$ and $y \in \Re_{\mu,t}^\omega(x)$,
\begin{equation}\label{measlocal}
m_\mu(y,x,\omega) = \sup \{ w(y) - w(x) \, : \, w\in \mathrm{Lip}(\Re_{\mu,t}^\omega(x)) \  \  \mbox{and}  \  \ H(Dw,y,\omega) \leq \mu \ \mbox{in} \ \Re_{\mu,t}^\omega(x) \big\}.
\end{equation}
This is immediate from \eqref{mprset}. 

In the proof of Lemma~\ref{indyass}, we require a refinement of~Lemma~\ref{localize}. To this state this, we define, for every nonempty closed set $K\subseteq \Rd$, $x,y\in K$ and $\omega\in \Omega$,
\begin{equation}\label{e.defmKmu}
m^K_\mu(y,x,\omega):= \sup \{ w(y) - w(x) \, : \, w\in \mathrm{Lip}(K) \  \  \mbox{and}  \  \ H(Dw,y,\omega) \leq \mu \ \mbox{in} \ K \big\}.
\end{equation}
It is immediate that 
\begin{equation} \label{e.messsing}
\mbox{$m^K_\mu$ is $\mathcal G(K)$--measurable}
\end{equation}
and
\begin{equation}\label{e.gotchup}
K_1 \subseteq K_2 \qquad \mbox{implies that} \qquad m^{K_1}_\mu(y,x,\omega) \geq m^{K_2}_\mu(y,x,\omega) \ \ \mbox{for all}  \ \ x,y\in K_1.
\end{equation}
Thus~\eqref{measlocal} yields that
\begin{equation}\label{e.gotcha}
R^\omega_{\mu,t} \subseteq K \subseteq \Rd \qquad \mbox{implies that} \qquad   m_\mu(\cdot,0,\omega) = m_\mu^K(\cdot,0,\omega) \quad \mbox{in} \ \Re^\omega_{\mu,t}.
\end{equation}

\begin{lem}
Assume that $K\subseteq \Rd$ is nonempty, closed and $0\in K$. Then
\begin{equation}\label{e.gotbunk}
\inf_{y\in \partial K} m^K_\mu(y,0,\omega) \geq t  \quad \mbox{implies that} \quad m^K_\mu(\cdot,0,\omega) \equiv m_\mu(\cdot,0,\omega) \quad \mbox{in} \  \left\{ m^K_\mu(\cdot,0,\omega) \leq t\right\}.
\end{equation}
\end{lem}
\begin{proof}
The argument is similar to the proof of Lemma~\ref{localize}. Consider the function $m^K_\mu(\cdot,0,\omega) \wedge (t-\ep)$ and extend this to $\Rd$ by giving it the value $(t-\ep)$ outside of $K$. This is a global subsolution by Lemma~\ref{convtrick} and the assumption that $m^K_\mu(\cdot,0,\omega) \geq t$ on $\partial K$. By the maximality of $m_\mu$ we deduce that $m_\mu \geq m^K_\mu(\cdot,0,\omega) \wedge (t-\ep)$. Sending $\ep \to 0$ yields $m_\mu(\cdot,0,\omega) \geq m_\mu^K(\cdot,0,\omega)$ in $\left\{ m^K_\mu(\cdot,0,\omega) \leq t\right\}$. In light of~\eqref{e.gotchup}, this completes the proof of~\eqref{e.gotbunk}. 
\end{proof}

 We next define, for every compact set $K$ of $\Rd$, $y \in \Rd$ and $\omega\in \Omega$,
\begin{equation}\label{mmuK}
m_\mu(y,K,\omega) := \inf_{z\in K} m_\mu(y,z,\omega) \qquad \mbox{and} \qquad m_\mu(K,y,\omega) := \inf_{z\in K} m_\mu(z,y,\omega).
\end{equation}

The next proposition provides representation formulas for these functions, which are needed to deduce~\eqref{infKmeas} below. The proof (sketch) is given in Appendix~\ref{appmappp}.

\begin{prop}\label{e.weakner}
For every $\mu>0$, the functions in~\eqref{mmuK} are the unique nonnegative solutions of
\begin{equation}\label{mmuKeqs}
H(Dm_\mu(\cdot,K,\omega),\cdot,\omega) = \mu \quad \mbox{and} \quad H(-Dm_\mu(K,\cdot,\omega),\cdot,\omega) = \mu \quad \mbox{in} \ \Rd \setminus K
\end{equation}
which satisfy the boundary condition $m_\mu(\cdot,K,\omega) = m_\mu(K,\cdot,\omega) = 0$ on $\partial K$.
Moreover, they are also given by the representation formulas
\begin{equation}\label{Krep1}
m_\mu(y,K,\omega) = \min_{x\in K} \, \sup\left\{ w(z) - w(x) \, :  \, w\in \Lip, \ H(Dw,z,\omega) \leq \mu \ \mbox{in} \ \Rd \setminus K \right\},
\end{equation}
and
\begin{equation}\label{Krep2}
m_\mu(K,z,\omega) = \min_{x\in K} \, \sup\left\{ w(x) - w(z) \, :  \, w\in \Lip, \ H(Dw,y,\omega) \leq \mu \ \mbox{in} \ \Rd \setminus K \right\}.
\end{equation}
\end{prop}

It is immediate from~\eqref{Krep1} and~\eqref{Krep2} that, for every compact $K\in \Rd$ and $y\in \Rd$,
\begin{equation}\label{infKmeas}
\omega \mapsto m_\mu(y,K,\omega) \quad \mbox{and} \quad \omega \mapsto m_\mu(K,y,\omega) \quad \mbox{are \ $\mathcal {G}\left(\Rd \!\setminus \!K\right)$-measurable.}
\end{equation}

\subsection{The approximate cell problem: basic properties}
\label{ssacpa}

We summarize the properties of the \emph{approximate cell problem}
\begin{equation}\label{amp}
\delta v^\delta + H(p+Dv^\delta,y,\omega) = 0 \quad \mbox{in} \ \Rd. 
\end{equation}
Here $p\in \Rd$ and $\delta > 0$ are given parameters and $v^\delta= v^\delta(y,\omega\,;p)$. We note that the assertions in this section do not depend in any way on the assumptions~\eqref{sqc} or of~\eqref{cntl}, nor do they depend on the random parameter~$\omega$, and therefore they hold (with appropriate changes in the notation) for the Hamiltonians encountered in Section~\ref{APs}. 

We begin with a comparison principle for~\eqref{amp}, which can be reduced to Proposition~\ref{comp}, below, by an argument which perturbs the solutions by adding appropriate terms with linear growth (alternatively, a proof can be found in~\cite{CIL,Ba}).

\begin{prop}\label{appcmp}
Let $\delta > 0$, $p\in \Rd$ and $\omega\in \Omega$. Suppose  $u,-v\in \USC(\Rd)$ satisfy
\begin{equation}\label{}
\delta u + H(p+Du,y,\omega) \leq 0 \leq \delta v + H(p+Dv,y,\omega) \quad \mbox{in} \ \Rd,
\end{equation}
and $v$ is bounded below on $\Rd$. Then $u\leq v$ in $\Rd$.
\end{prop}

For each $\delta > 0$, $p\in \Rd$ and $\omega\in \Omega$, we define
\begin{equation}\label{dvdform}
v^\delta(y,\omega\,;p) := \sup\left\{ w(y) \, : \, w\in \USC(\Rd) \ \ \mbox{satisfies} \ \ \delta w + H(p+Dw,y,\omega) \leq 0 \ \ \mbox{in} \ \Rd \right\},
\end{equation}
with the differential inequality in the definition interpreted either in the viscosity or in the almost everywhere sense, as these are equivalent in our situation by Lemma~\ref{convtrick}. It is clear that the constant function
\begin{equation*}\label{}
w: = -\frac1\delta \esssup_{y\in \Rd} H(p,y,\omega) 
\end{equation*}
belongs to the admissible class, hence $v^\delta(\cdot,\omega\,;p) \geq w$. Similarly, the function
\begin{equation*}\label{}
v:= -\frac1\delta\essinf_{y\in \Rd} H(p,y,\omega)
\end{equation*}
is a bounded supersolution of~\eqref{amp}, and Proposition~\ref{appcmp} yields $v^\delta(\cdot,\omega\,;p) \leq v$. We have shown that, for all $y,p\in \Rd$, $\omega\in \Omega$ and $\delta > 0$,
\begin{equation}\label{dvdsup2}
-\esssup_{ y\in \Rd} H(p,y,\omega) \leq \delta v^\delta(y,\omega\,;p) \leq -\essinf_{y\in \Rd} H(p,y,\omega).
\end{equation}


Immediate from~\eqref{dvdform} and~\eqref{stnary} is that the $v^\delta$'s are stationary functions. That is, for all $y,z,p\in\Rd$, $\omega\in\Omega$ and $\delta > 0$,
\begin{equation}\label{dvdstat}
v^\delta(y,\tau_z\omega\,;p) = v^\delta(y+z,\omega\,;p)
\end{equation}

We summarize some further properties of the $v^\delta$'s in the following proposition. The proofs, which are standard in the theory of viscosity solutions, are sketched in Appendix~\ref{appmappp}. For the statements we need to define the constants 
\begin{equation}\label{Kp}
K_p:= \sup\left\{ |p-q| \, : \, \omega\in \Omega, \ \ \inf_{y\in\Rd} H(q,y,\omega) \leq \sup_{y\in\Rd} H(p,y,\omega) \right\}
\end{equation}
and  
\begin{equation} \label{pip}
\Pi_p:=\max_{\sigma = \pm 1}\sup_{\omega\in \Omega} \sup_{|q| \leq K_p} \sup_{y\in\Rd} \big| H(p+q,y,\omega) - H(p+(1+ \sigma )q,y,\omega)\big|.
\end{equation}
Observe that $K_p$ is bounded above for~$|p|$ bounded by~\eqref{coer}. It follows from this and~\eqref{reg} that $\Pi_p$ is also bounded above for bounded~$|p|$.

\begin{prop} \label{vdeltas}
For every $p\in \Rd$, $\omega\in \Omega$ and $\delta > 0$,
the following hold:
\begin{enumerate}

\item[(i)] The function $v^\delta(\cdot,\omega\,;p)$ is the unique solution of~\eqref{amp} belonging to $\BUC(\Rd)$. 

\item[(ii)] For every $x,y\in\Rd$,
\begin{equation}\label{vdlip2}
\big| v^\delta(x,\omega\,;p) - v^\delta(y,\omega\,;p)\big|  \leq K_p|x-y|.
\end{equation}

\item[(iii)] For all $q,y\in \Rd$,
\begin{equation}\label{vdpdepp}
\big| \delta v^\delta(y,\omega\,;p) - \delta v^\delta(y,\omega\,;q)\big| \leq \sup_{|z| \leq K_p \vee K_q} \sup_{x\in\Rd}  \big| H(p+z,x,\omega) - H(q+z,x,\omega) \big|.
\end{equation}

\item[(iv)] For every $\eta \geq \delta$ and $y\in\Rd$,
\begin{equation}\label{dvddepd}
\big| \delta v^\delta (y,\omega\,;p) - \eta v^\eta(y,\omega\,;p) \big| \leq \Pi_p \!\left( 1 - \frac\delta\eta \right).
\end{equation}
\end{enumerate}
\end{prop}

\subsection{Identifying~$\overline H$}
\label{Hbarheur}

In this subsection we review how the effective Hamiltonian~$\overline H$ may be identified either via limits of either the maximal subsolutions $m_\mu$ or alternatively of the solutions $v^\delta$ of the approximate cell problem.

First we do a consistency check to derive a formula for $\overline H$ in terms of the metric problem. To guess what the formulas should be, it is helpful to use the \emph{theatrical scaling}: for each $\ep > 0$, we rescale by defining
\begin{equation*} \label{}
m_\mu^\ep(x,\omega) := \ep m_\mu\left( \frac x\ep,0,\omega\right)
\end{equation*}
and observe that $m_\mu^\ep(\cdot,\omega)$ is a solution of the metric problem for $H^\ep(p,x,\omega):=H(p,\tfrac x\ep,\omega)$, that is, 
\begin{equation*} \label{}
H\left(Dm_\mu^\ep,\frac x\ep,\omega\right) = \mu \quad \mbox{in} \ \Rd\setminus \{ 0\}
\end{equation*}
with $0=m_\mu^\ep(0,\omega) \leq m_\mu^\ep(x,\omega)$ in $\Rd$. If the statement of qualitative homogenization holds for this problem (here we are \emph{not} being rigorous and in fact using a circular argument!), then we have
\begin{equation} \label{limheur1}
m_\mu^\ep (x,\omega) \longrightarrow \overline m_\mu(x) \qquad \mbox{as \ $\ep \to 0$, \ locally uniformly in $x\in \Rd$,}\quad \mbox{$\Prob$-a.s.}
\end{equation}
where $\overline m_\mu$ should be the solution of the metric problem for the effective Hamiltonian $\overline H$, that is,
\begin{equation} \label{limheureq}
\overline H(D\overline m_\mu) = \mu \quad \mbox{in} \ \Rd\setminus \{ 0\}.
\end{equation}
Now let us reverse the change of variables to write the limit~\eqref{limheur1} in the original scaling (we also write in terms of $t=1/\ep$):
\begin{equation} \label{limheur2}
\limsup_{t\to \infty} \sup_{x\in B_{R}} \left| \frac{m_\mu(tx,0,\omega)}{t} - \overline m_\mu(x)\right| = 0 \qquad \mbox{for every} \  R>0, \quad \mbox{$\Prob$-a.s.},
\end{equation}
It is immediate from the form of this limit that $\overline m_\mu$ must be positively homogeneous. The subadditive property of the $m_\mu$'s easily translates into a subadditivity property for $\overline m_\mu$ and, therefore, $\overline m_\mu$ is convex. It follows that $\overline m_\mu$ may be written as a maximum of planes $x\mapsto p\cdot x$ over $p$ belonging to some closed convex set. We deduce that, if the equation~\eqref{limheureq} is to hold, \emph{this convex set must be the $\mu$-sublevel set of $\overline H$}:
\begin{equation*} \label{}
\overline m_\mu(x) = \sup\left\{ p\cdot x \, : \, \overline H(p) \leq \mu \right\}.
\end{equation*}
We may invert this formula to write $\overline H$ in terms of $\overline m_\mu$ as
\begin{equation}\label{Hbarformm}
\overline H(p) = \inf\left\{ \mu \geq 0 \, : \, \overline m_\mu(y) \geq p \cdot y \ \ \mbox{for all} \ \ y\in \Rd\right\}.
\end{equation}

We next see how $\overline H$ may be identified as a limit of the solutions of the approximate cell problem, by a similar heuristic. As usual, it is helpful to use the theatrical scaling. With $v^\delta(\cdot,\omega\,;p)$ defined by~\eqref{dvdform}, we set
\begin{equation*} \label{}
v_\ep(x,\omega):= \ep v^\ep\!\left( \frac x\ep,\omega\,;p\right)
\end{equation*}
and check that $v_\ep(\cdot,\omega)$ satisfies the equation
\begin{equation*} \label{}
v_\ep + H\left( p+Dv_\ep,\frac x\ep,\omega \right) = 0 \quad \mbox{in} \ \Rd.
\end{equation*}
We expect that
\begin{equation*} \label{}
v_\ep(x,\omega) \longrightarrow \overline v(x) \qquad \mbox{as \ $\ep \to 0$, \ locally uniformly in $x\in \Rd$,}\quad \mbox{$\Prob$-a.s.},
\end{equation*}
where $\overline v$ should be the solution of the problem
\begin{equation*} \label{}
\overline v+\overline H(p+D\overline v) = 0 \quad \mbox{in} \ \Rd.
\end{equation*}
But notice that we have a formula for the latter: $\overline v$ is a constant function, namely, $\overline  v\equiv -\overline H(p)$. So rewriting the limit in terms of the original scaling, we expect that, for every $p\in \Rd$,
\begin{equation} \label{dvdheureq}
\limsup_{\delta \to 0} \sup_{y\in B_{R/\delta}} \left| \delta v^\delta(y,\omega\,;p) + \overline H(p) \right| = 0 \qquad \mbox{for every} \  R>0, \quad \mbox{$\Prob$-a.s.} 
\end{equation}

The strategy of the proof of qualitative homogenization from~\cite{ASo3} consists of reversing the heuristic argument above. Here is an outline of the method, each step of which is quantified in this paper:
\begin{enumerate}

\item Apply the subadditive ergodic theorem to deduce that the limit~\eqref{limheur2} holds. The function~$\overline m_\mu$ is produced in the process. In fact, it is necessary to prove a more general fact which allows the vertex of the metric problem to be more free: namely
\begin{equation}\label{mmuconv}
\Prob\left[ \ \mbox{for all}\ y,z\in \Rd,\quad \limsup_{t \to \infty} \, \left| \frac1t m_\mu(ty,tz,\cdot) - \overline m_\mu(y-z) \right| = 0 \right] = 1.
\end{equation}
Then \emph{define} $\overline H$ by the formula~\eqref{Hbarformm}.

\item Using the comparison principle, argue that~\eqref{mmuconv} implies the limit~~\eqref{dvdheureq}, at least away from the flat spot $\{ p\,:\, \overline H(p) = \min \overline H\}$. The basic idea is to compare $m_\mu(\cdot,\omega)$ to $v^\delta(\cdot,\omega\,;p)$ where $\mu = \overline H(p)$. If $\delta v^\delta(0,\omega\,;p)$ is found to be too large or small, then this information is translated in terms of the metric problem to yield that $m_\mu(y,z,\omega) $ is relatively small or large compared to $\overline m_\mu(y-z)$, for some $|y|,|z|\simeq \delta^{-1}$. See~\eqref{upwinded} below, for example, for a quantitative version of this assertion. Meanwhile, on the flat spot, the proof is completely different and necessarily indirect: the metric problem cannot ``see" the flat spot. (Indeed, note that the error estimates we obtain for $m_\mu$ degenerate as $\mu\downarrow 0$, and in fact the convergence rate turns out to depend in a more delicate way on the law of~$H$.)

\item Using that $v^\delta$ is an approximate corrector, we argue that~\eqref{dvdheureq} implies the full statement of qualitative homogenization. This has been well-known for some time and also follows from a (more routine) comparison argument. A quantitative version appears in Lemma~\ref{incluHH}.

\end{enumerate}

We next give some details about how we select the points $y$ and $z$ in the comparison argument in~Step~(2) of the outline above. These observations will be needed in Section~\ref{EEdvd}. From elementary convex geometric considerations (we again refer to~\cite{ASo3} for details) we deduce that, for every $p \not\in \intr\{ q\in\Rd\, : \, \overline H(q) = 0\}$, we can find a direction $e$ so that the plane with slope $p$ touches $\overline m_\mu$ from below at $e$. Precisely, setting $\mu := \overline H(p)$, there exists $e\in \Rd$ with $|e|=1$ such that
\begin{equation}\label{suppe}
\overline  m_\mu(e) - p\cdot e = 0 = \min_{x\in\Rd} \left( \overline  m_\mu(x) - p\cdot x \right).
\end{equation}
The points $y$ and $z$ found in the comparison argument are chosen in such a way that $y-z \simeq te$ for some $t\simeq \delta^{-1}$. Geometrically the idea is clear: $m_\mu$ is a cone, and if we look far from the origin in the direction of the vector $e$, then $\overline m_\mu$ starts to resemble the plane $p\cdot x$. Since $x\mapsto p\cdot x + v^\delta(x,\omega;p)$ should resemble the same plane, it is natural to compare it with $m_\mu(\cdot,y,\omega)$. This is precisely the idea of the proof of Theorem~\ref{acpEE} in Section~\ref{EEdvd}.

\subsection{Other preliminary results}
To control the oscillations of solutions of the metric problem around their means, we use the ``martingale method of bounded differences" based on the Azuma inequality~\cite{Az}. See McDiarmind~\cite{Mc} and Alon and Spencer~\cite{AlS} for an overview of this probabilistic method, as well as a proof of Azuma's inequality, which is stated as follows. 

\begin{prop}[Azuma's inequality] \label{azuma}
Let $\{ X_k\}_{k\in\N}$ be a discrete martingale with $X_0\equiv 0$. Assume that there exists a constant $A> 0$ such that, for each $k\in\N$,
\begin{equation*}\label{}
\esssup_{\Omega} |X_{k+1}-X_k| \leq A.
\end{equation*}
Then, for each $\lambda > 0$ and $N\geq 1$,
\begin{equation*}\label{}
\Prob\left[ |X_N| > \lambda \right] \leq \exp\left( -\frac{\lambda^2}{2A^2N} \right).
\end{equation*}
\end{prop}

We next state Hammersley's generalization of Fekete's lemma on subadditive functions. For a proof, see~\cite[Theorem 2]{H2}. 

\begin{lem}[Hammersley-Fekete lemma]\label{Hamm}
Suppose that $\xi > 0$ and $f:[\xi,\infty) \to \R$ satisfies, for every $s,t\geq \xi$,
\begin{equation*}\label{}
f(s+t) \geq f(s) + f(t) -\Delta(s+t),
\end{equation*}
where $\Delta:[\xi,\infty) \to \R$ is nondecreasing such that
\begin{equation*}\label{}
\int_\xi^\infty \frac{\Delta(s)}{s^2} \, ds < \infty. 
\end{equation*}
Then $\tau := \lim_{s\to \infty} f(s)/s \in (-\infty,\infty]$ exists and, for every $t > \xi$,
\begin{equation*}\label{}
\tau \geq \frac{f(t)}{t} + \frac{\Delta(t)}{t} - 4\int_{2t}^\infty \frac{\Delta(s)}{s^2} \, ds.
\end{equation*}
\end{lem}

Several of our arguments rely on the comparison principle for viscosity solutions of first-order equations, typically in the following form (see~\cite{CIL} or~\cite{Ba} for a proof).

\begin{prop} \label{comp}
Let $G \in C(\Rd \times\Rd)$, $U$ be a bounded open subset of $\Rd$, and $u,-v \in \USC(\overline U)$ and $f\in C(U)$ satisfy
\begin{equation*}\label{}
G(Du,y) < f(y) < G(Dv,y) \quad \mbox{in} \ U. 
\end{equation*}
Then
\begin{equation*}\label{}
\sup_{U} (u-v) = \max_{\partial U} (u-v). 
\end{equation*}
\end{prop}

\section{Estimating the fluctuations of the metric problem} \label{Met1}

There are essentially two steps in the proof of Theorem~\ref{mpEE}. The first is to obtain exponential error estimates controlling the fluctuations of $m_\mu(y,0,\cdot)$ about its mean
\begin{equation*}\label{mean}
M_\mu(y):=\E \left[ m_\mu(y,0,\cdot) \right].
\end{equation*}
This is the focus of this section. In Section~\ref{Met2}, we complete the proof of Theorem~\ref{mpEE} by estimating the difference between the deterministic quantities $M_\mu(y)$ and $\overline m_\mu(y)$, which is more involved. 

Throughout this section we assume that $H$ satisfies \eqref{assum}, but we do not assume~\eqref{plushyp}. We also fix $K \geq 1$ and $\mu$ such that
\begin{equation}\label{CondMu}
0 < \mu \leq K.
\end{equation}
We denote by $C$ and $c$ positive constants depending only on $K$, the underlying dimension $\d$ and the assumptions for $H$, and which may vary from line to line. Several of our estimates depend on a lower bound for $\mu$, and since we must keep track of this dependency, we explicitly display dependence on~$\mu$.

The goal of this section is to prove the following exponential estimate for  the fluctuations of $m_\mu(y,0,\cdot)$.

\begin{prop} \label{kesten}
There exists $C>0$ such that, for each $\lambda > 0$ and $|y| > 1$,
\begin{equation}\label{oscbnd}
\Prob\Big[\, \left| m_\mu(y,0,\cdot) -M_\mu(y) \right| > \lambda \Big] \leq \exp\left(-\frac{\mu \lambda^2}{ C |y|} \right).
\end{equation}
\end{prop}

\subsection{A discretization scheme} \label{dscheme}

In the proof of Proposition~\ref{kesten}, below, it is useful to employ a  discretization scheme which allows us to essentially condition on the identity of the reachable set~$\Re_{\mu,t}$ in order to apply the independence assumption~\eqref{indy} in the form of Lemma~\ref{indyass}, below.

To introduce the discretization, we define, for every $r> 0$, 
\begin{equation*}\label{}
\mathcal K_r : = \left\{ A \in \B \, : \, A=\overline A \subseteq \overline B_r \right\}.
\end{equation*}
Recall that $\mathcal K_r$ is a compact metric space under the \emph{Hausdorff distance} (c.f. Munkres~\cite{Munk}), which is  defined by
\begin{align*}\label{}
\dist_H(E,F) := \inf_{x\in E} \sup_{y\in F} |x-y| \vee \inf_{y\in F} \sup_{x\in E} |x-y| = \ \inf \left\{ \ep > 0 : E \subseteq F+B_\ep \ \mbox{and} \ F \subseteq E+B_\ep \right\}.
\end{align*}
Note that, for every $E,F\in \mathcal B$,
\begin{equation} \label{distdist}
\dist_H(E,F) \geq \dist\left(E,\Rd\setminus F\right).
\end{equation}
Fix a small parameter $\delta > 0$. Then, by the compactness of $\mathcal K_r$, there exists $\ell =\ell(\delta,\d,r)\in \N$ and a disjoint partition $\Gamma_1,\ldots,\Gamma_\ell \subseteq \mathcal K_r$ of $\mathcal K_r$ into Borel subsets with $\diam_H(\Gamma_i) \leq \delta$. Let $K_i \in \mathcal B$ be the closure of the union of the elements of $\Gamma_i$. Then for each $A\in \mathcal K_r$, there exists a unique $1 \leq i \leq \ell$ such that $A\in \Gamma_i$, which in turn implies that $A\subseteq K_i$. Also define 
\begin{equation*}\label{}
\widetilde K_i:= K_i + B_{1}
\end{equation*}
so that $\dist(K_i , \Rd\setminus \widetilde K_i) = 1$. We have arranged things so that, for every $1\leq i \leq \ell$,
\begin{equation}\label{indyKs}
\mathcal G(K_i) \quad \mbox{and} \quad \mathcal G\left( \Rd \!\setminus \! \widetilde K_i \right) \quad \mbox{are independent}
\end{equation}
and, for each $A\in \mathcal K_r$ and $1\leq i \leq \ell$,
\begin{equation}\label{dutyfree}
{\rm if }\; A \in \Gamma_i\ ,  \qquad \mbox{then} \qquad A\subseteq K_i \subseteq \widetilde K_i \subseteq A+B_{1+\delta}.
\end{equation}
We remark that $\ell$, the partition $\{ \Gamma_i\} $ as well as the $K_i$'s depend on $r$ and $\delta$, but for convenience we do not explicitly display this dependence.

The following lemma captures the intuitively obvious assertion that the behavior of the medium inside the set $\Re_{\mu,t}^\omega$, conditioned on the event that $\Re_{\mu,t}^\omega\in \Gamma_i$ (which implies, in particular, $\Re^\omega_{\mu,t}\subseteq K_i$) is independent of the behavior of $m_\mu(y,\widetilde K_i,\omega)$. Recall that the latter is defined in~\eqref{mmuK} and is independent of $\mathcal G(K_i)$ by~\eqref{infKmeas}. Roughly speaking, this statement is a pre-processed form of the independence assumption which, as we will see, is particularly well-adapted to our needs in the proof of Proposition~\ref{kesten}.

To state the lemma it is necessary to define, for each $\mu> 0$, the filtration $\left\{ \mathcal F_{\mu,t} \right\}_{t\geq 0}$ by $\mathcal F_{\mu,0}:= \left\{ \Omega, \emptyset \right\}$ and, for every $t> 0$,
\begin{equation}\label{filtrate}
\mathcal F_{\mu,t} := \ \mbox{$\sigma$--field generated by} \quad \omega\mapsto H(p,x,\omega) \indc_{\{ \omega\,:\,x\in \Re_{\mu,t}^\omega\}}, \quad p,x\in\Rd.
\end{equation}
For every $0 < t < s$ and $\omega\in \Omega$, we have $\Re_{\mu,t}^\omega \subseteq \Re_{\mu,s}^\omega$ (see~\eqref{capture2}), and therefore $\mathcal F_{\mu,t} \subseteq \mathcal F_{\mu,s}$ provided $0\leq t \leq s$. Thus $\{ \mathcal F_{\mu,t} \}_{t\geq 0}$ is indeed a filtration. Observe that, for every $y\in \Rd$, 
\begin{equation}\label{mserbl1}
\omega\mapsto m_\mu (y,0,\omega) \indc_{\{ \omega\,:\,x\in \Re_{\mu,t}^\omega\}}(\omega) \quad \mbox{is $\mathcal F_{\mu,t}$--measurable.}
\end{equation}
Indeed, this is immediate from the formula~\eqref{measlocal}. Moreover, we see from this and~\eqref{capture} that, for every $y\in \Rd$ and $t\geq L_\mu |y|$,
\begin{equation}\label{mserbl2}
\omega\mapsto m_\mu (y,0,\omega) \quad \mbox{is $\mathcal F_{\mu,t}$--measurable.}
\end{equation}

\begin{lem} \label{indyass}
For each $1\leq i \leq \ell$, $t>0$ and $A\in \mathcal{F}_{\mu,t}$,
\begin{equation}\label{gamimeas}
\indc_{A\cap\{ \omega \,: \, \Re^\omega_{\mu,t}\in \Gamma_i\}} \quad \mbox{is} \ \ \mbox{$\mathcal{G}(K_i)$-measurable.}
\end{equation}
Moreover,
\begin{equation}\label{cacaboudin}
 \E \left[ m_\mu\big(y,\widetilde K_i,\cdot\big) \indc_{\{ \omega \,: \, \Re^\omega_{\mu,t}\in \Gamma_i\}}  \, \Big\vert\, \mathcal{F}_{\mu,t}\right] = \E\left[ m_\mu\big(y,\widetilde K_i,\cdot\big) \right] \indc_{\{ \omega \,: \, \Re^\omega_{\mu,t}\in \Gamma_i\}}.
\end{equation}
Here $m_\mu(y,K,\cdot)$ is defined in~\eqref{mmuK}.
\end{lem}
\begin{proof}
It suffices to show~\eqref{gamimeas} for $A$ of the form 
\begin{equation}\label{}
A= \left\{ \omega\in \Omega \, : \, H(p,x,\omega) \leq \alpha \right\} \cap \left\{  \omega\in \Omega \, : \,  m_\mu(x,0,\omega) \leq t\right\},
\end{equation}
where $p,x\in\Rd$ and $\alpha \in \R$, since such events $A$ generate $\mathcal F_{\mu,t}$. Recall the definition of $m^K_\mu$ in~\eqref{e.defmKmu} for nonempty and closed $K\subseteq \Rd$, the fact that $m^K_\mu$ is $\mathcal G(K)$--measurable and the fact from~\eqref{e.gotbunk} that, assuming $0\in K$, we have
\begin{equation}\label{RemKt2}
m_\mu^{K}(y,0,\omega) \geq t \quad \mbox{on} \ \partial K \qquad \mbox{implies that} \qquad \Re_{\mu,t}^\omega= \left\{ y\in \Rd\,:\, m_\mu^{K}(y,0,\omega)\leq t\right\}.
\end{equation}
Due to~\eqref{dutyfree} and~\eqref{RemKt2}, 
\begin{equation*}
A \cap \left\{ \omega \in \Omega \, : \, \mathcal{R}_{\mu,t}^\omega \in \Gamma_i \right\} = \left\{ \omega\in \Omega \, : \, H(p,x,\omega) \leq \alpha\;,  \;   \{ m_\mu^{ K_i}(\cdot,\omega)\leq t\} \in \Gamma_i \; {\rm and} \; 
m^{K_i}_\mu(x,\omega)\leq t\right\}.
\end{equation*}
If $x\in K_i$, then this set clearly belongs to $\mathcal G(K_i)$. If $x\not \in K_i$, then it is impossible that $x\in \Re^\omega_{\mu,t}$ and $\{ y \,:\, m_\mu^{ K_i}(y,\omega)\leq t \} \in \Gamma_i$, in view of~\eqref{dutyfree} and~\eqref{RemKt2}. Thus the above set is empty (and in particular belongs to $\mathcal G(K_i)$) in the case that $x\not\in K_i$. This confirms~\eqref{gamimeas}.

According to \eqref{indyKs}, \eqref{infKmeas}, and \eqref{gamimeas}, for every $A\in \mathcal F_{\mu,t}$, the event $A\cap \{ \omega \,: \, \Re^\omega_{\mu,t}\in \Gamma_i\}$ is independent of the random variable $m_\mu\big(y,\widetilde K_i,\cdot\big)$. Hence for every $A\in \mathcal F_{\mu,t}$,
\begin{align*}
\E \left[ m_\mu\big( y,\widetilde K_i,\cdot) \indc_{A \cap \{ \omega \,: \, \Re^\omega_{\mu,t}\in \Gamma_i\}} \right] = \E \left[ m_\mu\big( y,\widetilde K_i,\cdot) \right] \Prob\left[ A\cap \{ \omega \,: \, \Re^\omega_{\mu,t}\in \Gamma_i\} \right].
\end{align*}
The claim~\eqref{cacaboudin} now follows.
\end{proof}

\subsection{Controlling the fluctuations of~$m_\mu(y,0,\cdot)$}
We proceed with the demonstration that, for large $|y|$, the probability that $m_\mu(y,0,\cdot)$ is relatively far from its mean is small. We use an argument inspired by the pioneering work of Kesten~\cite{K2} in the theory of first-passage percolation, who introduced a martingale method based on Azuma's concentration inequality. We also benefit with some very elegant simplifications of the argument due recently to Zhang~\cite{Z}.

Notice that, unlike in percolation theory (or its continuum analogue), our Hamiltonian is not assumed to be positively homogeneous. In this generality, it is necessary to keep track of the dependence of the estimates on a lower bound for $\mu$. We recall that, in view of~Proposition~\ref{existMP}(iv) and~\eqref{CondMu}, there exist $l_\mu,L_\mu > 0$ such that
\begin{equation}\label{cLlL}
0 < c\mu \leq l_\mu \leq L_\mu \leq C
\end{equation}
and, for every $x,y\in \Rd$, 
\begin{equation}\label{control}
l_\mu |y-x| \leq m_\mu(y,x,\omega) \leq L_\mu|y-x|.
\end{equation}
In the control theory interpretation (see Remark~\ref{contform}), this important estimate, which we use many times below, provides upper and lower bounds on the lengths of optimal paths  connecting two points $x,y\in\Rd$. 

\begin{proof}[{\bf Proof of Proposition~\ref{kesten}}]
To setup the argument, we fix $y\in \Rd$ with $|y| > 1$ and define
\begin{equation}\label{}
T:= L_\mu |y| \qquad \mbox{and} \qquad r:= T/ l_\mu  = L_\mu  |y| / l_\mu
\end{equation}
so that (see~\eqref{capture}), for every $\omega\in \Omega$,
\begin{equation}\label{e.capping}
\Re_{\mu,T}^\omega \subseteq \overline B_{r}.
\end{equation}
We employ the discretization scheme, for fixed $\delta > 0$, as described in Subsection~\ref{dscheme}, with the notation introduced there, and we define a continuous-time martingale $\{ X_t \}_{t\geq 0}$ adapted to $\{ \mathcal F_{\mu,t} \}_{t\geq 0}$ by setting, for each $t\geq 0$,
\begin{equation}\label{}
X_t : = \E \big[ m_\mu(y,0,\cdot) \, \vert \, \mathcal{F}_{\mu,t} \big] - M_\mu(y).
\end{equation}
Here $\mathcal F_{\mu,t}$ is the filtration defined in~\eqref{filtrate}. Due to $\mathcal F_{\mu,0} = \{ \varnothing, \Omega \}$, $T=L_\mu|y|$ and \eqref{mserbl2},
\begin{equation}\label{snag}
X_0 \equiv 0 \qquad \mbox{and} \qquad X_t(\omega) \equiv m_\mu(y,0,\omega) - M_\mu(y) \quad \mbox{for every} \ t \geq T.
\end{equation}
Our goal is to apply Azuma's inequality in order to estimate the oscillations of $X_T$. We must first obtain an estimate of the form
\begin{equation}\label{azzwts}
\esssup_{\omega\in\Omega} \left|X_s(\omega) - X_t(\omega) \right|  \leq A + B|s-t|.
\end{equation}

\emph{Step 1.} We derive an inequality of the form~\eqref{azzwts}. Owing to~\eqref{mserbl2}, for every $0<t\leq s$,
\begin{equation}\label{}
\E\big[ m_\mu(y,0,\cdot) \indc_{\{ \omega \,: \,y\in \Re^\omega_{\mu,t} \}} \, \vert \, \mathcal F_{\mu,s} \big] = m_\mu(y,0,\cdot) \indc_{\{ \omega \,: \, y\in \Re^\omega_{\mu,t} \}}
\end{equation}
and hence
\begin{equation}\label{dMiff1}
X_s - X_t = \E \big[ m_\mu(y,0,\cdot) \indc_{\{ \omega\,:\, y\not\in \Re_{\mu,t} \}} \, \vert \, \mathcal{F}_{\mu,s} \big] - \E \big[ m_\mu(y,0,\cdot) \indc_{\{ \omega\,:\,y\not\in \Re_{\mu,t} \}} \, \vert \, \mathcal{F}_{\mu,t} \big].
\end{equation}
Using \eqref{DPRT}, we find that
\begin{equation*}\label{}
m_\mu(y,0,\omega) \indc_{\{ \omega\,:\, y\not\in \Re_{\mu,t} \}}(\omega) = \left( m_\mu(y,\Re_{\mu,t}^\omega,\omega) + t \right) \indc_{\{ \omega\,:\,y\not\in \Re_{\mu,t} \}}(\omega)
=  m_\mu(y,\Re_{\mu,t}^\omega,\omega) + t \indc_{\{ \omega\,:\,y\not\in \Re_{\mu,t} \}}(\omega)
\end{equation*}
and, since $\{ \omega\,:\,y\not\in \Re_{\mu,t} \} \in \mathcal F_{\mu,t}$, we may simplify \eqref{dMiff1} to write
\begin{equation}\label{dMiff2}
X_s - X_t = \E \big[ m_\mu(y,\Re_{\mu,t},\cdot)\, \vert \, \mathcal{F}_{\mu,s} \big] - \E \big[ m_\mu(y,\Re_{\mu,t},\cdot)  \, \vert \, \mathcal{F}_{\mu,t} \big].
\end{equation}
According to~\eqref{dynprog2}, for every $0< t \leq s$,
\begin{equation*}\label{}
m_\mu(y,\Re_{\mu,s}^\omega,\omega) \leq m_\mu(y,\Re_{\mu,t}^\omega,\omega) \leq (s-t) +  m_\mu(y,\Re_{\mu,s}^\omega,\omega).
\end{equation*}
Combining the last two lines, we obtain
\begin{align}\label{dMiff3}
|X_s - X_t| \leq (s-t) + \big| \E\left[  m_\mu(y,\Re_{\mu,s},\cdot) \, \vert \, \mathcal{F}_{\mu,s} \right] - \E\left[ m_\mu(y, \Re_{\mu,t},\cdot) \, \vert \, \mathcal{F}_{\mu,t} \right] \big|.
\end{align}

We next use the discretization scheme to estimate  $\E\left[ m_\mu(y, \Re_{\mu,t},\cdot) \, \vert \, \mathcal{F}_{\mu,t} \right]$ by approximating the integral represented by the expectation as a sum of characteristic functions. With $K_i$, $\widetilde K_i$ and $\Gamma_i$ as described there, observe that, by~\eqref{subadd},~\eqref{dutyfree} and~\eqref{control},
\begin{multline}\label{flexxed}
m_\mu(y,\widetilde K_i,\cdot) \indc_{\{\omega\,:\,\Re_{\mu,t}\in \Gamma_i\}}\leq m_\mu(y,\Re_{\mu,t},\cdot) \indc_{\{\omega\,:\,\Re_{\mu,t}\in \Gamma_i\}} \\ \leq \left(L_\mu (1+\delta)+ m_\mu(y,\widetilde K_i,\cdot) \right) \indc_{\{\omega\,:\,\Re_{\mu,t}\in \Gamma_i\}}.
\end{multline}
Taking the conditional expectation of~\eqref{flexxed} with respect to $\mathcal F_{\mu,t}$ and applying~\eqref{cacaboudin}, we get
\begin{multline}\label{flexxxed}
\E \left[  m_\mu(y,\widetilde K_i,\cdot) \right] \indc_{\{\omega\,:\,\Re_{\mu,t}\in \Gamma_i\} } \leq \E \left[ m_\mu(y,\Re_{\mu,t},\cdot) \indc_{\{\omega\,:\,\Re_{\mu,t}\in \Gamma_i\}} \, \Big\vert \, \mathcal F_{\mu,t} \right] \\  \leq \left( L_\mu(1+\delta) + \E \left[  m_\mu(y,\widetilde K_i,\cdot) \right] \right) \indc_{\{\omega\,:\,\Re_{\mu,t}\in \Gamma_i\} }
\end{multline}
Since $\{ \Gamma_i \}$ is a disjoint partition of $\mathcal K_r$, we also have, in view of~\eqref{e.capping}, for every $1\leq t\leq s \leq T$,
\begin{equation}\label{musc}
\E \left[  m_\mu(y,\Re_{\mu,t},\cdot) \, \Big \vert \, \mathcal {F}_{\mu,t} \right] = \sum_{i,j=1}^\ell \E \left[ m_\mu(y,\Re_{\mu,t},\cdot) \indc_{\{\omega\,:\,\Re_{\mu,t}\in \Gamma_i\}} \, \Big \vert \, \mathcal {F}_{\mu,t} \right]  \indc_{\{\omega\,:\,\Re_{\mu,s}\in \Gamma_j\}}.
\end{equation}
Multiplying~\eqref{flexxxed} by $\indc_{\{\omega\,:\,\Re_{\mu,s}\in \Gamma_j\}}$ and summing over the indices $i$ and $j$ yields, in light of~\eqref{musc},
\begin{equation}\label{capped1}
0 \leq \E \left[  m_\mu(y,\Re_{\mu,t},\cdot) \, \Big \vert \, \mathcal {F}_{\mu,t} \right] - \sum_{i,j=1}^\ell \E \left[  m_\mu(y,\widetilde K_i,\cdot) \right] \indc_{\{\omega\,:\,\Re_{\mu,t}\in \Gamma_i\}} \indc_{\{\omega\,:\,\Re_{\mu,s}\in \Gamma_j\}} \leq L_\mu(1+\delta).
\end{equation}
In the same way, after interchanging $s$ for $t$ and $j$ for $i$, we also obtain
\begin{equation}\label{capped2}
0\leq \E \left[ m_\mu(y,\Re_{\mu,s},\cdot) \, \Big \vert \, \mathcal {F}_{\mu,s} \right] - \sum_{i,j=1}^\ell \E \left[ m_\mu(y,\widetilde K_j,\cdot) \right] \indc_{\{\omega\,:\,\Re_{\mu,t}\in \Gamma_i\}} \indc_{\{\omega\,:\,\Re_{\mu,s}\in \Gamma_j\}} \leq  L_\mu  (1+\delta).
\end{equation}
It follows that
\begin{multline}\label{boke}
\left| \E\big[  m_\mu(y,\Re_{\mu,s},\cdot) \, \vert \, \mathcal{F}_{\mu,s} \big] - \E\big[ m_\mu(y, \Re_{\mu,t},\cdot) \, \vert \, \mathcal{F}_{\mu,t} \big] \right| \\
\leq L_\mu (1+\delta) + \sum_{i,j=1}^\ell \left| \E \left[ m_\mu(y,\widetilde  K_i,\cdot) -  m_\mu(y,\widetilde K_j,\cdot) \right]\right| \indc_{\{\omega\,:\,\Re_{\mu,t}\in \Gamma_i\}} \indc_{\{\omega\,:\,\Re_{\mu,s}\in \Gamma_j\}}.
\end{multline}
If, for some $i,j =1,\ldots,\ell$, there exists $\omega$ belonging to the event that $\Re_{\mu,t}^\omega \in \Gamma_i$ and $\Re_{\mu,s}^\omega \in \Gamma_j$, then 
\begin{equation*}\label{}
\dist_H\left(\widetilde K_i,\widetilde K_j\right) \leq \dist_H\left(K_i,K_j\right) \leq \dist_H\left(\Re_{\mu,t}^\omega,\Re_{\mu,s}^\omega\right) + 2\delta \leq \frac{(s-t)}{l_\mu } + 2\delta.
\end{equation*}
Using~\eqref{lips}, we conclude that, for every $i,j=1,\ldots,\ell$,
\begin{multline*}\label{}
\left| \E \left[ m_\mu(y,\widetilde  K_i,\cdot) \right] - \E \left[  m_\mu(y,\widetilde K_j,\cdot) \right]\right| \indc_{\{\omega\,:\,\Re_{\mu,t}\in \Gamma_i\}} \indc_{\{\omega\,:\,\Re_{\mu,s}\in \Gamma_j\}} \\ \leq \left( \frac{L_\mu }{l_\mu }(s-t) + 2 L_\mu \delta \right)\indc_{\{\omega\,:\,\Re_{\mu,t}\in \Gamma_i\}} \indc_{\{\omega\,:\,\Re_{\mu,s}\in \Gamma_j\}}. 
\end{multline*}
Combining this with \eqref{boke} and sending $\delta \to 0$ yields
\begin{equation}\label{blok}
\left| \E\big[  m_\mu(y,\Re_{\mu,s},\cdot) \, \vert \, \mathcal{F}_{\mu,s} \big] - \E\big[ m_\mu(y,\Re_{\mu,t},\cdot) \, \vert \, \mathcal{F}_{\mu,t} \big] \right| \leq L_\mu + \frac{L_\mu }{l_\mu } (s-t).
\end{equation}
Finally, from \eqref{dMiff3}, we finally get, for every $0<s<t\leq T$,
\begin{equation}\label{dMiff4}
|X_t - X_s| \leq L_\mu + \left( \frac{L_\mu }{l_\mu } +1 \right)(s-t).
\end{equation}
This also holds for $0<s<t$ without further restriction by the second assertion of~\eqref{snag}.

\emph{Step 2.}
We finish the argument by applying Azuma's inequality, using~\eqref{dMiff4}. Define a discrete martingale sequence $\widetilde X_k : = X_{hk}$ with $h:=l_\mu L_\mu/ (l_\mu+L_\mu)$ and observe that, according to~\eqref{dMiff4}, for all $k\in\N$,
\begin{equation*}\label{}
\big| \widetilde X_{k+1} - \widetilde X_k \big| \leq 2L_\mu.
\end{equation*}
An application of Azuma's inequality (Proposition~\ref{azuma}) yields, for every $\lambda > 0$ and $N \in \N$,
\begin{equation}\label{azumapp}
\Prob\left[ \big| \widetilde X_N \big| > \lambda \right] \leq \exp\left( \frac{-\lambda^2}{8L_\mu^2N}\right).
\end{equation}
Let $N$ be the smallest integer larger than $T/h$ so that $\widetilde X_N = X_T = m_\mu(y,0,\cdot) - M_\mu(y)$. It follows that, since $|y| > 1$ and $T=L_\mu |y|$,
\begin{equation}\label{Nineq}
N \leq \frac{T}{h} + 1 \leq \frac{L_\mu (l_\mu+L_\mu) |y|}{l_\mu L_\mu} + 1\leq \frac{(2l_\mu+L_\mu) |y|}{l_\mu}.
\end{equation}
From \eqref{snag}, \eqref{azumapp} and \eqref{Nineq} we deduce 
\begin{equation*}\label{}
\Prob\Big[ \left| m_\mu(y,0,\cdot) - M_\mu(y) \right| > \lambda \Big] = \Prob\left[ \big| \widetilde X_N \big| > \lambda \right] \\
\leq \exp\left( \frac{-\lambda^2l_\mu}{8L_\mu^2(2l_\mu+L_\mu)|y|}\right). \qedhere
\end{equation*}
\end{proof}

\begin{remark} \label{notsharp}
Integrating~\eqref{oscbnd} yields
\begin{equation*}\label{varbnd}
\Var\left( m_\mu(y,0,\cdot) \right) = \int_0^\infty \Prob \Big[\left| m_\mu(y,0,\cdot) - M_\mu(y) \right| > \lambda^{\frac12} \Big] d\lambda \leq \int_0^\infty \exp\left(-\frac{\mu\lambda}{C |y|} \right) d\lambda = \frac{C}{\mu}|y|,
\end{equation*}
which mirrors the bound on the variance of the time constant obtained by Kesten~\cite{K2} for first passage percolation. Obtaining an optimal estimate for the fluctuations of the latter is a well-known open problem. It is conjectured that the variance of the time constant should behave, in dimension $\d=2$, like $O\big(|y|^\frac23\big)$ for large $|y|$, and it is believed that the oscillations should decrease in higher dimensions. Nevertheless, it is still open in every dimension $\d \geq 2$ whether, for some $\alpha<1$, this quantity is bounded by $O\left(|y|^\alpha\right)$ as $|y| \to \infty$. We expect that it will be similarly challenging to prove such a bound for our quantity $\Var(m_\mu(y,0,\cdot))$, and still more difficult to find the optimal exponent for the algebraic rate of homogenization of \eqref{HJq}.

In analogy with the best known variance bound in first-passage percolation, due to~Benjamini, Kalai and Schramm~\cite{BKS}, we expect that an estimate of the form
\begin{equation}\label{ylogy}
\Var\left( m_\mu(y,0,\cdot) \right) \leq  \frac{C}{\mu}\!\left( \frac{|y|}{ \log|y|}\right)
\end{equation}
can be proved, in dimensions $\d \geq 2$, by an application of Talagrand's concentration inequality~\cite{Ta}. In fact, as we were completing the writing of this paper, we received a new preprint by Matic and Nolen~\cite{MN} who have obtained, in a slightly different setting, a bound like~\eqref{ylogy} for a certain class of Hamilton-Jacobi equations in special i.i.d.~environments.
\end{remark}

\section{Estimating the statistical bias of the metric problem} \label{Met2}

Having estimated the oscillations of $m_\mu(y,0,\cdot)$ about its mean $M_\mu(y)$ in Proposition~\ref{kesten}, in order to prove Theorem~\ref{mpEE} it remains to estimate the rate at which the means $t^{-1}M_\mu(ty)$ converge, as $t\to \infty$, to their limit $\overline m_\mu(y)$. On one side our task is trivial. Indeed, by~\eqref{mmustat} and~\eqref{subadd}, we have
\begin{equation}\label{subaddM}
M_\mu(y+z) = \E \left[ m_\mu(y+z,0,\cdot) \right] \leq \E\left[ m_\mu(y,0,\cdot)\right] + \E\left[ m_\mu(y+z,y,\cdot)\right] = M_\mu(y) + M_\mu(z).
\end{equation}
It follows from Fekete's lemma (Lemma~\ref{Hamm} in the special case that $\Delta \equiv 0$) and~\eqref{mmuconv} that, for every $y\in \Rd$,
\begin{equation}\label{tyi}
M_\mu(y) \geq \inf_{t\geq1} t^{-1} M_\mu(ty) = \lim_{t\to \infty} t^{-1} M_\mu(ty) = \overline m_\mu(y).
\end{equation}
The estimate~\eqref{easy} is then immediate.

In order to prove~\eqref{hard} we are confronted with the more difficult task of finding good upper bounds for $M_\mu(y) - \overline m_\mu(y)$, which are stated in the following proposition.

\begin{prop} \label{means}
There exists $C >0$ such that, for every $|y|>1$,
\begin{equation}\label{meanseq}
M_\mu(y) \leq \overline m_\mu(y) + C\left( \frac{|y|^{\frac12}}{\mu^{\frac32}} + \frac{|y|^\frac23}{\mu} \right) \left( \log\left(1+\frac{|y|}{\mu}\right)\right)^\frac12.
\end{equation}
\end{prop}

The previous proposition provides the desired estimate for the difference between $t^{-1} M_\mu(ty)$ and $\overline m_\mu(y)$ for large $t>0$, and now the proof of Theorem~\ref{mpEE} follows:

\begin{proof}[{\bf Proof of~Theorem~\ref{mpEE}}]
The first inequality follows from~\eqref{tyi} and~\eqref{oscbnd} and second  from~\eqref{meanseq} and~\eqref{oscbnd}.
\end{proof}

Proving Proposition~\ref{means} is the focus of the rest of this section. A typical argument for obtaining such an estimate and the strategy we use here involves approximating $M_\mu(y)$ by another quantity which is \emph{superadditive} (see also the discussion in Hammersley~\cite{H}). Fekete's lemma may then be applied ``from the other side" to obtain an estimate on the deviation of this approximate quantity from its asymptotic limit. The desired estimate in terms of the original quantity then follows, depending on the quality of the approximation. This strategy was used by Alexander~\cite{A} in the context of first-passage percolation to obtain estimates on the deviation of the expected passage time from the limiting time constant.

In our context, it turns out to be more convenient to first obtain estimates for the difference between the quantities 
\begin{equation*}\label{}
\E \left[ m_\mu(H_t,0,\cdot) \right] : = \E \left[ \min_{z\in H_t} m_\mu(z,0,\cdot) \right]\qquad \mbox{and} \qquad \overline m_\mu(H_t) : = \min_{z\in H_t} \overline m_\mu(z),
\end{equation*}
where $H_t$ is a given plane at a distance $t$ from the origin. We then argue that the value of $\E \left[ m_\mu(H_t,0,\cdot) \right]$ must be close to 
\begin{equation*}\label{}
M_\mu(H_t) : = \min_{z\in H_t} M_\mu(z),
\end{equation*}
which yields good estimates for the deviation of the latter quantity from  $\overline m_\mu(H_t)$. This is then transformed, using a simple geometric argument, into an estimate for $M_\mu(y) - \overline m_\mu(y)$ for large $|y|$. 

Here is an illustration of the outline of key steps in the proof of Proposition~\ref{means}:
\begin{equation*} \label{}
M_\mu(H_t) - \overline m_\mu(H_t) \hspace{1em}  = \hspace{1em}   \underbrace{M_\mu(H_t) - \E \left[ m_\mu(H_t,0,\cdot) \right]}_{\mbox{{\small estimated by Lemma~\ref{cmte-minE}}}} \hspace{1em} + \hspace{1em}  \underbrace{\E \left[ m_\mu(H_t,0,\cdot) \right] - \overline m_\mu(H_t)}_{\mbox{{\small estimated by Lemma~\ref{EHtrate}}}}
\end{equation*}
\begin{equation*} \label{}
\mbox{estimate for} \ \ M_\mu(H_t) - \overline m_\mu(H_t)
\hspace{1em}
 \underset{\mbox{{\small Lemma~\ref{CHlem}}}}{\xrightarrow{\hspace{5em}}} \hspace{1em}
\mbox{estimate for} \ \ M_\mu(y) - \overline m_\mu(y).
\end{equation*}

As in Section~\ref{Met1}, we fix $K>0$ and $\mu$ satisfying~\eqref{CondMu}. The symbols $C$ and $c$ denote positive constants which may depend on $K$ and $H$ and may vary in each occurrence. 

\subsection{Introduction of the approximating quantity}
It is difficult to work directly with statistical properties of the quantity $m_\mu(H_t,0,\cdot)$. We consider instead an approximating quantity to which the independence assumption is easier to apply. Fix a unit direction $e \in \Rd$ which for notational convenience we take to be $e=e_\d=(0,\ldots,0,1)$. For each $t> 0$, define the plane
\begin{equation}\label{}
H_t := te + \{ e \}^\perp = \left\{ (x',t) \, : \, x'\in \R^{\d-1} \right\}
\end{equation}
and its discrete analogue
\begin{equation}\label{}
\widehat H_t : = \left\{ (n,t) \, : \, n\in \Z^{\d-1} \right\}.
\end{equation}
We also denote, for $t> 0$, the halfspaces
\begin{equation}\label{}
H_t^+=\left\{(x',x_d)\in \Rd \, : \, x_d\geq t \right\} \qquad \mbox{and} \qquad H_t^-=\left\{(x',x_d)\in \Rd \, : \, x_d\leq t \right\}.
\end{equation}
Define, for each $\sigma,t>0$, the quantities
\begin{equation}\label{defgt}
G_{\mu,\sigma} (t):= \sum_{y\in \widehat{H}_t}  \E \left[ \exp\left( -\sigma m_\mu(y,0,\cdot) \right) \right] \qquad \mbox{and} \qquad g_{\mu,\sigma} (t):= - \frac{1}{\sigma} \log G_{\mu,\sigma}(t).
\end{equation}
Below we will see that $g_{\mu,\sigma}(t)$  is a good approximation of $\E \left[ m_\mu(H_t,0,\cdot) \right]$ for appropriate choices of the parameter $\sigma>0$. Since $g_{\mu,\sigma}$ is a logarithm of the expectation of the sum of exponentials, it can be used naturally with the independence assumption (see the proof of Lemma~\ref{superadd}, below). 

We begin with a technical lemma, also used many times below, which asserts that a substantial portion of the quantity $G_{\mu,\sigma}(t)$ is contributed by lattice points $(n,t) \in \widehat H_{t}$ with $|n| \leq O(t)$. This implies in particular that $G_{\mu,\sigma}$ and $g_{\mu,\sigma}$ are finite.

\begin{lem} \label{chop}
There exists $C> 0$ such that, for every $t> 0$, $0< \sigma \leq 1$ and $R\geq 2(L_\mu/l_\mu)t$,
\begin{equation}\label{chopin}
G_{\mu,\sigma}(t) \leq \C(l_\mu \sigma)^{1-d}  \sum_{y\in \widehat H_t \cap B_R }  \E\left[ \exp\left( -\sigma m_\mu(y,0,\cdot) \right) \right].  
\end{equation}
\end{lem}
\begin{proof}
According to~\eqref{control}, for every $\omega\in \Omega$ and $y\in\Rd$,
\begin{equation}\label{un}
\exp\left( -L_\mu \sigma |y| \right) \leq \exp\left(-\sigma m_\mu(y,0,\omega) \right) \leq \exp\left( -l_\mu \sigma |y| \right).
\end{equation}
Thus
\begin{multline*}
\sum_{y\in \widehat H_t \setminus B_R } \exp\left( -\sigma m_\mu(y,0,\omega) \right) \leq \sum_{y\in \widehat H_t \setminus B_R } \exp \left( -l_\mu \sigma |y| \right)   \leq \sum_{y\in \Z^{\d-1} \setminus B_{R} } \exp \left( -l_\mu \sigma |y| \right) \\ \leq C \int_{R}^\infty r^{\d-2} \exp\left( -l_\mu \sigma r \right)\, dr = C(l_\mu \sigma)^{1-d}\int_{l_\mu \sigma R}^\infty r^{\d-2} \exp\left(-r \right)\, dr
\end{multline*}
and, using the inequality $r^{d-2} \leq C \exp(r/2)$ to estimate the last integral on the right side, we obtain
\begin{equation}\label{une}
\sum_{y\in \widehat H_t \setminus B_R } \exp\left( -\sigma m_\mu(y,0,\omega) \right) \leq C (l_\mu \sigma)^{1-d}\int_{l_\mu \sigma R}^\infty \exp\left(-r/2 \right)\, dr.
\end{equation}
On the other hand,
\begin{equation}\label{deux}
 \exp\left(-\sigma m_\mu(te,0,\omega) \right) \geq \exp\left(-L_\mu
\sigma t\right) = \int_{2L_\mu \sigma t }^\infty \exp\left( - r/2\right)\, dr.
\end{equation}
Since $2R \geq (L_\mu /l_\mu)t$, it follows from \eqref{une} and \eqref{deux} that, for every $\omega\in \Omega$,
\begin{align*}\label{}
 \sum_{y\in \widehat H_t \setminus B_R } \exp\left( -\sigma m_\mu(y,0,\omega) \right) & \leq C (l_\mu\sigma)^{1-\d}
 \exp\left(-\sigma m_\mu(te,0,\omega) \right)  \\
 & \leq C(l_\mu \sigma)^{1-\d} \sum_{y \in \hat H_t\cap B_R} \exp\left(-\sigma m_\mu(y,0,\omega) \right).
\end{align*}
Taking expectations yields
\begin{align*}\label{}
G_{\mu,\sigma}(t) & =  \sum_{y\in \widehat H_t \cap B_R}  \E \left[ \exp\left( -\sigma m_\mu(y,0,\cdot) \right)\right] + \sum_{y\in \widehat H_t \setminus B_R } \E \left[ \exp\left( -\sigma m_\mu(y,0,\cdot) \right)  \right] \\ 
&\leq \left(1+C (l_\mu \sigma)^{1-\d} \right)  \sum_{y\in \widehat H_t \cap B_R } \E \left[ \exp\left( -\sigma m_\mu(y,0,\cdot) \right) \right].
\end{align*}
We now obtain~\eqref{chopin}, since $\sigma\leq 1$ implies $l_\mu\sigma \leq C$ and hence $1+C(l_\mu \sigma)^{1-\d} \leq (1+C)(l_\mu\sigma)^{1-\d}$.
\end{proof}

Next we show that $g_{\mu,\sigma}(t)$ gives a good upper bound for $\E\left[ m_\mu(H_t,0,\cdot)\right]$ for large $t$ and appropriate choices of $\sigma>0$. 

\begin{lem} \label{gooda}
There exists $C> 0$ such that, for every $t> 1$ and $0 < \sigma \leq 1$,
\begin{equation}\label{goodappr}
\E \left[ m_\mu(H_t,0,\cdot) \right] - C \left( \frac{\sigma t}{\mu^2} + \frac1\sigma \log \left(1+ \frac t{\sigma\mu} \right) \right) \leq g_{\mu,\sigma}(t) \leq \E\left[ m_\mu(H_t,0,\cdot) \right] + C.
\end{equation}
\end{lem}
\begin{proof}
The upper bound in \eqref{goodappr} is easy. Using~\eqref{lips}, we have
\begin{multline*}\label{}
G_{\mu,\sigma}(t) \geq \E \left[ \, \sup_{y\in \widehat H_t} \exp\left( -\sigma m_\mu(y,0,\cdot) \right) \right] = \E \left[ \exp\left( -\sigma \inf_{y\in \widehat H_t} m_\mu(y,0,\cdot) \right) \right] \\
\geq \E \left[\exp\left(  -\sigma \left( m_\mu(H_t,0,\cdot) +L_\mu(\d-1)^{\frac12}\right) \right)\right]  =\exp\left(-\sigma L_\mu(\d-1)^{\frac12} \right) \,\E \left[ \exp\left(-\sigma m_\mu(H_t,0,\cdot) \right) \right].
\end{multline*}
After taking the logarithm of both sides of this inequality, an application of Jensen's inequality and a rearrangement yield the second inequality of \eqref{goodappr} with $C = (\d-1)^{\frac12} L_\mu$.

To obtain the lower bound, we use both \eqref{oscbnd} and \eqref{chopin}. 
For every $|y| >1$, we have
\begin{align*}
\lefteqn{\E \left[ \exp\left( -\sigma m_\mu(y,0,\cdot) \right) \right] = \int_0^\infty \sigma \exp(-\sigma s) \Prob \left[ m_\mu(y,0,\cdot) \leq s \right] \, ds} \qquad \qquad & \\
& \leq  \exp\left( -\sigma M_\mu(y) \right) + \int_0^{M_\mu(y)} \sigma\exp(-\sigma s) \Prob \left[ m_\mu(y,0,\cdot) \leq s \right] \, ds\\
& = \left( 1 + \sigma \int_0^{M_\mu(y)} \exp(\sigma \lambda) \Prob\left[ m_\mu(y,0,\cdot) - M_\mu(y) \leq -\lambda \right]\, d\lambda \right) \exp\left( -\sigma M_\mu(y) \right).
\end{align*}
Applying \eqref{oscbnd}, we obtain
\begin{equation*}
\E \left[ \exp\left( -\sigma m_\mu(y,0,\cdot) \right) \right] \leq \left( 1 + \sigma \int_0^{M_\mu(y)} \exp\left(\sigma \lambda - \frac{\mu \lambda^2}{C |y|}\right) \, d\lambda \right) \exp\left( -\sigma M_\mu(y) \right).
\end{equation*}
We estimate the integrand above by completing the square, i.e.,
\begin{equation*}\label{}
\sigma \lambda - \frac{\mu \lambda^2}{C|y|} = -\frac{\mu}{C|y|} \left( \lambda - \frac{\sigma C |y|}{2\mu} \right)^2 + \frac{1}{4\mu} \sigma^2 C|y| \leq \frac{1}{4\mu} \sigma^2 C|y|,
\end{equation*}
and thus obtain
\begin{equation}\label{commybn}
\E \left[ \exp\left( -\sigma m_\mu(y,0,\cdot) \right) \right]  \leq \left( 1 +\sigma M_\mu(y) \exp\left(\frac{1}{4\mu} \sigma^2 C|y| \right) \right) \exp\left( -\sigma M_\mu(y) \right).
\end{equation}
Summing \eqref{commybn} over $y\in \widehat H_t\cap B_R$, with $R:=2(L_\mu / l_\mu)t$, and applying \eqref{chopin}, we get
\begin{align*}
G_{\mu,\sigma}(t) & \leq Cl_\mu^{1-d}\sigma^{1-\d} \sum_{y\in \widehat H_t \cap B_R} \left( 1 + \sigma M_\mu(y) \exp\left( \frac{1}{4\mu} \sigma^2 C|y|  \right) \right) \exp\left( -\sigma M_\mu(y) \right) \\
& \leq Cl_\mu^{1-d}\sigma^{1-\d} \sum_{y\in \widehat H_t \cap B_R} \left( 1 + \sigma M_\mu(y) \exp\left( \frac{1}{4\mu} \sigma^2 C|y| \right) \right) \exp\left( -\sigma \E \left[ m_\mu(H_t,0,\cdot) \right]   \right) \\
& \leq C l_\mu^{1-d}\sigma^{1-\d} R^{\d-1} \exp\left( -\sigma\E\left[ m_\mu(H_t,0,\cdot) \right] \right) \left( 1 + \sigma L_\mu t \exp\left( \frac{1}{4\mu} \sigma^2 CR\right) \right)
\end{align*}
In view of~\eqref{cLlL}, we have $l_\mu^{1-d} \leq C\mu^{1-d}$ and $R\leq Ct/\mu$. Using these with $\sigma \leq 1$, $t>1$, we obtain
\begin{equation*}
G_{\mu,\sigma}(t) \leq C t^{\d} \sigma^{1-\d} \mu^{2-2\d} \exp\left( -\sigma\E\left[ m_\mu(H_t,0,\cdot) \right] + \frac{ C \sigma^2 t}{\mu^2}\right).
\end{equation*}
Taking logarithms, dividing by $-\sigma$ and rearranging this expression yields:
\begin{equation*}
g_{\mu,\sigma}(t) = -\frac1\sigma \log G_{\mu,\sigma}(t) \geq \E \left[ m_\mu(H_t,0,\cdot) \right] - \frac{C\sigma t}{\mu^2} - \frac1\sigma\log\left( \frac{Ct^d}{\sigma^{d-1}\mu^{2d-2}} \right).
\end{equation*}
We can estimate the logarithm factor in the last term on the right side as follows:
\begin{equation*} \label{}
\log\left(\frac{Ct^d}{\sigma^{d-1}\mu^{2d-2}} \right) \leq C\log\left(1+ \frac{t}{\sigma\mu}\right).
\end{equation*}
This completes the proof of the lower bound of~\eqref{goodappr} and hence of the lemma.
\end{proof}

\subsection{The (almost) superadditivity of $g_{\mu,\sigma}$ and estimates for $\E \left[ m_\mu(H_t,0,\omega) \right] - \overline m_\mu(H_t)$} 

The next step is to prove that $g_{\mu,\sigma}$ is essentially superadditive, which is summarized in the following lemma. Unlike the approach taken in~\cite{A}, we do not use an abstract result like the van den Berg-Kesten inequality, which does not seem to easily apply in the continuous setting. We opt instead for a simpler ``splitting technique" to apply the independence assumption more directly. A similar technique was employed by Sznitman~\cite{SzEE}.

The critical property of the $m_\mu$'s needed here, which allows us to exploit the independence of the random medium, is the \emph{dynamic programming principle}. It asserts that, if every path from $x$ to $y$ passes through a surface, then, for some $z$ on the surface,  the cost of moving from $x$ to $y$ is equal to the sum of the cost of moving from $x$ to $z$ and from $z$ to $y$. Precisely, for every open $U \subseteq \Rd$ with $x\in U$ and every $y\in \Rd \setminus U$ and $\omega\in \Omega$, 
\begin{equation}\label{dynprog}
m_\mu(y,x,\omega) = \min_{z\in \partial U} \big( m_\mu(y,z,\omega) + m_\mu(z,x,\omega) \big).
\end{equation}
See Proposition~\ref{existMP}(vi). 

\begin{lem} \label{superadd}
There exists $C> 0$ such that, for every $s,t> 1$ and $0< \sigma \leq 1$,
\begin{equation}\label{superaddeq}
g_{\mu,\sigma}(t+s) \geq g_{\mu,\sigma}(t) + g_{\mu,\sigma}(s) - 
\frac {C}\sigma \left( 1+ \log\left(1+\frac{s+t}{\sigma\mu}\right) \right).
\end{equation}
\end{lem}
\begin{proof}
Fix $s,t> 1$, $y\in H_{s+t}$ such that $|y|\leq R:=2(L_\mu/l_\mu)(s+t)$ and $\omega \in \Omega$. Observe that, in view of  \eqref{dynprog}, 
\begin{equation*}
m_\mu(y,0,\omega) = \min_{z\in H_t} \left( m_\mu (z,0,\omega) + m_\mu(y,z,\omega)\right) \geq m_\mu (H_t,0,\omega) +m_\mu(y,H_{t},\omega).
\end{equation*}
Thus, using~\eqref{lips},
\begin{equation*}
m_\mu(y,0,\omega) \geq  m_\mu (H_t,0,\omega) +m_\mu(y,H_{t+1},\omega) - L_\mu.
\end{equation*}
Since $0\in H^-_t$ and $y\in H^+_{t+1}$, we have 
\begin{equation*}\label{}
 m_\mu (H_t,0,\omega)=  m_\mu (H_t^+,0,\omega) \quad \mbox{and} \quad   m_\mu(y,H_{t+1},\omega) = m_\mu(y,H_{t+1}^-,\omega),
\end{equation*}
which is immediate from $H_t \subseteq H_t^\pm$,~\eqref{mmuK} and~\eqref{control}. Applying \eqref{infKmeas}, we conclude that 
\begin{equation}\label{}
m_\mu (H_t^+,0,\cdot) \ \ \mbox{is} \ \ \mathcal{G}(H^-_t)\mbox{-measurable} \quad \mbox{and} \quad m_\mu (y,H_{t+1}^-,\cdot) \ \ \mbox{is} \ \  \mathcal{G}(H_{t+1}^+)\mbox{-measurable.}
\end{equation}
In light of~\eqref{indy}, these random variables are independent and thus
\begin{equation*}
\E\left[ \exp\left( -\sigma m_\mu(y,0,\cdot) \right) \right]
\leq \exp \left( \sigma L_\mu\right) \E\left[ \exp\left(-\sigma  m_\mu (H_t^+,0,\cdot)\right)\right]
\E\left[ \exp\left(-\sigma m_\mu(y,H_{t+1}^-,\cdot) \right)\right].
\end{equation*}

Returning to the discrete setting, we next claim that
\begin{equation}\label{hbecahb}
m_\mu (H_t^+,0,\omega) =m_\mu(H_t,0,\omega) \geq \min_{z\in \hat{H}_t\cap B_R}  m_\mu (z,0,\omega)-  L_\mu (d-1)^{\frac12}.
\end{equation}
Indeed, it is clear from~\eqref{control} that any $z\in \hat H_t$ attaining the (implicit) minimum  on the left side of~\eqref{hbecahb} must belong to $B_R$, and~\eqref{hbecahb} then follows from~\eqref{lips}. In a similar way,  since $|y|\leq R$, 
\begin{equation*}
m_\mu(y,H_{t+1}^-,\omega) \geq \min_{z\in \hat{H}_{t}\cap B_{2R}} m_\mu(y,z,\omega)- L_\mu\left(1+(d-1)^{\frac12}\right).
\end{equation*}
Combining these inequalities, we obtain
\begin{equation*}
\E\left[ \exp\left(-\sigma m_\mu(y,0,\cdot) \right) \right]
\leq \exp\left(C\sigma \right) \sum_{ z, z'\in \hat{H}_t\cap B_{2R}} \E\left[ \exp\left(-\sigma  m_\mu (z,0,\cdot)\right)\right]
\E\left[ \exp\left(-\sigma  m_\mu(y,z',\cdot) \right) \right].
\end{equation*}
Note that, if $z'\in \hat{H}_t$, then $y-z'\in \hat{H}_s$. So, 
in view of the definition of $G_{\mu, \sigma}$ and~\eqref{mmustat}, we have
\begin{equation*}
\sum_{ z'\in \hat{H}_t} 
\E\left[ \exp\left(-\sigma  m_\mu(y,z',\cdot) \right) \right]=
\sum_{ z'\in \hat{H}_t} 
\E\left[ \exp\left(-\sigma  m_\mu(y-z',0,\cdot) \right) \right]=
G_{\mu,\sigma}(s).
\end{equation*}
Therefore
\begin{equation*}
\E\left[ \exp\left(-\sigma m_\mu(y,0,\cdot) \right) \right]
\leq \exp\left(C \sigma \right) G_{\mu, \sigma}(t) G_{\mu,\sigma}(s).
\end{equation*}
Summing over all $y\in \hat H_{t+s} \cap B_{R}$, and using Lemma~\ref{chop} yields, in view of the definition of $R$,
\begin{align*}
G_{\mu, \sigma}(s+t) & \leq C R^{\d-1} l_\mu^{1-d} \sigma^{1-\d} \exp\left( C\sigma\right) G_{\mu, \sigma}(t) G_{\mu, \sigma}(s) \\ & \leq C (s+t)^{d-1}\mu^{2-2d} \sigma^{1-d} \exp(C\sigma) G_{\mu, \sigma}(t) G_{\mu, \sigma}(s).
\end{align*}
We obtain the lemma after taking the logarithm of both sides of this expression, dividing by~$-\sigma$, rearranging the resulting the expression  and then estimating a logarithm term in a similar way as near the end of the proof of Lemma~\ref{gooda}.
\end{proof}

We next use Lemma~\ref{Hamm} to obtain a rate of convergence for the means $t^{-1} \E \left[ m_\mu(H_t,0,\cdot) \right]$ to their limit $\overline m_\mu(H_t)$.

\begin{lem}\label{EHtrate}
There exists $C>0$ such that, for every $t> 1$,
\begin{equation} \label{EHtrateq}
\E \left[ m_\mu( H_t,0,\cdot) \right] \leq \overline m_\mu(H_t) + C \left(\frac{t}{\mu^2}\log \left(1+\frac t\mu\right)\right)^{\frac12}.
\end{equation}
\end{lem}
\begin{proof}
According to Lemma~\ref{superadd}, the quantity $g_{\mu,\sigma}$ is almost superadditive. More precisely, for all $s,t> 0$, we have
\begin{equation}\label{ineq:hammer}
g_{\mu,\sigma}(s+t)\geq g_{\mu,\sigma}(s)+g_{\mu,\sigma}(t) - \Delta_{\mu,\sigma}(s+t),
\end{equation}
where, taking $k$ to be the constant $C$ in~\eqref{superaddeq} and
\begin{equation*}
\Delta_{\mu,\sigma}(t) :=  \frac {k}\sigma \left( 1+ \log\left(1+\frac{t}{\sigma\mu}\right) \right).
\end{equation*}
Since $\Delta_{\mu,\sigma}$ is increasing on $[1,\infty)$ and
\begin{equation*}\label{}
\int_1^\infty \frac{\Delta_{\mu,\sigma}(t)}{t^2}\, dt < \infty,
\end{equation*}
we may apply Lemma~\ref{Hamm} to deduce that $\overline g_{\mu,\sigma} : = \lim_{t\to \infty} g_{\mu,\sigma}(t)/t$ exists and,  for every $t> 1$,
\begin{equation}\label{slhamm}
\frac{1}{t}g_{\mu,\sigma}(t) - 4 \int_{2t}^\infty \frac{\Delta_{\mu,\sigma}(s)}{s^2}\, ds \leq \overline g_{\mu,\sigma}.
\end{equation}
An easy integration by parts yields
\begin{equation}\label{slchampignon}
4\int_{2t}^\infty \frac{\Delta_{\mu,\sigma}(s)}{s^2}\, ds \leq \frac{C}{\sigma t} \left( 1+ \log\left( 1+ \frac{t}{\sigma\mu} \right) \right).
\end{equation}
In view of the second inequality in \eqref{goodappr}, we also have 
\begin{equation}\label{berthilion}
\overline g_{\mu,\sigma} \leq \liminf_{t\to \infty}\frac{1}{t} \left( \E\left[ m_\mu(H_t,0,\cdot) \right] + C\right) = \liminf_{t\to \infty}\frac{1}{t}  \E\left[ m_\mu(H_t,0,\cdot) \right].
\end{equation}
We next claim that
\begin{equation} \label{caramelaubeurresale}
\lim_{t\to \infty} \frac1t  \E\left[ m_\mu(H_t,0,\cdot) \right] = \overline m_\mu(H_1). 
\end{equation}
To see this, note that $$\frac1t m_\mu(H_t,0,\omega) = \frac1t \inf_{z\in H_1} m_\mu(tz,0,\omega) = \inf_{z\in H_1 \cap B_{(L_\mu/l_\mu)}} \frac{m_\mu(tz,0,\omega)}{t}$$
and, in view of~\eqref{mmuconv} and the fact that $z\mapsto t^{-1} m_\mu(tz,0,\omega)$ is Lipschitz uniformly in $t>0$, we deduce that \begin{equation*} \label{}
\Prob\left[ \limsup_{t\to \infty} \sup_{z\in H_1 \cap B_{(L_\mu/l_\mu)}} \left|\frac{m_\mu(tz,0,\omega)}{t} - \overline m_\mu(z)  \right| = 0 \right] = 1. 
\end{equation*}
We now obtain~\eqref{caramelaubeurresale} from these two lines and the dominated convergence theorem (which applies since $t^{-1} m_\mu(H_t) \leq L_\mu$).

Combining~\eqref{slhamm},~\eqref{slchampignon},~\eqref{berthilion} and~\eqref{caramelaubeurresale}, we obtain
\begin{equation*}
\frac{1}{t} g_{\mu,\sigma}(t)\leq 
\overline m_\mu(H_1) +\frac{C}{\sigma t} \left( 1+ \log\left( 1+ \frac{t}{\sigma\mu} \right) \right).
\end{equation*}
Multiplying by $t$, applying the first inequality in \eqref{goodappr} and using the homogeneity of $\overline m_\mu$ yields
\begin{equation*}
\E \left[ m_\mu(H_t,0,\cdot) \right] \leq \overline m_\mu(H_t) + C\left( \frac{\sigma t}{\mu^2} + \frac1\sigma+ \frac1{\sigma} \log\left( 1+\frac t{\sigma\mu} \right) \right),
\end{equation*}
and choosing $\sigma:= \mu t^{-\frac12}(\log(1+t/\mu))^{\frac12}$ completes the proof. 
\end{proof}

\subsection{Error estimates for $M_\mu(y)-\overline m_\mu(y)$ and the proof of~\eqref{hard}}
It is the rate of convergence of $t^{-1}M_\mu(ty)$ to $\overline m_\mu(y)$ that we wish to estimate, not that of $t^{-1} \E \left[ m_\mu(H_t,0,\cdot) \right]$ to $\overline m_\mu(H_1)$. In order to reach our desired goal, we must compare the quantities $t^{-1} M_\mu(ty)$ and $t^{-1}\E \left[ m_\mu(H_t,0,\cdot) \right]$. This is accomplished in two steps. The first is to show that $\E\left[ m_\mu(H_t,0,\cdot) \right]$ is very close to
\begin{equation*}\label{}
M_\mu(H_t):= \min_{y\in H_t} M_\mu(y).
\end{equation*}
This yields an estimate for the difference between $M_\mu(H_t)$ and $\overline m_\mu(H_t)$. The second step is to use elementary convex geometry to relate $M_\mu(y)$ to the values of $M_\mu(H)$ for all the possible planes $H$ passing through $y$.

\begin{lem} \label{cmte-minE}
There exists $C> 0$ such that, for every $t>1$,
\begin{equation}\label{wrestle}
 M_\mu(H_t) \leq \E \left[ m_\mu(H_t,0,\cdot) \right] + C\left( \frac{t}{\mu^2}\log \left(1+\frac t\mu \right)\right)^\frac12.
\end{equation}
\end{lem}
\begin{proof}
Let $R:= (L_\mu/l_\mu)t$. For every $\omega\in \Omega$, there exists $z\in H_t \cap B_R$ such that $m_\mu(z,0,\omega) = m_\mu(H_t,0,\omega)$. Hence there exists $\hat z \in \widehat H_t \cap B_R$ such that
\begin{equation*}\label{}
m_\mu(\hat z,0,\omega) \leq m_\mu(H_t,0,\omega) + L_\mu (\d-1)^{\frac12}. 
\end{equation*}
For every $z\in H_t$ we have $\E \left[ m_\mu(z,0,\cdot) \right] =M_\mu(z) \geq M_\mu(H_t)$ and thus, for every $\lambda >0$,
\begin{multline}\label{}
\left\{ \omega\in \Omega\,:\, M_\mu(H_t) - m_\mu(H_t,0,\omega) \geq \lambda + L_\mu (\d-1)^{\frac12} \right\} \\
\subseteq \bigcup_{z\in \hat H_t\cap B_R} \left\{ \omega\in \Omega\,:\, m_\mu(z,0,\omega) \leq M_\mu(z) - \lambda \right\}.
\end{multline}
Applying~\eqref{oscbnd} and using $R=(L_\mu/l_\mu)t\leq Ct/\mu$, we find
\begin{multline}\label{probbnd}
\Prob\left[ M_\mu(H_t) - m_\mu(H_t,0,\cdot) \geq \lambda + L_\mu(\d-1)^{\frac12} \right]  \leq C R^{\d-1} \max_{z\in H_t} \Prob\left[ m_\mu(z,0,\cdot) - M_\mu(z) \leq -\lambda \right] \\
\leq CR^{\d-1} \exp\left( -\frac{\mu\lambda^2}{CR} \right) \leq C \mu^{1-\d} t^{\d-1}\exp\left( -\frac{\mu^2\lambda^2}{Ct} \right).
\end{multline}
We wish to use the expression
\begin{equation}\label{spread}
M_\mu(H_t) - \E\left[ m_\mu(H_t,0,\cdot) \right]  \leq \int_0^\infty \Prob\left[ M_\mu(H_t) - m_\mu(H_t,0,\cdot) \geq \lambda  \right]\, d\lambda
\end{equation}
and then apply~\eqref{probbnd} to the right side of~\eqref{spread}, but due to the factor $t^{\d-1}$ on the right side of~\eqref{probbnd}, this bound is not very helpful unless $\lambda$ is large relative to $t$. With this in mind we fix $A> 1$, to be selected below, define 
\begin{equation*}\label{}
\lambda_1:= \left( \frac{At}{\mu^2} \log\left( 1  +\frac t\mu\right) \right)^{\frac12}
\end{equation*}
and then estimate the right side of~\eqref{spread} by 
\begin{align*}
\lefteqn{ \int_0^\infty \Prob\left[ M_\mu(H_t) - m_\mu(H_t,0,\cdot) \geq \lambda  \right]\, d\lambda } \qquad \qquad \\ 
& \leq  \lambda_1 + L_\mu(\d-1)^{\frac12} + \int_{\lambda_1}^\infty \Prob\left[ M_\mu(H_t) - m_\mu(H_t,0,\cdot) \geq \lambda + L_\mu (\d-1)^{\frac12} \right]\, d\lambda \\
& \leq \lambda_1 + L_\mu(\d-1)^{\frac12} + C\mu^{1-\d}t^{\d-1}\int_{\lambda_1}^\infty \exp\left(-\frac{\mu^2\lambda^2}{Ct}\right) \, d\lambda.
\end{align*}
Observe that
\begin{multline*}
\mu^{1-\d}t^{\d-1}\int_{\lambda_1}^\infty \exp\left(-\frac{\mu^2\lambda^2}{Ct}\right) \, d\lambda \leq \mu^{1-\d}t^{\d-1} \int_{\lambda_1}^\infty \exp\left( - \frac{\mu^2 \lambda_1\lambda}{Ct} \right) \, d\lambda \\
= C \mu^{1-\d}t^{\d-1} \frac{t}{\mu^2\lambda_1} \exp\left( -\frac{\mu^2\lambda_1^2}{Ct} \right) \leq  C \frac {t^d}{\mu^{d+1}} \left(1+\frac t{\mu^2}\right)^{-\frac AC}.
\end{multline*}
By selecting $A$ to be a large enough constant, the last expression on the right is at most $C$. Combining the last two sets of inequalities with~\eqref{spread}, we obtain 
\begin{equation}\label{}
M_\mu(H_t) - \E\left[ m_\mu(H_t,0,\cdot) \right]  \leq \lambda_1 + L_\mu(\d-1)^\frac12 + C \leq \lambda_1+ C,
\end{equation}
which implies~\eqref{wrestle}.
\end{proof}

Lemmas~\ref{EHtrate} and~\ref{cmte-minE} give an estimate on the difference of $M_\mu(H_t)$ and $\overline m_\mu(H_t)$.
\begin{cor}
There exists $C> 0$ such that, for every $t>1$,
\begin{equation}\label{MHtrate}
M_\mu(H_t) \leq \overline m_\mu(H_t) +C\left(\frac t {\mu^2} \log \left(1+\frac t\mu\right)\right)^{\frac12}.
\end{equation}
\end{cor}

The relationship between $M_\mu(H_t)$ and $M_\mu(y)$ depends on the following geometric lemma.

\begin{lem}\label{CHlem}
There exists $C> 0$ such that, for every $N\in \N^*$ and $\alpha > 0$,
\begin{equation}\label{convhull}
 \conv \left\{ y\in \Rd \, : \, M_\mu(y) \leq \alpha \right\} \subseteq \left\{ y \in \Rd \,: \, M_\mu(Ny) \leq (N+C/\mu)\alpha \right\}.
\end{equation}
\end{lem}
\begin{proof}
Let $\alpha>0$, $N\in \N^*$ and select $y\in \conv \{ y \in \Rd \,:\, M_\mu(y) \leq \alpha\}$. According to Carath\'eodory's theorem (see for example~\cite{Gb}), there exist $y_1,\ldots,y_{\d+1}$ and $\lambda_1,\ldots,\lambda_{\d+1} \in [0,1]$ such that 
\begin{equation*}\label{}
y= \sum_{j=1}^{\d+1} \lambda_j y_j, \qquad 1 = \sum_{j=1}^{\d+1} \lambda_j \qquad \mbox{and} \qquad M_\mu(y_j) \leq \alpha \ \ \mbox{for every} \ \ 1\leq j\leq \d+1.
\end{equation*}
For each $1\leq j \leq \d+1$, select $\sigma_j \in\Q$ such that $\sigma_jN \in \N$ and $0 \leq \lambda_j - \sigma_j \leq 1/N$, define
$z: =\sum_{j=1}^{\d+1} \sigma_j y_j$ and observe that, by~\eqref{control},
\begin{equation*}\label{}
|y-z| \leq \frac{\d+1}{N} \max_{1\leq j \leq N} |y_j| \leq \frac{(\d+1)\alpha}{N l_\mu}
\end{equation*}
and, by~\eqref{subadd},~\eqref{control} and~\eqref{cLlL},
\begin{equation} \label{bing}
M_\mu(Ny) \leq M_\mu(Nz) + L_\mu N |z-y| \leq M_\mu(Nz) + \left( \frac{(\d+1) L_\mu}{l_\mu} \right) \alpha \leq M_\mu(Nz) + \frac{C\alpha}{\mu}.
\end{equation}
Notice that, since $N\sigma_j \in \N$ and $Nz = \sum_{j=1}^{\d+1} (N\sigma_j)y_j$, we may apply~\eqref{subaddM} to deduce that
\begin{equation}\label{bang}
M_\mu(Nz) \leq \sum_{j=1}^{\d+1} (N\sigma_j)M_\mu(y_j) \leq N\alpha \sum_{j=1}^{\d+1} \sigma_j \leq N\alpha \sum_{j=1}^{\d+1} \lambda_j = N\alpha.
\end{equation}
Combining~\eqref{bing} and~\eqref{bang} yields the lemma.
\end{proof}

The previous lemma and~\eqref{MHtrate} yield a rate of convergence for $M_\mu(y)$ to $\overline m_\mu(y)$. 

\begin{proof}[{\bf Proof of Proposition~\ref{means}}]
The first step is to show that, for every $z\in \Rd$ such that $|z| > 1$,
\begin{equation}\label{zinCH} 
z \in \conv\left\{ y\in \Rd \,:\, M_\mu(y) \leq \overline m_\mu(z) + k\left(\frac{|z|}{\mu^2}\log \left(1+\frac{|z|}{\mu}\right)\right)^{\frac12} \right\},
\end{equation}
where $k>C$ where $C$ is as in~\eqref{MHtrate}. Suppose on the contrary that ~\eqref{zinCH} fails for some $z\in \Rd$ with $t:=|z|>1$. By elementary convex separation, there exists a plane $H$  with $z\in H$ such that 
\begin{equation*}\label{}
M_\mu(H) > \overline m_\mu(z) + A \quad \mbox{where} \quad A:= k \left(\frac t{\mu^2} \log \left(1+\frac t\mu\right)\right)^{\frac12}.
\end{equation*}
Since $H$ is at most a distance of $|z|=t$ from the origin, we may assume with no loss of generality that $H=H_s$ for some $s\leq t$. We deduce that
\begin{equation}\label{e.Sgrosse}
M_\mu(H_s) > \overline m_\mu(z) + A  \geq \overline m_\mu(H_s) +  k\left(\frac s{\mu^2} \log \left(1+\frac s\mu\right)\right)^{\frac12}.
\end{equation}
Using $\overline m_\mu(H_s) \geq 0$, $M_\mu(H_s) \leq L_\mu s$ and $\mu \leq K$, we see that by making $k$ larger, if necessary, we may deduce that $s>1$. Now~\eqref{e.Sgrosse} contradicts~\eqref{MHtrate}. We have proved~\eqref{zinCH}. 

We now fix $|y|>1$ and proceed with the demonstration of~\eqref{meanseq}. Note that we may assume $|y| \geq 1+\mu^{-1}$, since otherwise~\eqref{meanseq} follows for a suitable $C>0$ from $|y|>1$, $M_\mu(y) \leq L_\mu|y|$ and $m_\mu (y)\geq0$. Now apply~\eqref{zinCH} to $z:=y/N$, where $N \in\N^*$ is chosen below such that $N \leq |y|$, to obtain
\begin{equation*}\label{}
y/N \in \conv\left\{ x\in \Rd \,:\, M_\mu(x) \leq \overline m_\mu(y/N) + C\left( \frac{|y|}{N\mu^2} \log\left( 1+\frac{|y|}{N\mu}\right)\right)^{\frac12}\right\}
\end{equation*}
and, after an application of~\eqref{convhull},
\begin{align*}\label{}
M_\mu(y) & \leq  \left( N + \frac{C}{\mu} \right) \left( \overline m_\mu(y/N) + C \left( \frac{|y|}{N\mu^2} \log\left(1+\frac{|y|}{N\mu}\right)\right)^\frac12\right) \\
& \leq \overline m_\mu(y) + CL_\mu \frac{|y|}{N\mu} + 
C   \left( \frac {N|y|}{\mu^2}\log\left(1+\frac{|y|}{N\mu}\right)\right)^\frac12 + C  \left( \frac{|y|}{N\mu^4} \log\left(1+\frac{|y|}{N\mu}\right)\right)^\frac12. 
\end{align*}
Now we optimize $N$ to obtain~\eqref{meanseq}. In the case that $|y|^{-1}\leq \mu^3$, then we let $N$ be the smallest integer larger than $|y|^{\frac13}$ to get, after using that $|y| > 1$ and $\mu \leq K$,  that
\begin{equation}\label{}
M_\mu(y) \leq \overline m_\mu(y) + C \frac{|y|^{\frac23}}{\mu} \left( \log\left(1+\frac{|y|}{\mu}\right)\right)^\frac12.
\end{equation}
If, on the other hand, $0<\mu^3 < |y|^{-1}$, then we take $N$ be the smallest integer larger than $\mu^{-1}$ and find that
\begin{equation}\label{}
M_\mu(y) \leq \overline m_\mu(y) + C\frac{|y|^{\frac12}}{\mu^{\frac32}} \left( \log\left(1+|y|\right)\right)^\frac12.
\end{equation}
Note that in either case we have~\eqref{meanseq} and we have chosen $N$ so that $N\leq (1+\mu^{-1}) \vee |y| \leq |y|$, as required.
\end{proof}

\subsection{Some further error estimates}
We conclude this section with versions of~\eqref{easy} and~\eqref{hard} which hold uniformly for $y\in B_R$. These estimates, which are needed in the next section, follow from Theorem~\ref{mpEE} and a simple covering argument. 

\begin{lem} \label{mpEER}
There exists $C> 0$ such that, for every $\lambda \geq 4L_\mu$ and $R\geq 3$,
\begin{equation}\label{easyR}
\Prob\left[ \inf_{y\in B_R} \left( m_\mu(y,0,\cdot) - \overline m_\mu(y) \right) \leq -\lambda \right] \leq C R^\d \exp\left( -\frac{\mu\lambda^2}{CR}\right),
\end{equation}
and, if
\begin{equation}\label{lamb-condR}
\lambda \geq C\left( \frac{R^{\frac12}}{\mu^{\frac32}} + \frac{R^\frac23}{\mu} \right) \left( \log\left(1+\frac{R}{\mu}\right)\right)^\frac12,
\end{equation}
then
\begin{equation}\label{hardR}
\Prob\left[ \sup_{y\in B_R} \left( m_\mu(y,0,\cdot) - \overline m_\mu(y) \right) \geq \lambda \right] \leq C R^\d \exp\left(- \frac{\mu \lambda^2}{CR}\right).
\end{equation}
\end{lem}
\begin{proof}
We may select $y_1,\ldots,y_N \in B_R \setminus B_1$ with $N\leq CR^\d$ such that $B_R$ is covered by the balls $B(y_j,2)$. Then by~\eqref{lips} we have, for any $\lambda > 0$,
\begin{equation*}
\Prob\left[ \inf_{y\in B_R} \left( m_\mu(y,0,\cdot) - \overline m_\mu(y) \right) \leq -\frac \lambda2 -2L_\mu  \right] \leq \sum_{j=1}^N \Prob\left[ m_\mu(y_j,0,\cdot) - \overline m_\mu(y_j) \leq -\frac\lambda2 \right].
\end{equation*}
According to~\eqref{easy}, for each $1\leq j \leq N$,
\begin{equation*}\label{}
\Prob\left[ m_\mu(y_j,0,\cdot) - \overline m_\mu(y_j) \leq -\frac\lambda2 \right] \leq \exp\left( - \frac{\mu\lambda^2}{C|y_j|}\right) \leq \exp\left( - \frac{\mu\lambda^2}{CR}\right).
\end{equation*}
Since $N\leq CR^\d$, we obtain, for every $\lambda \geq 4L_\mu$,
\begin{equation*}\label{}
\Prob\left[ \inf_{y\in B_R} \left( m_\mu(y,0,\cdot) - \overline m_\mu(y) \right) \leq -\lambda \right] \leq CR^\d \exp\left( - \frac{\mu\lambda^2}{CR}\right).
\end{equation*}

The estimate~\eqref{hardR} is obtained in a very similar way from Theorem~\ref{mpEE} and a covering argument. We omit the proof. 
\end{proof}

\section{Error estimates for the approximate cell problem} \label{EEdvd}

Here we obtain estimates on the difference between $-\delta v^\delta(y,\omega\,;p)$ and $\overline H(p)$, study the rate for the almost sure convergence
\begin{equation}\label{limencore}
\lim_{\delta\to 0} \sup_{y\in B_{R/\delta}} \left| \delta v^\delta(y,\omega\,;p) + \overline H(p) \right| = 0
\end{equation}
and prove Theorems~\ref{acpEE} and~\ref{acpCR}. One difficulty arises from the fact that  the rate for the approximate cell problem may be very different depending on whether or not $p$ belongs to the interior of the flat spot $\{ \overline H = 0\}$. Recall that the flat spot  is never empty since, e.g.,  $\bar H(0)=0$ (see Appendix~\ref{appmappp}). Moreover, the flat spot $\{ \overline H = 0\}$ is \emph{not} in general equal to $\{ 0\}$ and may indeed have nonempty interior (see, e.g.,~\cite{ASo1}). We use the metric problem to control the $-\delta v^\delta$'s from above, and from below for $p$'s away from the flat spot $\{ \overline H = 0\}$. To obtain the upper bound on the flat spot we study directly the behavior of the $\delta v^\delta$'s.  

We recall here two important deterministic (i.e., uniform in $\omega\in\Omega$) estimates from Section~\ref{ssacpa}:
\begin{equation}\label{dvdsup}
-\sup_{\Rd \times \Omega} H(p,\cdot) \leq \delta v^\delta(y,\omega\,;p) \leq -\essinf_{\Rd \times \Omega} H(p,\cdot)
\end{equation}
and
\begin{equation}\label{vdlip}
\left| v^\delta(y,\omega\,;p) - v^\delta(z,\omega\,;p) \right| \leq K_p|y-z|,
\end{equation}
where $K_p> 0$ depends only on the assumptions for $H$ and an upper bound for $|p|$. Note that the left and right of~\eqref{dvdsup} are bounded  for bounded $|p|$ by~\eqref{regpx}.

\subsection{The ballistic regime}
We combine the exponential error estimates for the metric problem obtained in the previous section with a comparison argument to obtain estimates on the difference between $-\delta v^\delta(y,\omega\,;p)$ and $\overline H(p)$. The comparison argument, which was introduced in~\cite{ASo3} to prove homogenization, yields an estimate from below for $\delta v^\delta + \overline H(p)$ for all $p\in \Rd$ and from above only for $p$'s away from the flat spot.

In the next two proofs, we work with a fixed $p\in \Rd$ and denote by $C$ and $c$ positive constants which may vary in each occurrence and depend only on an upper bound for $|p|$ and the assumptions for $H$.

\begin{proof}[{\bf Proof of Theorem~\ref{acpEE}(i)}]
We actually prove a more general, deterministic statement: namely that, for every $0 <  \delta \leq \lambda \leq 1$, there exists a fixed constant $R\leq C/\delta$ and a finite set~$K\subseteq \Rd$ consisting of at most $C\delta^{-2d}$ elements (which will be identified in the argument) such that
\begin{multline}\label{upwinded}
\left\{ \omega\in \Omega\, : \, -\delta v^\delta(0,\omega\,;p) \geq \overline H(p) + \lambda \right\} \\ \subseteq \, \bigcup_{z\in K} \left\{ \omega\in\Omega\,:\, \inf_{y\in B(z,R)} \left( m_\mu(z,y,\omega) - \overline m_\mu(z-y) \right)  \leq  -\frac \lambda{8\delta} \right\},
\end{multline}
where $\mu:= \overline H(p) + \lambda/4$. Admitting~\eqref{upwinded} for the moment, let us see how to derive~\eqref{dvdEEabove} as a consequence of it and Theorem~\ref{mpEE} (or more precisely, its corollary, Lemma~\ref{mpEER}). We simply use~\eqref{pres},~\eqref{mmustat}, a union bound and~\eqref{easyR} to estimate the probability of right side of~\eqref{upwinded}, keeping in mind that $\mu=\lambda/4$, $R \leq C/\delta$ and $|K| \leq \delta^{-2d}$. We have:
\begin{align*}
\Prob\left[-\delta v^\delta(0,\cdot\,;p) \geq \overline H(p) + \lambda \right] & \leq C\delta^{-2\d} \ \Prob\left[ \inf_{y\in B_R}\left( m_\mu(y,0,\cdot) - \overline m_\mu(y)  \right)  \leq  -\frac {\lambda}{8\delta}  \right] \\
& \leq C\delta^{-2\d} R^\d \exp\left( -\frac{\mu\lambda^2}{C\delta^2R} \right) \\
& \leq C\delta^{-3\d} \exp\left( -\frac{\lambda^3}{C\delta} \right).
\end{align*}

The proof of~\eqref{upwinded} is by a simple comparison argument. We argue that, if $-\delta v^\delta(0,\omega;p)$ is too large, then we can find some translation of $m_\mu$ which is much too small-- otherwise $v^\delta(\cdot,\omega;p)$ and $y\mapsto m_\mu(y,z,\omega) - p\cdot y$ would touch somewhere, in violation of the comparison principle. 

For the rest of the argument, we fix $\omega\in \Omega$ for which $-\delta v^\delta (0,\omega\,;p) \geq \overline H(p) + \lambda$.

\emph{Step 1.} We prepare $v^\delta$ and $m_\mu$ for the comparison: we subtract a plane of slope $p$ from $m_\mu$ and, since we need to introduce some strictness in order to ensure that the two functions touch each other, we bend~$v^\delta$ by a negligible amount. Consider 
\begin{equation*}\label{}
w(y):= v^\delta(y,\omega\,;p) - v^\delta(0,\omega\,;p) + c\lambda \left(1+|y|^2 \right)^{\frac12} - c\lambda,
\end{equation*}
By~\eqref{reg} and~\eqref{vdlip}, if $0< c<1$ sufficiently small, then $w$ satisfies
\begin{equation}\label{e.beckbleh}
H(p+Dw,y,\omega) \geq -\delta v^\delta(y,\omega\,;p) - \frac14\lambda \quad \mbox{in} \ \Rd. 
\end{equation}
Define $U:= \left\{ y\in \Rd \, : \,  w(y) < \lambda/4\delta \right\}$ and notice that, for every $y\in U$,
\begin{equation*}\label{}
-\delta v^\delta (y,\omega\,;p) \geq -\delta v^\delta(0,\omega\,;p) - \frac14\lambda \geq \overline H(p) + \frac34 \lambda.
\end{equation*}
In particular,
\begin{equation}\label{ostar}
H(p+Dw,y,\omega) \geq \overline H(p) + \frac12 \lambda \quad \mbox{in} \ U.
\end{equation}
According to \eqref{dvdsup}, there exists $y_1\in \Rd$ such that $|y_1| \leq C/\lambda\delta$ and
\begin{equation}\label{minwy1}
w(y_1) = \inf_{y\in \Rd} w(y) \leq w(0) = 0.
\end{equation}
Next we denote $\hat w(y):= w(y) -w(y_1) + c\lambda\left( 1 + |y-y_1|^2 \right)^{\frac12} - c\lambda$ and observe that, by~$|Dw| \leq K_p+1$ and~\eqref{reg}, if $c> 0$ is small enough, then we obtain from~\eqref{ostar} that
\begin{equation}\label{}
H(p+D\hat w,y,\omega) \geq \overline H(p) + \frac13 \lambda \quad \mbox{in} \ U.
\end{equation}
Also notice that $\hat w \geq w$, $\hat w(y_1) = 0$ and $\hat w(y) \geq c\lambda\left( 1 + |y-y_1|^2 \right)^{\frac12} - c\lambda$.  Therefore, if we define $V:=\left\{ y\in \Rd \, : \,  \hat w(y) < \lambda/4\delta \right\}$, then, for some $0 < R \leq C/\delta$, we have
\begin{equation}\label{Vgtus}
V \subseteq U \cap B(y_1,R).
\end{equation}
To prepare $m_\mu$ for the comparison, recall that $\mu= \overline H(p) + \lambda/4$, define 
\begin{equation*} \label{}
\hat m(y):= -m_\mu(y_1,y,\omega) - p\cdot (y-y_1)
\end{equation*}
and observe that, according to~\eqref{flipx},
\begin{equation}\label{}
H(p+D\hat m,y,\omega) \leq \mu= \overline H(p) +\frac14\lambda \quad \mbox{in}  \ \Rd.
\end{equation}

\emph{Step 3.} We compare $\widehat w$ and $\widehat m$ in $V$ and then unwind the consequences. Since $y_1\in V$, an application of Proposition~\ref{comp} gives
\begin{equation}\label{cpsty}
\max_{y\in \partial V} \left( -m_\mu(y_1,y,\omega) + p\cdot(y_1-y) \right) = \max_{\partial V} \left( \hat m - \hat w\right) + \frac{\lambda}{4\delta} \geq \hat m(y_1) - \hat w(y_1) + \frac{\lambda}{4\delta}  = \frac{\lambda}{4\delta}.
\end{equation}
Since $\mu > \overline H(p)$ we see from~\eqref{Hbarformm} that $p\in \partial \overline m_\mu(0)$, in particular, for every $z\in \Rd$,
\begin{equation}\label{mnupz}
\overline m_\mu(z) \geq p \cdot z.
\end{equation}
Using~\eqref{cpsty},~\eqref{mnupz} and the fact that $\partial V \subseteq \bar B(y_1,R)$, we obtain
\begin{equation*}\label{}
\inf_{y\in \bar B(y_1,R)} \left( m_\mu(y_1,y,\omega) - \overline m_\mu(y_1-y) \right)  \leq -\frac \lambda{4\delta}. 
\end{equation*}
In view of~\eqref{lips} and $|y_1| \leq C/\lambda\delta$, we deduce that, for some $c> 0$ small enough, we may ``snap to a grid" to deduce that there exists
\begin{equation*}\label{}
z\in K:=\{ c\lambda k/\delta \, : \, k\in \Zd \} \cap B_{C/\lambda\delta},
\end{equation*}
such that 
\begin{equation*}\label{}
\inf_{y\in B(z,R)} \left( m_\mu(z,y,\omega) - \overline m_\mu(z-y) \right)  \leq  -\frac \lambda{8\delta}.
\end{equation*}
Note that $K$ has $C\lambda^{-2d} \leq C\delta^{-2d}$ elements. This completes the proof of~\eqref{upwinded}.
\end{proof}

\begin{proof}[{\bf Proof of Theorem~\ref{acpEE}(ii)}]
The argument is similar to the proof of Theorem~\ref{acpEE}(i) above, but the two are not completely analogous and the details here are a bit more complicated. In particular, it is here that we need the existence of $|e|=1$ satisfying~\eqref{suppe}.

To setup the argument, let $0<\delta\leq 1$ and $\lambda > 0$ such that~\eqref{lamb-condout2} holds. Set $\mu:=\overline H(p)$. Since $\mu > 0$ by assumption, there exists $e\in \Rd$ with $|e|=1$ such that~\eqref{suppe} holds. 

The deterministic statement we prove is this: there exists $R\leq C/\delta$ and a finite set $K \subseteq \Rd$ with at most $C\lambda^{-2d}$ elements such that 
\begin{equation}\label{e.donckers}
\left\{ \omega\in \Omega \, : \, -\delta v^\delta (0,\omega\,;p) \leq \overline H(p) - \lambda \right\} \\
\subseteq \bigcup_{z\in K} \left( E_{1}(z) \cup E_2(z) \right)
\end{equation}
where we define the events $E_1(z), E_2(z) \in \mathcal F$ for each $z\in \Rd$ by
\begin{equation*} \label{}
E_1(z) := \left\{ \omega \in \Omega \,:\, m_\mu(z,z-Re,\omega) - \overline m_\mu(Re)   \geq \frac{\lambda}{10\delta}  \right\}
\end{equation*}
and 
\begin{equation*} \label{}
E_2(z) := \left\{ \omega \in \Omega \,:\, \sup_{y\in B(z,R)} \left( - m_\mu(y,z-Re,\omega)  + \overline m_\mu(y-z+Re)  \right) \geq \frac{\lambda}{10\delta} \right\}.
\end{equation*}
Postponing the demonstration of~\eqref{e.donckers}, let us finish the proof of the theorem. Using~\eqref{pres},~\eqref{mmustat}, a union bound and $|K| \leq C\lambda^{-2\d} \leq C\delta^{-2\d}$, we find that
\begin{equation}\label{basec}
\Prob\left[ -\delta v^\delta (0,\cdot\,;p) \geq \overline H(p) + \lambda \right] \leq C\delta^{-2\d} \left( \Prob\left[ E_1(Re) \right] + \Prob\left[ E_2(Re) \right] \right).
\end{equation}
Applying~\eqref{easyR}, we get
\begin{equation*}\label{}
\Prob\left[ E_2(Re) \right] = \Prob\left[ \inf_{y\in B_R} \left( m_\mu(y,0,\omega)  - \overline m_\mu(y) \right) \leq -\frac{\lambda}{8\delta} \right] \leq C R^\d \exp\left( -\frac{\mu\lambda^2}{C\delta^2R} \right) \leq C\delta^{-\d} \exp \left( -\frac{\mu\lambda^2}{C\delta} \right)
\end{equation*}
and, using the assumption~\eqref{lamb-condout2}, we apply~\eqref{hard} to get
\begin{equation*}\label{}
\Prob\left[ E_1(Re) \right] = \Prob\!\left[ m_\mu(Re,0,\cdot) - \overline m_\mu(Re)  \geq \frac{\lambda}{8\delta}\right] \leq \exp\left( - \frac{\mu\lambda^2}{C\delta^2R}\right) \leq  \exp\left( - \frac{\mu\lambda^2}{C\delta}\right).
\end{equation*}
Combining the last two sets of inequalities with~\eqref{basec} yields~\eqref{dvdEEbelow}.

We have left to prove~\eqref{e.donckers}, for which we make a comparison argument. Fix $\omega\in \Omega$ for which $-\delta v^\delta (0,\omega\,;p) \leq \overline H(p) - \lambda$.

\emph{Step 1.} We prepare $v^\delta$ and $m_\mu$ for the comparison. According to~\eqref{reg},~\eqref{vdlip} and Lemma~\ref{convtrick}, if $c> 0$ is chosen sufficiently small, then
\begin{equation*}\label{}
w(y):= v^\delta(y,\omega\,;p) - v^\delta(0,\omega\,;p) - c\lambda |y|
\end{equation*}
satisfies
\begin{equation}\label{e.froutbleh}
H(p+Dw,y,\omega) \leq -\delta v^\delta(y,\omega\,;p) + \frac14\lambda \quad \mbox{in} \ \Rd. 
\end{equation}
(Note that, in contrast to Step 1 in the proof of Theorem~\ref{acpEE}(i) above, we have perturbed $v^\delta$ by a nonsmooth function. Thus, unlike the derivation of~\eqref{e.beckbleh}, the inequality~\eqref{e.froutbleh} does not immediately hold in the viscosity sense. This relies on the level-set convexity of $H$ and explains the appeal to Lemma~\ref{convtrick}.)

Define $U:= \left\{ y\in \Rd \, : \,  w(y) > -\lambda/4\delta \right\}$ and observe that, for every $y\in U$,
\begin{equation*}\label{}
-\delta v^\delta (y,\omega\,;p) \leq -\delta v^\delta(0,\omega\,;p)  + \frac14\lambda \leq \overline H(p) - \frac34 \lambda
\end{equation*}
and, therefore,
\begin{equation*}\label{}
H(p+Dw,y,\omega) \leq \overline H(p) - \frac12 \lambda \quad \mbox{in} \ U.
\end{equation*}
According to \eqref{dvdsup}, there exists $y_2\in \Rd$ such that $|y_2| \leq C/\lambda\delta$ and
\begin{equation}\label{maxwy1}
w(y_2) = \sup_{y\in \Rd} w(y) \geq w(0) = 0.
\end{equation}
By $|Dw| \leq K_p+1$ and~\eqref{reg}, there exists $c> 0$ such that $\hat w (y) : = w(y) - w(y_2) -  c\lambda |y-y_2|$ satisfies
\begin{equation*}\label{}
H(p+D\hat w,y,\omega) \leq \overline H(p) - \frac13 \lambda \quad \mbox{in} \ U.
\end{equation*}
Let $V:= \left\{ y\in \Rd \, : \,  \hat w(y) > -\lambda/4\delta \right\}$ and note that, in light of the fact that $\hat w\leq w$, $\hat w(y_2) = 0$ and $\hat w(y) \leq -c\lambda |y-y_2|$, there exists $0 < R \leq C/\delta$ such that
\begin{equation}\label{Vstuck}
V \subseteq U \cap B(y_2,R).
\end{equation}
Define $\hat m(y):= m_\mu(y,y_2-Re,\omega) - p\cdot y$ and observe that, in view of~\eqref{mpagan},
\begin{equation*}\label{}
H(p+D\hat m,y,\omega) = \mu =\overline H(p) \quad \mbox{in} \ \Rd\setminus \{ y_2-Re\} \supseteq V.
\end{equation*}

\emph{Step 2.} We apply the comparison principle: comparing $\hat w$ to $\hat m$ in $V$ yields the inequality
\begin{equation*}\label{}
\hat w(y_2) - \hat m(y_2) \leq \sup_{y\in V} \left( \hat w(y) - \hat m(y) \right) = \max_{y\in \partial V}  \left( \hat w(y) - \hat m(y) \right).
\end{equation*}
Using~\eqref{maxwy1} and that $\hat w(y_2) = 0$ and $\hat w\equiv -\lambda/4\delta$ on $\partial V$, we obtain
\begin{equation*}\label{}
\max_{y\in \partial V}  \left( m_\mu(y_2,y_2-Re,\omega) - m_\mu(y,y_2-Re,\omega) - p\cdot(y_2-y) \right)  = \max_{y\in \partial V} \left( \hat m(y_2) - \hat m(y) \right) \geq \frac{\lambda}{4\delta}.
\end{equation*}
We next split the left side of the above inequality into two pieces, one of which must be at least half of the right side. Recalling that $e$ has been chosen so that~\eqref{suppe} holds, we deduce that either
\begin{equation}\label{option1}
\frac\lambda{8\delta}  \leq   m_\mu(y_2,y_2-Re,\omega) - p\cdot (Re) =m_\mu(y_2,y_2-Re,\omega) - \overline m_\mu(Re)  
\end{equation}
or else
\begin{multline}\label{option2}
\frac\lambda{8\delta}  \leq \max_{y\in \partial V} \left( - m_\mu(y,y_2-Re,\omega) + p\cdot (y-y_2+Re) \right) \\ \leq \max_{y\in \partial V} \left( - m_\mu(y,y_2-Re,\omega)  + \overline m_\mu(y-y_2+Re)  \right).
\end{multline}
Using~\eqref{lips},~\eqref{Vstuck} and that $|y_2| \leq C/\lambda\delta$, we may ``snap to a grid" to obtain that there exists
\begin{equation*}\label{}
z\in K:=\{ c\lambda k/\delta \, : \, k\in \Zd \} \cap B_{C/\lambda\delta},
\end{equation*}
such that either
\begin{equation}\label{option1g}
m_\mu(z,z-Re,\omega) - \overline m_\mu(Re)   \geq \frac{\lambda}{10\delta} 
\end{equation}
or else 
\begin{equation}\label{option2g}
\sup_{y\in B(z,R)} \left( - m_\mu(y,z-Re,\omega)  + \overline m_\mu(y-z+Re)  \right) \geq \frac{\lambda}{10\delta}.
\end{equation}
That is, either $\omega\in E_1(z)$ or else $\omega \in E_2(z)$ for some $z\in K$. Note that $|K| = C\lambda^{-2d} \leq C\delta^{-2d}$. 
\end{proof}

A covering argument now yields explicit error estimates for~\eqref{limencore} in balls of radius~$O\left(\delta^{-1}\right)$.

\begin{lem}\label{dvdEE2}
For each $K> 0$, there exist $C,c>0$, depending on $K$ and $H$, such that, for each $p\in B_K$, $R > 0$ and $0 < \delta \leq c$,
\begin{equation}\label{dvdEE2up}
\Prob\left[ \sup_{y\in B_{R/\delta}}-\delta v^\delta (y,\cdot\, ; p) \geq \overline H(p) + C|\log \delta|^{\frac13} \delta^{\frac13} \right] \leq R^\d \delta^2,
\end{equation}
and, if $\overline H(p) > 0$, then
\begin{multline}\label{dvdEE2dnF}
\Prob\left[ \inf_{y\in B_{R/\delta}}-\delta v^\delta (y,\cdot\, ; p) \leq \overline H(p) - C\left( \overline H(p)^{-\frac32} \delta^{\frac12} +  \overline H(p)^{-1} \delta^{\frac13} \right) \left( 1 + |\log \delta| + |\log \overline H(p)| \right)^{\frac12}  \right] \\ \leq R^\d \delta^2.
\end{multline}
\end{lem}
\begin{proof}
Since both of the estimates are obtained from the first two statements of Theorem~\ref{acpEE} using a similar argument, we prove only~\eqref{dvdEE2up}. To do so, we apply~\eqref{dvdEEabove} with $\lambda:= A |\log\delta|^{\frac13} \delta^{\frac13}$, for $A> 0$ chosen sufficiently large, and use a simple covering argument. Notice that if $0 < \delta <\frac12$ is sufficiently small, depending on $A$, then $\delta \leq \lambda$. There exist points $y_1,\ldots,y_N\in B_{R/\delta}$ such that $N \leq CR^\d  \lambda^{-\d}$ and the balls $B(y_j,\lambda/2\delta K_p)$ cover $B_{R/\delta}$. According to~\eqref{dvdstat}, \eqref{vdlip} and~\eqref{dvdEEabove},
\begin{align*}\label{}
\Prob\left[ \sup_{y\in B_{R/\delta}}-\delta v^\delta (y,\cdot\, ; p) \geq \overline H(p) + \lambda  \right] & \leq \Prob\left[ \max_{1\leq j \leq N}-\delta v^\delta (y_j,\cdot\, ; p) \geq \overline H(p) + \frac\lambda2  \right] \\
& \leq CN\Prob\left[ -\delta v^\delta (0,\cdot\, ; p) \geq \overline H(p) + \frac\lambda2  \right] \\
& \leq C R^\d  \lambda^{-\d} \delta^{-3\d} \exp\left( - \frac{\lambda^3}{C\delta}\right) \\
& \leq C R^\d \delta^{-4\d} \exp\left( -\frac{A^3|\log\delta|}{C}\right).
\end{align*}
We therefore obtain~\eqref{dvdEE2up} if we choose $A>0$ so that $A^3 \geq C (4\d+2)$. 
\end{proof}

We next apply~Lemma~\ref{dvdEE2} along a certain subsequence $\delta_n\to 0$ to prove, with the help of~\eqref{dvddepd} and the Borel-Cantelli lemma, the first two statements of Theorem~\ref{acpCR}.

\begin{proof}[{\bf Proof of Theorem~\ref{acpCR}(i) and Theorem~\ref{acpCR}(ii)}]
The arguments for the two statements are almost identical, so we prove only~\eqref{aboveRas}. Let $R> 0$, $\delta_n = n^{-1}$ and apply~\eqref{dvdEE2up} to obtain
\begin{equation*}\label{}
\sum_{n=1}^\infty \, \Prob\left[ \sup_{y\in B_{R/\delta_n}}-\delta_n v^{\delta_n} (y,\cdot\, ; p) \geq \overline H(p) + C |\log \delta_n|^{\frac13} \delta_n^{\frac13} \right] \leq C + R^\d \sum_{n> 1/c_2} \delta_n^2 \leq C+ R^\d \sum_{n=1}^\infty \frac1{n^2} < \infty. 
\end{equation*}
By the Borel-Cantelli lemma, we deduce that, there exists $C>0$ such that, for every $R> 0$,
\begin{equation*}\label{}
\Prob\left[ \limsup_{n\to \infty} \sup_{y\in B_{R/\delta_n}} \frac{-\delta_n v^{\delta_n} (y,\cdot\, ; p) - \overline H(p)}{ C |\log \delta_n|^{\frac13} \delta_n^{\frac13}} \leq 1 \right] = 1.
\end{equation*}
Intersecting these events for each $R=1,2,3,\ldots$, we find an event $\Omega_1$ of full probability such that, for every $R> 0$ and $\omega\in \Omega_1$,
\begin{equation*}\label{}
\limsup_{n\to \infty} \sup_{y\in B_{R/\delta_n}} \frac{-\delta_n v^{\delta_n} (y,\cdot\, ; p) - \overline H(p)}{ C |\log \delta_n|^{\frac13} \delta_n^{\frac13}} \leq 1.
\end{equation*}
Notice that $\delta_{n+1} / \delta_n = 1-\delta_{n+1}$ and so, according to~\eqref{dvddepd}, for any $\delta_{n+1} \leq \eta < \delta_n$ and $\omega\in\Omega$,
\begin{equation*}\label{squisheta}
\sup_{y\in \Rd} \left| {\delta_n} v^{\delta_{n}}(y,\omega\,;p) - {\eta} v^{\eta}(y,\omega\,;p) \right| \leq C \delta_{n+1} \leq C\delta_n.
\end{equation*}
Hence for every $\omega\in \Omega_1$ we have
\begin{align*}
\limsup_{\delta\to 0} \sup_{y\in B_{R/\delta}} \frac{-\delta v^\delta(y,\omega\,;p) - \overline H(p)}{ C |\log \delta|^{\frac13} \delta^{\frac13}} & = \limsup_{n\to \infty} \sup_{\delta_{n}\leq \eta < \delta_{n-1}} \sup_{y\in B_{R/\eta}} \frac{-\eta v^\eta(y,\omega\,;p) - \overline H(p)}{ C |\log \eta|^{\frac13} \eta^{\frac13}} \\
& \leq  \limsup_{n\to \infty} \sup_{y\in B_{2R/\delta_n}} \frac{-\delta_n v^{\delta_n} (y,\omega\, ; p) - \overline H(p) + C\delta_n }{ C|\log \delta_{n}|^{\frac13} \delta_{n}^{\frac13}} 
 \leq 1. \qedhere
\end{align*}
\end{proof}

\subsection{The sub-ballistic regime}

We show that the behavior of $\delta v^\delta (0,\cdot\, ; p)$ for $p$'s on the flat spot $\{ \overline H = 0 \}$ is determined by the distribution of $H(0,0,\cdot)$ near its maximum and that with further (quite reasonable) assumptions on this distribution, we obtain exponential error estimates and an algebraic rate of convergence for $-\delta v^\delta(0,\omega\,;p)$ to $\overline H(p)$.

We begin with the simple observation, which is probably well-known and essentially taken from~\cite{ASo1}, that $-\delta v^\delta(\cdot,\omega\,;p)$ is controlled pointwise from below by $H(0,\cdot,\omega)$.

\begin{lem} \label{dumb}
For every $p\in \Rd$ and $\omega\in \Omega$,
\begin{equation}\label{disdomdn}
-\delta v^\delta(0,\omega\, ;p) \geq \sup_{R>0} \left( \sup_{y\in B_R} H(0,y,\omega) -K_p R \delta \right),
\end{equation}
where the constant $K_p>0$ is defined in~\eqref{Kp}.
\end{lem}
\begin{proof}
Fix $p,y\in \Rd$ and $\omega\in \Omega$. Due to~\eqref{acp} and~\eqref{cntl}, $v^\delta(\cdot,\omega\,;p)$ satisfies
\begin{equation}\label{blague2}
\delta v^\delta + H(0,y,\omega) \leq \delta v^\delta + H(p+Dv^\delta,y,\omega) = 0 \quad \mbox{in} \ \Rd. 
\end{equation}
While this holds in the viscosity sense, there are no derivatives in the expression on the left of~\eqref{blague2}, and we deduce that
\begin{equation}\label{fsub}
-\delta v^\delta(y,\omega\, ; p) \geq H(0,y,\omega). 
\end{equation}
Combining~\eqref{fsub} and \eqref{vdlip} yields \eqref{disdomdn}.
\end{proof}

For $p$'s on the flat spot, the rate of convergence for $-\delta v^\delta(0,\omega\,;p)$ to $0=\overline H(p)$ given by~\eqref{disdomdn} is essentially optimal (see Lemma~\ref{stupid} below). We next exhibit exponential error estimates for $-\delta v^\delta(0,\cdot\,;p)$ under the additional hypothesis~\eqref{plushyp}. 

\begin{prop} \label{dvdEEF}
Assume~\eqref{plushyp}. There exist $C, c>0$ such that, for all $\delta > 0$ and $0 < \lambda \leq c$,
\begin{equation}\label{dvdEEFdn}
\Prob\left[ -\delta v^\delta(0,\cdot\,;p) \leq -\lambda \right]  \leq \exp\left( - \frac{\lambda^{\d+\theta}}{C\delta^\d} \right).
\end{equation}
\end{prop}
\begin{proof}
Fix $p\in \Rd$ and apply~\eqref{disdomdn} with $R:=\lambda/(2K_p\delta)$ to discover that
\begin{equation*}
\Prob\left[ -\delta v^\delta(0,\cdot\,;p) \leq -\lambda \right] \leq \Prob\left[ \sup_{y\in B_R} H(0,y,\cdot) \leq K_pR\delta-\lambda  \right] = \Prob\left[ \sup_{y\in B_R} H(0,y,\cdot) \leq -\frac\lambda 2\right].
\end{equation*}
Select $y_1,\ldots, y_N \in B_R$ such that $N \geq c R^\d$ and, for every $i\neq j$, $|y_i - y_j| > 1$. Then~\eqref{stnary} and~\eqref{indy} yield
\begin{equation*}
\Prob\left[ \sup_{y\in B_R} H(0,y,\cdot) \leq -\frac\lambda 2\right] \leq \Prob\left[ \sup_{1\leq j \leq N} H(0,y_j,\cdot) \leq -\frac\lambda 2\right]  = \prod_{1\leq j \leq N} \Prob\left[ H(0,y_j,\cdot) \leq -\frac\lambda 2\right].
\end{equation*}
Using ~\eqref{plushyp}, we find
\begin{multline*}
\prod_{1\leq j \leq N} \Prob\left[ H(0,y_j,\cdot) \leq -\frac\lambda 2\right] = \left( \Prob\left[ H(0,0,\cdot) \leq -\frac\lambda 2\right] \right)^N  \leq \left(1-c \lambda^\theta \right)^N \\ \leq \exp\left( -cN\lambda^{\theta} \right) \leq \exp\left( -c \lambda^{\d+\theta} \delta^{-\d}\right).
\end{multline*}
Combining the lines above yields~\eqref{dvdEEFdn}. 
\end{proof}

We now obtain, under assumption~\eqref{plushyp}, exponential error estimates and a rate of convergence, from below, for $-\delta v^\delta$ for all $p\in \Rd$. First, we combine~\eqref{dvdEEbelow} and~\eqref{dvdEEFdn} to obtain error estimates independent of $\overline H(p)$. 

\begin{proof}[{\bf Proof of Theorem~\ref{acpEE}(iii)}]
We begin with the observation that there exists $c>0$ such that, for every $0< \delta <c$ is sufficiently small, then
\begin{equation} \label{gumpa}
\delta^{\frac16} | \log\delta|^{\frac1{4}} \geq c\, \sup_{0< \sigma \leq 1} \left( \sigma \wedge \left( \sigma^{-\frac32}\delta^{\frac12} + \sigma^{-1} \delta^{\frac13}\right)\left( 1+ |\log\delta|+ |\log \sigma| \right)^{\frac12} \right).
\end{equation}
To see this, fix $\sigma > 0$ such that $\delta^{\frac16} | \log\delta|^{\frac1{4}} \leq \sigma$ and  observe that if $c>0$ is sufficiently small, then
\begin{align*}
\sigma^{-\frac32} \delta^{\frac12} \left( 1+ |\log\delta|+ |\log \sigma| \right)^{\frac12}  \leq C\delta^{\frac14} \left| \log \delta\right|^{\frac18} < \delta^{\frac16}\left| \log \delta \right|^{\frac14}
\end{align*}
and
\begin{align*}
\sigma^{-1} \delta^{\frac13} \left( 1+ |\log\delta|+ |\log \sigma| \right)^{\frac12}  \leq C \delta^{\frac16} \left| \log \delta \right|^{\frac14}.
\end{align*} 
This completes the proof of~\eqref{gumpa}.

We deduce that, for any $0<\delta < c$ and $\lambda >0$ satisfying~\eqref{lambang} for sufficiently large $C>0$,
\begin{equation}\label{dichotomy}
\lambda \geq \sup_{\sigma > 0} \left( \sigma \wedge C\left( \sigma^{-\frac32}\delta^{\frac12} + \sigma^{-1} \delta^{\frac13}\right)\left( 1+ |\log\delta|+ |\log \sigma| \right)^{\frac12} \right).
\end{equation}
Now,~\eqref{dichotomy} ensures that~\eqref{dvdEEbelow} and~\eqref{dvdEEFdn} overlap in an appropriate way to yield the theorem. Indeed, we may apply~\eqref{dvdEEFdn} in case $\lambda \geq 2\overline H(p)$, while in the case that $\lambda < 2\overline H(p)$, we may apply~\eqref{dvdEEbelow}, since~\eqref{dichotomy} ensures that $\lambda < 2\overline H(p)$ implies~\eqref{lamb-condout2}.
\end{proof}

Arguing in a similar way as in the proof of Lemma~\ref{dvdEE2}, we obtain the following result as an application~\eqref{EEFq}. The details are left to the reader.

\begin{lem}\label{dvEEF}
Assume~\eqref{plushyp} and let $\alpha$ and $\beta$ be defined as in~\eqref{alphbeta}. Then there exist  $C,c> 0$ such that, for each $R > 0$ and $0 < \delta \leq  c$,
\begin{equation}\label{dnF2}
\Prob\left[ \inf_{y\in B_{R/\delta}}-\delta v^\delta (y,\cdot\, ; p) \leq \overline H(p) - C |\log \delta|^{\beta} \delta^{\alpha} \right] \leq R^\d \delta^2.
\end{equation}
\end{lem}

Using~\eqref{dnF2} we complete the proof of Theorem~\ref{acpCR}. 

\begin{proof}[{\bf Proof of Theorem~\ref{acpCR}(iii)}]
The statement follows from~\eqref{dnF2} by an argument very similar to the proof of Theorem~\ref{acpCR}(i) given above. 
\end{proof}

\subsection{The rate may be arbitrarily slow on the flat spot}
\label{doucement}
 As explained in Appendix~\ref{appmappp}, the vector $p=0$ belongs to the flat spot: that is, $\bar H(0)=0$. We can also see this from~\eqref{dvdheureq} by observing that~\eqref{cntl} implies that $v^\delta(\cdot,\omega\,;0)\equiv 0$. We show here that~\eqref{plushyp} is necessary for the existence of an algebraic rate of convergence like~\eqref{dnF2} for the limit~\eqref{limencore} at $p=0$. Furthermore, without some assumption on the distribution of the random variable $H(0,0,\cdot)$ near its maximum, there is no restraint on how slowly the limit~\eqref{limencore} may converge for $p=0$.

We begin by exhibiting an upper bound for $-\delta v^\delta(0,\omega\,;0)$ to match~\eqref{disdomdn}, which shows that, for $p$'s on the flat spot, the rate given in Lemma~\ref{dumb} is essentially optimal.

\begin{lem} \label{stupid}
There exists $C \geq 1$ such that, for every $\omega\in \Omega$ and $R,\delta > 0$,
\begin{equation}\label{disdomup}
-\delta v^\delta (0,\omega\,;0) \leq (- \delta R) \vee \frac{\delta R}{C+\delta R} \, \sup_{y\in B_R} H(0,y,\omega).
\end{equation}
\end{lem}
\begin{proof}
Fix $\omega\in \Omega$, $R, \delta > 0$ and define, for every $0\leq \alpha \leq 1$,
\begin{equation*}\label{}
h_R :=  -\sup_{y \in  B_R} H(0,y,\omega)\quad \mbox{and} \quad h_{\alpha,R} :=  -\sup_{(p,y) \in B_\alpha \times B_R} H(p,y,\omega),
\end{equation*}
where we set $h_{\alpha,R}:= h_R$ in the case that $\alpha=0$. Observe that $h_R\geq 0$ and, due to~\eqref{reg}, for any $0 \leq \alpha \leq 1$,
\begin{equation}\label{har}
|h_R - h_{\alpha,R}| \leq C \alpha.
\end{equation}
Next fix $0 \leq \alpha \leq 1$ to be selected below, define $u(x) := \alpha(R-|x|)_+$ and note that $u \in \Lip$, $|Du| \leq \alpha$ in $\Rd$ and $u \equiv 0$ in $\Rd\setminus B_R$. We conclude that
\begin{equation*}\label{}
\delta u + H(Du,y,\omega) \leq \left( \delta \alpha R - h_{\alpha,R}\right) \vee 0 \quad \mbox{in} \ \Rd.
\end{equation*}
Choose $0 \leq \alpha \leq 1$ to be the largest number for which $\alpha \delta R \leq h_{\alpha,R} \wedge \delta R$. The comparison principle (Proposition~\ref{appcmp}) yields that $\delta\alpha R = \delta u(0) \leq \delta v^\delta(0,\omega\,;0)$. According to~\eqref{har},
\begin{equation*}\label{}
\alpha \geq \frac{h_R}{C+\delta R} \wedge 1. 
\end{equation*}
which concludes the proof. \end{proof}

The previous lemma states that the rate $-\delta v^\delta(0,\omega\,;0)$ converges to $0$ is controlled from below, up to a factor of $2$, by the maximum of $H(0,\cdot,\omega)$ in the ball $B_{C / \delta}$. Using the independence assumption and an easy covering argument, we relate the latter to the distribution of $H(0,0,\omega)$ near its maximum to recover the following estimate.

\begin{lem} \label{H0dbad}
There exists $C > 0$ such that, for every $\delta > 0$ and $0< \lambda \leq 1$, 
\begin{equation}\label{logbad}
\log \Prob\left[ -\delta v^\delta(0,\cdot\,;0) \leq -\lambda \right] \geq C\delta^{-\d}\, \log \left( 1 - C\lambda^{-\d} \,\Prob\left[ H(0,0,\cdot) > -4\lambda \right]\right),
\end{equation}
where the inequality is vacuous if the argument in the logarithm on the right side is negative.
\end{lem}
\begin{proof}
Fix $0 < \lambda \leq 1$. There exists a finite collection of points $\{ y_{ij} \,:\, 1\leq i \leq N, 1 \leq j \leq M\} \subseteq B_R$ such that $N\leq CR^\d$,  $M\leq C\lambda^{-\d}$, $B_R$ is covered by the balls $B(y_{ij},c\lambda)$ and $|y_{ij} - y_{kj}| > D$ if $i\neq k$. According to~\eqref{stnary},~\eqref{indy} and~\eqref{reg}, we have
\begin{multline*}
\Prob \left[ \sup_{y\in B_R} H(0,y,\cdot) \leq -2\lambda \right] \geq \Prob \left[ \sup_{1\leq i \leq N} \sup_{1\leq j \leq M} H(0,y_{ij},\cdot) \leq -4\lambda \right] \\
 = \Prob \left[ \sup_{1 \leq j \leq M} H(0,y_{1j},\cdot) \leq -4\lambda \right]^N \geq \big( 1 - M \Prob \left[ H(0,0,\cdot) > -4\lambda \right]\big)^N.
\end{multline*}
Setting $R:=C /\delta$, with $C\geq 1$ as in Lemma~\ref{stupid}, applying~\eqref{disdomup} and taking the logarithm of the resulting expression yields~\eqref{logbad}.
\end{proof}

We next show that the assumption~\eqref{plushyp} is essentially necessarily for~\eqref{limencore} to have an algebraic rate of convergence at $p=0$. 

\begin{prop} \label{plushyp-roi}
Assume, contrary to~\eqref{plushyp}, that there exists $C> 0$ and $\theta > \d$ such that, for every $0 < \lambda \leq 1/C$,
\begin{equation}\label{conthyp}
\Prob\left[ H(0,0,\cdot) > -\lambda \right] \leq C \lambda^\theta.
\end{equation}
Then there exists $c> 0$ such that, for every $0 < \delta \leq c$,
\begin{equation}\label{badest}
\Prob \left[ -\delta v^\delta(0,\cdot\,;0) \leq -\delta^\gamma \right] \geq c\quad \mbox{for} \quad \gamma:= \d/(\theta-\d). 
\end{equation}
Moreover, for $\omega$ belonging to an event of full probability, we have, for every $\eta > \gamma$,
\begin{equation}\label{tresmauvais}
\liminf_{\delta \to 0} \frac{\delta v^\delta(0,\omega\,;0)}{\delta^\eta} = +\infty.
\end{equation}
\end{prop}
\begin{proof}
Let $\lambda > 0$. Using the elementary inequality $-2t \leq \log(1-t)$ for $0 \leq t \leq \tfrac12$, we apply~\eqref{logbad} and~\eqref{conthyp} to obtain, for sufficiently small $\delta,\lambda> 0$,
\begin{equation*}\label{}
\log \Prob\left[ -\delta v^\delta(0,\cdot\,;0) \leq -\lambda \right] \geq C\delta^{-\d} \log \left( 1 - C\lambda^{\theta-\d} \right) \geq - C\delta^{-\d} \lambda^{\theta-\d}.
\end{equation*}
Setting $\lambda := \delta^\beta$ yields~\eqref{badest} for sufficiently small $\delta > 0$, while on the other hand setting $\lambda = \delta^\alpha$ for $\alpha> \beta$, applying the Borel-Cantelli lemma along a subsequence and arguing as in the proof of Theorem~\ref{acpCR} yields~\eqref{tresmauvais}. 
\end{proof}

It is clear from~\eqref{logbad} that, if $\Prob\left[ H(0,0,\cdot) > -\lambda \right]$ decays to $0$ very quickly as $\lambda \to 0$, then $-\delta v^\delta(0,\omega\,;0)$ will converge to zero very slowly. We conclude this section by exhibiting an example demonstrating an arbitrarily slow rate. 

\begin{example}
Let $\rho : [0,1] \to [0,\infty)$ be a given increasing continuous function with $\rho(0) = 0$. We will construct a Hamiltonian $H$, satisfying~\eqref{assum}, such that
\begin{equation}\label{catestrophe}
\liminf_{\delta \to 0} \frac{\delta v^\delta(0,\omega\,;p)}{\rho(\delta)} \geq 1 \quad \mbox{almost surely.}
\end{equation}
According to~\eqref{logbad}, it suffices to exhibit an $H$ for which
\begin{equation}\label{pasmal}
C \log \left( 1 - C \rho(\delta)^{-\d}\, \Prob\left[ H(0,0,\cdot) > -4\rho(\delta) \right]\right) \geq -\delta^{\d+2}.
\end{equation}
Indeed,~\eqref{logbad} and~\eqref{pasmal} yield
\begin{equation*}\label{}
\Prob\left[ \delta v^\delta(0,\cdot\,;0) \geq \rho(\delta) \right] \geq \exp(-C\delta^2) \geq 1 - C\delta^2
\end{equation*}
for $\delta > 0$ sufficiently small. Then the Borel-Cantelli lemma and an argument similar to the proof of Theorem~\ref{acpCR}, using~\eqref{dvddepd}, give~\eqref{catestrophe}.

It is easy to check that for~\eqref{pasmal}, it suffices to have
\begin{equation}\label{pasmal2}
\Prob\left[ H(0,0,\cdot) > -\sigma  \right] \leq c\sigma^\d \left( \rho^{-1}(\sigma) \right)^{\d+2} =: \hat \rho(\sigma),
\end{equation}
and so we need to construct a Hamiltonian with a very thin distribution near its maximum. This is quite simple. We may take, for example,
\begin{equation*}\label{}
H(p,y,\omega) : = \frac12 |p|^2 - \phi(V(y,\omega))
\end{equation*}
where $V$ is a Poissonian potential~(see for example~\cite{Szb}) and $\phi:[0,\infty) \to (0,\infty)$ is a continuous, decreasing function such that $\phi(t)$ decays very slowly to 0 as $t\to \infty$ (the precise rate of decay required can be explicitly calculated in terms of $\hat \rho$). We leave it to the interested reader to fill in the details.
\end{example}

\subsection{Uniform error estimates for the approximate cell problem}
The proofs of Theorems~\ref{EEH} and~\ref{CRH}, given in the next section, depend on the following extensions of Theorems~\ref{acpEE} and~\ref{acpCR} which hold uniformly for bounded $|p|$ and for $y$ in balls of radius $O\big( \delta^{-N} \big)$, for any $N\geq 1$. We omit the arguments, since the error estimates follow easily from an application of Theorem~\ref{acpEE} combined with~\eqref{dvdstat} and a routine covering argument, and then the convergence rates from the latter using the nearly same argument as in the proof of Theorem~\ref{acpCR}. 

\begin{prop}\label{ballaz}
For each $K > 0$, there exist constants $C,c > 0$, depending on $K$ and $H$, such that, for each $R> 0$, the following hold:
\begin{enumerate}

\item[(i)] For every $0 <  \lambda \leq 1$ and $0 < \delta \leq c\lambda$,
\begin{equation}\label{dvdEEaboveR}
\Prob\left[ \sup_{(y,p) \in B_R\times B_K} \left( -\delta v^\delta(y,\cdot\,;p) - \overline H(p)\right) \geq \lambda \right] \leq CR^\d  \delta^{-5\d} \exp\left( -\frac{\lambda^3}{C\delta} \right).
\end{equation}

\item[(ii)] If~\eqref{plushyp} holds and $\lambda,\delta >0$ satisfy $0< \delta \leq c \lambda$ and
\begin{equation} \label{lambangR}
\lambda \geq C \delta^{\frac16} | \log\delta|^{\frac{1}{4}},
\end{equation}
then
\begin{multline}\label{EEFqR}
\quad \ \ \Prob\left[ \inf_{(y,p) \in B_R\times B_K} \left( -\delta v^\delta (y,\cdot\, ; p) - \overline H(p)\right)  \leq - \lambda \right] \\ \leq C R^\d \delta^{-5\d } \exp\left( - \frac1{C} \left(  \frac{\lambda^3}{\delta} \wedge \frac{\lambda^{\d+\theta}}{\delta^\d}\right) \right).
\end{multline}
\end{enumerate}
\end{prop}

\begin{prop} \label{acpCRU}
There exists an event $\Cl[O]{O-acpCRU}\in \mathcal F$ of full probability such that, for every $\omega\in \Cr{O-acpCRU}$, the following hold:
\begin{enumerate}
\item[(i)] For every $R,K> 0$ and $N\geq 1$,
\begin{equation}\label{aboveRas2}
\limsup_{\delta\to 0}  \sup_{y\in B_{(R/\delta)^N}} \ \sup_{p\in B_K} \frac{-\delta v^\delta(y,\omega\,;p) - \overline H(p)}{ \delta^{\frac13}|\log \delta|^{\frac13} } < \infty.
\end{equation}

\item[(ii)] If~\eqref{plushyp} holds and we define $\alpha$ and $\beta$ as in~\eqref{alphbeta}, then, for every $R,K> 0$ and $N\geq 1$,
\begin{equation}\label{belowRasF3}
\liminf_{\delta\to 0} \inf_{y\in B_{(R/\delta)^N}} \ \sup_{p\in B_K} \frac{-\delta v^\delta(y,\omega\,;p) - \overline H(p)}{ \delta^{\alpha} |\log \delta|^{\beta} } >  -\infty.
\end{equation}
\end{enumerate}

\end{prop}

\section{Error estimates for homogenization}  \label{EEue}
 
We now present the proofs of Theorems~\ref{EEH} and~\ref{CRH}. The main step is to precisely quantify how the $\delta v^\delta$'s control the $u^\ep$'s, so that we may apply the results of the previous section to obtain error estimates and a rate of convergence for the latter.

For each $\ep,T > 0$, let $u^\ep=u^\ep(\cdot,\cdot,\omega), u=u(\cdot,\cdot) \in \BUC(\Rd\times [0,T])$ be the unique solutions of 
\begin{equation*}
\left\{ 
\begin{aligned}
& u^\ep_t + H\left(Du^\ep,\frac x\ep,\omega \right) = 0 & \mbox{in} & \ \Rd\times(0,\infty), \\
& u^\ep(\cdot,0,\omega) =   u_0,
\end{aligned} 
\right. \qquad \mbox{and} \qquad \left\{ 
\begin{aligned}
& u_t + \overline H\left(Du\right) = 0 & \mbox{in} & \ \Rd\times(0,\infty), \\
& u(\cdot,0) = u_0,
\end{aligned}
\right.
\end{equation*}
where $u_0\in C^{0,1}(\Rd)$ is the given initial data with $\| u_0\|_{C^{0,1}(\Rd)} \leq K$. 

It turns out (see, e.g.,~\cite{Ba}) that there exists a constant $L>0$, depending on $K$ and the assumptions for $H$, such that, for all $\ep > 0$, $x,y\in \Rd$, $s,t\geq 0$ and $\omega\in \Omega$,
\begin{equation}\label{ueplip}
 |u^\ep(x,t,\omega) - u^\ep(y,s,\omega)| \leq L\left( |x-y| + |s-t| \right).
\end{equation}
and 
\begin{equation}\label{ulip}
|u(x,t) - u(y,s)| \leq L\left( |x-y| + |s-t| \right).
\end{equation}
These estimates are derived principally from the coercivity of $H$. Recall that, due to~\eqref{dvdsup} and~\eqref{dvdheureq}, the effective Hamiltonian shares the same rate of coercivity assumed  in~\eqref{coer}. It also follows easily from this that for each $\ep > 0$, $x\in \Rd$, $0 \leq t\leq T$ and $\omega\in \Omega$,
\begin{equation}\label{uesup}
|u(x,t)| + |u^\ep(x,t,\omega)| \leq K + LT \leq C(1+T).
\end{equation}

The important link between the $\delta v^\delta$'s and the $u^\ep$'s is summarized in the following lemma. Then Theorems~\ref{EEH} and~\ref{CRH} follow relatively easily from it and Propositions~\ref{ballaz} and~\ref{acpCRU}. The basic idea is that the event that  $|u^\ep(x,t,\omega) - u(x,t)|$ is large should  only be observed if $|\delta v^\delta + \overline H(p)|$ is also large. The proof, which is rather technical and lengthy,  follows along the lines of~\cite{ICD} with necessary modifications to deal with to the lack of uniform estimates on the difference between $-\delta v^\delta$ and $\overline H(p)$.
It essentially consists of quantifying the perturbed test function method~\cite{E2} to argue that, if $-\delta v^\delta$ is close to $\overline H(p)$, then, up to an appropriate error, it properly captures the oscillations of $u^\ep$, that is, 
\begin{equation}\label{}
u(x,t) \approx u^\ep(x,t,\omega) - \ep v^\delta\left( \frac x\ep,\omega \,; Du(x,t) \right),\qquad \ep \ll \delta.
\end{equation}
Rather than apply the comparison principle, we must use the proof of it, following~\cite{ICD}. The difficulty is that, since~$u$ is not in general~$C^1$, we cannot insert~$p=Du(x,t)$ into~$v^\delta(x,\omega\,;p)$. There are other technical difficulties (the presence of \emph{three} nonsmooth functions, the fact that~$v^\delta$ is not smooth in~$p$) which we handle by the standard viscosity theoretic technique of doubling (or rather tripling) the variables. 

Throughout this section, we fix $K>0$, assume $H$ satisfies~\eqref{assum}, and let $C$ and $c$ denote positive constants which may vary in each occurrence and depend only on $K$ and $H$.

\begin{lem}\label{incluHH}
There exists $C > 1$ such that, for every $\lambda, T, \ep, \delta > 0$ satisfying
\begin{equation}\label{HHconds}
\lambda \leq 1, \quad T\geq 1, \quad \ep \leq \lambda T \quad \mbox{and}  \quad \frac{ C \ep}{ T\lambda^2} \leq \delta 
\end{equation}
and each initial datum $u_0 \in C^{0,1}(\Rd)$ with $\| u_0 \|_{C^{0,1}(\Rd)} \leq K$, we have
\begin{multline} \label{incluHH-1}
\left\{ \omega\in \Omega \, : \, \inf_{x\in B_T} \inf_{0 < t \leq T} \left( u^\ep(x,t,\omega) - u(x,t) \right) <-  C \lambda T \right\} \\
\subseteq \left\{ \omega\in \Omega \, : \,  \sup_{|y| \leq C\delta \lambda T^2/\ep^2  } \ \sup_{|p| \leq C} \left( -\delta v^\delta(y,\omega\,;p) - \overline H(p) \right) \geq \lambda  \right\},
\end{multline}
and
\begin{multline}\label{incluHH-2}
\left\{ \omega \in \Omega \, : \, \sup_{x\in B_T} \sup_{0 \leq t \leq T} \left( u^\ep(x,t,\omega) - u(x,t) \right) > C \lambda T  \right\} \\ \subseteq \left\{ \omega \in \Omega \, : \, \inf_{|y| \leq C\delta \lambda T^2/\ep^2  } \ \inf_{|p| \leq C} \left( -\delta v^\delta (y,\omega\,;p) - \overline H(p)\right)\leq - \lambda \right\}.
\end{multline}

\end{lem}
\begin{proof}
Since the arguments for~\eqref{incluHH-1} and~\eqref{incluHH-2} are nearly identical, we prove only~\eqref{incluHH-2}. 

\emph{Step 1.} We setup the argument. With $L>0$ as in~\eqref{ueplip} and~\eqref{ulip}, define $\zeta:\Rd\to\Rd$ by
\begin{equation*}\label{}
\zeta(x):= \frac{L\wedge |x|}{|x|} x,
\end{equation*}
and notice that $\zeta(x)$ has the same direction as $x$ and
\begin{equation} \label{ss}
|\zeta(x)| = L\wedge |x| \quad \mbox{and} \quad |\zeta(x) - \zeta(y)| \leq |x-y|.
\end{equation}
Fix $T, \lambda, \ep, \delta > 0$ satisfying
\begin{equation}\label{HHconds2}
\lambda \leq 1, \quad T\geq 1, \quad \ep \leq \lambda T \quad \mbox{and}  \quad \frac{A \ep}{ T\lambda^2} \leq \delta 
\end{equation}
where the constant $A \geq 2$ is selected below. Also fix parameters $\alpha,\gamma > 0$ to be chosen below and consider the auxiliary function $\Phi:\Rd \times \Rd \times [0,T] \times [0,T] \times \Omega \to \R$ given by
\begin{multline}\label{auxfun}
\Phi(x,y,t,s,\omega):= u^\ep ( x,t,\omega) - u(y,s) - \ep v^\delta \left( \frac x\ep , \omega\, ; \zeta\left( \frac{x-y}{\alpha}  \right) \right) -  \frac1{2\alpha}|x-y|^2  -\frac1{2\ep} (t-s)^2 \\- \lambda s -\gamma  \left(1+|x|^2 \right)^{\frac12} + \gamma.
\end{multline}
Using~\eqref{dvdsup},~\eqref{uesup} and~\eqref{ss}, we have, for each $(x,y,t,s,\omega) \in \Rd \times \Rd \times [0,T] \times [0,T] \times \Omega$,
\begin{equation}\label{crudeP}
|\Phi(x,y,t,s,\omega)| \leq C(1+T) + C \ep\delta^{-1}  - \frac{1}{2\alpha}|x-y|^2 - \frac{1}{2\ep} (t-s)^2 -\lambda s- \gamma \left(1+|x|^2 \right)^{\frac12} + \gamma.
\end{equation}
It follows that, for each $\omega\in \Omega$, the function $\Phi(\cdot,\omega)$ attains its global maximum at some point of $\Rd \times \Rd\times [0,T] \times [0,T]$. Set
\begin{equation*}\label{}
M(\omega):=\max_{\Rd \times \Rd\times [0,T] \times [0,T]} \Phi(\cdot,\omega)
\end{equation*}
and denote by $E\in \mathcal F$ the event that the global maximum of $\Phi(\cdot,\omega)$ is attained by some point $(x,y,t,s)$ with either $t=0$ or $s=0$; that is,
\begin{equation*}\label{}
E:=\left\{ \omega\in \Omega\,: \, M(\omega) = \sup_{(x,y,t,s) \in \Rd\times\Rd \times [0,T]\times [0,T]}  \Phi(x,y,0,s,\omega) \vee \Phi(x,y,t,0,\omega) \right\}
\end{equation*}
Note that, for notational simplicity, we omit the dependence of $E$ on $\alpha$ and $\ep$. 

The lemma follows from the fact that, under an appropriate choice of the parameters,
\begin{equation}\label{clmone}
\sup_{\omega\in E} \sup_{x\in B_T} \sup_{0 < t \leq T} \left( u^\ep(x,t,\omega) - u(x,t)\right)  \leq C \lambda T. 
\end{equation}
and, for some $R \leq C\delta \lambda T^2/ \ep^2$, 
\begin{equation}\label{clmtwo}
\Omega \setminus E \subseteq \left\{ \omega \in \Omega \, : \, \inf_{(y,p) \in B_{R} \times B_{L}} \left( -\delta v^\delta (y,\omega\,;p) - \overline H(p)\right)\leq - \frac \lambda2 \right\}.
\end{equation}
Proving~\eqref{clmone} is relatively easy, while~\eqref{clmtwo} involves a more involved comparison argument. 

\emph{Step 2.} We prepare for the proof of~\eqref{clmtwo} by recording two elementary estimates that necessarily hold at any global maximum point of $\Phi(\cdot,\omega)$ and for any $\omega\in \Omega$. For the moment, we fix $\omega\in \Omega$ and a point $(x_0, y_0, t_0, s_0) \in \Rd \times \Rd\times [0,T] \times [0,T]$ satisfying
\begin{equation}\label{bgpt}
\Phi( x_0, y_0, t_0, s_0,\omega) = M(\omega).
\end{equation}
By~\eqref{bgpt} and~\eqref{crudeP}, we have
\begin{equation*} \label{}
\Phi( x_0, y_0, t_0, s_0,\omega) \geq \Phi(0,0,0,0,\omega) \geq -C(1+T) - \frac{C\ep}{\delta}.
\end{equation*}
Substituting the definition of $\Phi$ and rearranging, using that $1+\ep \delta^{-1} \leq 2T$ by~\eqref{HHconds2}, we get
\begin{equation}\label{easybnds}
\gamma |x_0| + \frac{1}{2\ep} (t_0-s_0)^2 \leq C(1+T) + \frac{C\ep}{\delta}\leq CT.
\end{equation}
Recall that a Lipschitz function with Lipschitz constant $k$ cannot be touched from below (or above) by a $C^1$ function $\varphi$ unless $|D\varphi| \leq k$ at the touching point. We use this observation to deduce that, if $s_0\neq 0$, then by~\eqref{ulip} and the fact that $s\mapsto u(y,s) + \lambda s + (s-t)^2/2\ep$ has a local minimum at $s=s_0$, we have
\begin{equation}\label{s0t0}
|s_0-t_0| \leq (L +\lambda) \ep \leq (L+1)\ep.
\end{equation}
The inequality~\eqref{s0t0} is also satisfied for a similar reason if $t_0\neq 0$, and trivially if $s_0=t_0=0$, so it holds without restriction. We also use a similar idea to get that 
\begin{equation}\label{liplam}
|x_0-y_0| \leq L\alpha. 
\end{equation}
If not, then $y\mapsto \zeta( (x_0-y) / \alpha)$ is constant in a neighborhood of $y_0$ and we obtain from~\eqref{bgpt} that
\begin{equation} \label{e.touchu}
y \mapsto u(y,s_0) +  \frac1{2\alpha}|x_0-y|^2 \quad \mbox{has a local minimum at} \ y = y_0.
\end{equation}
In view of~\eqref{ulip}, we deduce from~\eqref{e.touchu} that $\alpha^{-1}|x_0-y_0| \leq L$. So we see that $|x_0-y_0| \leq L \alpha$ holds anyway, in contradiction to the assumption that it did not. Thus we obtain~\eqref{liplam} and, in particular,
\begin{equation}\label{zetafix}
\zeta\left( \frac{x_0-y_0}{\alpha} \right) = \frac{x_0-y_0}{\alpha}.
\end{equation}

We now begin the argument for~\eqref{clmtwo}, following the classical proof of the comparison principle for viscosity solutions and~\cite{ICD}. For the next several steps, we work with fixed $\omega\in \Omega\setminus E$ and $(x_0, y_0, t_0, s_0) \in \Rd \times \Rd\times (0,T] \times (0,T]$ such that~\eqref{bgpt} holds. 

\emph{Step 3.} We give the first part of the proof of~\eqref{clmtwo}. Here we fix $(x,t)=(x_0,t_0)$, allow $(y,s)$ to vary and use the equation for $u$. The goal is to derive~\eqref{goal1}, below.

From~\eqref{bgpt}, we see that 
\begin{multline}\label{frtp}
(y,s) \mapsto u(y,s) + \ep v^\delta \left( \frac {x_0}\ep , \omega\, ; \zeta\left( \frac{x_0-y}{\alpha}  \right)  \right) +  \frac1{2\alpha}|x_0-y|^2 +\frac1{2\ep} (t_0-s)^2 + \lambda s \\ \mbox{has a local minimum at} \ (y,s) = (y_0,s_0).
\end{multline}
According to~\eqref{vdpdepp} and~\eqref{ss},
\begin{multline}\label{lstln}
\ep v^\delta \left( \frac {x_0}\ep , \omega\, ; \zeta\left( \frac{x_0-y}{\alpha}  \right)  \right) - \ep v^\delta \left( \frac {x_0}\ep , \omega\, ; \zeta\left( \frac{x_0-y_0}{\alpha}  \right)  \right) \\ \leq \frac{C\ep}{\delta} \left| \zeta\left( \frac{x_0-y}{\alpha}  \right) - \zeta\left( \frac{x_0-y_0}{\alpha}  \right) \right| \leq \frac{C\ep|y-y_0|}{\delta\alpha}.
\end{multline}
Using~\eqref{frtp}, \eqref{lstln}, the fact that equality holds in~\eqref{lstln} at $y=y_0$ and by enlarging $C> 0$ slightly, we obtain that
\begin{multline}\label{frto}
(y,s) \mapsto u(y,s) +  \frac1{2\alpha}|x_0-y|^2 +\frac1{2\ep} (t_0-s)^2 + \lambda s + C\frac{\ep}{\delta\alpha} |y-y_0| \\ \mbox{has a strict local minimum at} \ (y,s) = (y_0,s_0).
\end{multline}
It follows that, for all sufficiently small $\beta > 0$, there exist  $(y_\beta,s_\beta)\in \Rd\times[0,T]$ such that $(y_\beta,s_\beta) \rightarrow (y_0,s_0)$ as $\beta \to 0$ and
\begin{multline}\label{frto2}
(y,s) \mapsto u(y,s) +  \frac1{2\alpha}|x_0-y|^2 +\frac1{2\ep} (t_0-s)^2 + \lambda s + C\frac{\ep}{\delta\alpha}  \left(\beta+|y-y_0|^2 \right)^{\frac12} \\ \mbox{has a local minimum at} \ (y,s) = (y_\beta,s_\beta).
\end{multline}
Using the equation for $u$, we obtain
\begin{equation*}\label{}
-\lambda + \frac1\ep (t_0 - s_\beta) + \overline H\left( \frac{x_0 - y_\beta}{\alpha} - Q_\beta \right) \geq 0,
\end{equation*}
where $Q_\beta := C\frac{\ep}{\delta\alpha}  \left( \beta+|y_\beta-y_0|^2 \right)^{-\frac12} (y_\beta-y_0)$. Since $|Q_\beta| \leq C\ep/\delta\alpha$, the Lipschitz continuity of $\overline H$ yields
\begin{equation*}\label{}
-\lambda + \frac1\ep (t_0 - s_\beta) + \overline H\left( \frac{x_0 - y_\beta}{\alpha}  \right) \geq - C\frac{\ep}{\delta\alpha},
\end{equation*}
and, after letting $\beta \to 0$, we find
\begin{equation}\label{goal1}
-\lambda + \frac1\ep (t_0-s_0) + \overline H\left( \frac{x_0-y_0}{\alpha} \right) \geq - C\frac{\ep}{\delta\alpha}. 
\end{equation}

\emph{Step 4.}
We give the second step in the proof of~\eqref{clmtwo}. Here we fix $(y,s) = (y_0,s_0)$ and let $(x,t)$ vary, in order to use the equations for $u^\ep$ and $v^\delta$. The intermediate goal is to prove~\eqref{goal2}, below, to complement~\eqref{goal1}. This leads to some analysis that is a little more complicated that what we have just performed above to produce~\eqref{goal1}, due to the fact that we have two (in general, nonsmooth) functions $u^\ep$ and $v^\delta$ which depend on the variable $x$. This sort of technical difficulty is typically handled using an ``iterated" perturbed test function argument (an idea introduced in~\cite{E2}). We need to quantify this idea, and so, following~\cite{ICD}, we fix another parameter $\sigma > 0$ (which will be sent to zero shortly) and introduce a second auxiliary function $\Psi:\Rd\times \Rd \times [0,T] \to \R$ defined by
\begin{multline}\label{auxfun2}
\Psi(x,z,t):= u^\ep ( x,t,\omega) - \ep v^\delta \left( \frac z\ep , \omega\, ; \zeta\left( \frac{z-y_0}{\alpha}\right) \right) -  \frac1{2\alpha}|x-y_0|^2  -\frac1{2\ep} (t-s_0)^2 \\ - \gamma\left(1+|x|^2 \right)^{\frac12} - \frac{1}{2\sigma} |z-x|^2 - \gamma \left(1 + |x-x_0|^2 \right)^{\frac12}-\frac\gamma2(t-t_0)^2.
\end{multline}
The last two terms in~\eqref{auxfun2} provide some strictness and therefore, by~\eqref{bgpt}, there exist points $(x_\sigma,z_\sigma,t_\sigma) \in \Rd \times\Rd \times [0,T]$ such that $(x_\sigma,z_\sigma,t_\sigma)\to (x_0,x_0,t_0)$ as $\sigma \to 0$ and
\begin{equation*}\label{}
\Psi(x_\sigma,z_\sigma,t_\sigma) = \sup_{\Rd\times\Rd\times[0,T]} \Psi. 
\end{equation*}
Freezing $z=z_\sigma$ and letting $(x,t)$ vary, we have
\begin{multline*}\label{}
(x,t) \mapsto u^\ep(x,t) -  \frac1{2\alpha}|x-y_0|^2  -\frac1{2\ep} (t-s_0)^2 - \gamma\left(1+|x|^2 \right)^{\frac12} - \frac{1}{2\sigma} |z_\sigma-x|^2\\ - \gamma \left(1 + |x-x_0|^2 \right)^{\frac12} -\frac\gamma2(t-t_0)^2\qquad 
 \mbox{has a local maximum at} \ (x,t) = (x_\sigma,t_\sigma).
\end{multline*}
It follows from the equation for $u^\ep$ that
\begin{equation*}\label{}
\frac1\ep (t_\sigma-s_0)+\gamma(t_\sigma-t_0) + H\left( \frac{x_\sigma-y_0}{\alpha} + \frac{x_\sigma-z_\sigma}{\sigma} + P_\sigma,\frac{x_\sigma}{\ep} ,\omega\right) \leq 0,
\end{equation*}
where $P_\sigma:= \gamma \left( 1+|x_\sigma|^2 \right)^{-\frac12} x_\sigma + \gamma \left( 1+|x_\sigma-x_0|^2 \right)^{-\frac12} (x_\sigma-x_0)$. Since $|P_\sigma| \leq C\gamma$, we use~\eqref{reg} to obtain
\begin{equation}\label{dbl1}
\frac1\ep (t_\sigma-s_0) +\gamma(t_\sigma-t_0)+ H\left( \frac{x_\sigma-y_0}{\alpha} + \frac{x_\sigma-z_\sigma}{\sigma}, \frac{x_\sigma}{\ep} ,\omega\right) \leq C \gamma. 
\end{equation}
On the other hand, freezing $(x,t) = (x_\sigma,t_\sigma)$ and letting $z$ vary, we get
\begin{equation*}\label{}
z \mapsto \ep v^\delta \left( \frac z\ep , \omega\, ; \zeta\left( \frac{z-y_0}{\alpha}\right) \right) +  \frac{1}{2\sigma} |z-x_\sigma|^2 \quad \mbox{has a local minimum at} \ z=z_\sigma. 
\end{equation*}
According to~\eqref{ss},
\begin{equation}\label{zetamv}
\left| \zeta\left( \frac{z-y_0}{\alpha} \right) - \zeta\left( \frac{z_{\sigma}-y_0}{\alpha} \right) \right| \leq \frac{C}{\alpha}|z-z_\sigma|.
\end{equation}
with equality at $z=z_\sigma$. Hence by using~\eqref{vdpdepp} again and enlarging $C$ slightly, we obtain 
\begin{multline*}\label{}
z \mapsto \ep v^\delta \left( \frac z\ep , \omega\, ; \zeta\left( \frac{z_\sigma-y_0}{\alpha} \right) \right) +  \frac{1}{2\sigma} |z-x_\sigma|^2 + C\frac{\ep}{\delta\alpha} |z-z_\sigma| \\ \mbox{has a strict local minimum at} \ z=z_\sigma. 
\end{multline*}
Thus we can find points $z_{\sigma,\kappa} \to z_\sigma$ as $\kappa \to 0$, such that 
\begin{multline*}\label{}
z \mapsto \ep v^\delta \left( \frac z\ep , \omega\, ; \zeta\left( \frac{z_\sigma-y_0}{\alpha} \right) \right) +  \frac{1}{2\sigma} |z-x_\sigma|^2 + C\frac{\ep}{\delta\alpha} \left(\kappa+ |z-z_\sigma|^2 \right)^{\frac12} \\ \mbox{has a local minimum at} \ z=z_{\sigma,\kappa}. 
\end{multline*}
Using the equation for $v^\delta$, we discover that
\begin{equation*}\label{}
\delta v^\delta \left( \frac{z_{\sigma,\kappa}}{\ep} , \omega\,;\zeta\left( \frac{z_\sigma-y_0}{\alpha}\right) \right) + H\left( \zeta\left(\frac{z_\sigma-y_0}{\alpha}\right) + \frac{x_\sigma - z_{\sigma,\kappa}}{\sigma} - P_{\sigma,\kappa} , \frac{z_{\sigma,\kappa}}{\ep} ,\omega \right) \geq 0,
\end{equation*}
where $P_{\sigma,\kappa} := C\frac{\ep}{\delta\alpha} \left( \kappa + |z_{\sigma,\kappa}-z_{\sigma}|^2 \right)^{-\frac12} (z_{\sigma,\kappa}-z_{\sigma})$. Since $|P_{\sigma,\kappa}| \leq C\ep / \delta\alpha$, we use~\eqref{reg} to get
\begin{equation*}\label{}
\delta v^\delta \left( \frac{z_{\sigma,\kappa}}{\ep} , \omega\,; \zeta\left( \frac{z_\sigma-y_0}{\alpha}\right) \right) + H\left( \zeta \left( \frac{z_\sigma-y_0}{\alpha}\right) + \frac{x_\sigma - z_{\sigma,\kappa}}{\sigma} , \frac{z_{\sigma,\kappa}}{\ep} ,\omega \right) \geq -C\frac{\ep}{\delta\alpha}. 
\end{equation*}
Letting $\kappa \to 0$ yields
\begin{equation}\label{dbl2}
\delta v^\delta \left( \frac{z_{\sigma}}{\ep} , \omega\,; \zeta\left( \frac{z_\sigma-y_0}{\alpha}\right) \right) + H\left( \zeta \left( \frac{z_\sigma-y_0}{\alpha}  \right) + \frac{x_\sigma-z_\sigma}{\sigma} , \frac{z_{\sigma}}{\ep} ,\omega \right) \geq -C\frac{\ep}{\delta\alpha}.
\end{equation}
Comparing~\eqref{dbl1} and~\eqref{dbl2}, using~\eqref{reg},~\eqref{vdpdepp} and~\eqref{zetamv} and then sending $\sigma \to 0$ yields
\begin{equation}\label{goal2}
\delta v^\delta \left( \frac{x_0}{\ep}, \omega\,;\frac{x_0-y_0}{\alpha} \right) \geq \frac1\ep (t_0-s_0) - C\frac{\ep}{\delta\alpha} - C\gamma.
\end{equation}

\emph{Step 5.} We complete the proof of~\eqref{clmtwo}. By combining~\eqref{goal1} and~\eqref{goal2} and setting $\gamma:= \ep / (\delta\alpha)$, we deduce that
\begin{equation*}\label{}
-\delta v^\delta\left( \frac{x_0}{\ep} , \omega\, ; \frac{x_0-y_0}{\alpha} \right) \leq \overline H\left( \frac{x_0-y_0}{\alpha}\right) + C\frac{\ep}{\delta\alpha} - \lambda. 
\end{equation*}
Select $\alpha:=  \lambda T$ to deduce that, for every $\delta \geq A\ep / T\lambda^2 = A\ep / (\alpha\lambda)$, we have $A\ep / \delta\alpha \leq \lambda$, and hence
\begin{equation*}\label{}
-\delta v^\delta\left( \frac{x_0}{\ep} , \omega\, ; \frac{x_0-y_0}{\alpha} \right) \leq \overline H\left( \frac{x_0-y_0}{\alpha}\right) - \frac12\lambda
\end{equation*}
provided that $A\geq 2$ is chosen sufficiently large. In light of~\eqref{easybnds} and~\eqref{liplam}, we have shown that
\begin{equation}\label{gotcha}
\Omega\setminus E \subseteq \left\{ \omega \in \Omega \, : \, \inf_{y\in B_R} \inf_{p\in B_{L}} \left( -\delta v^\delta (y,\omega\,;p) - \overline H(p)\right)\leq - \frac12\lambda \right\},
\end{equation}
where 
\begin{equation*}\label{}
R := \frac{|x_0|}{\ep} \leq \frac{CT}{\gamma\ep} = \frac{C\delta\lambda T^2}{\ep^2}.
\end{equation*}
This completes the proof of~\eqref{clmtwo}. 

\emph{Step 6.} We give the argument for~\eqref{clmone}. Note that $\delta \geq A\ep/T\lambda^2 \geq A\ep/T\lambda$ and $2\gamma \leq \lambda$ since $A\geq 2$. Therefore, for any $\omega\in \Omega$, $x\in B_T$ and $0\leq t \leq T$,
\begin{multline}\label{flod}
u^\ep(x,t,\omega) - u(x,t) = \Phi(x,x,t,t,\omega) + \ep v^\delta(x,\omega\,;0) + \lambda t + \gamma(1+|x|^2)^{\frac12} - \gamma \\  \leq M(\omega) + C \ep\delta^{-1} +  \lambda T + \gamma T  \leq M(\omega) + 3\lambda T.
\end{multline}
Fix $\omega\in E$ and let $(x_0,y_0,t_0,s_0)\in \Rd\times\Rd\times[0,T]\times[0,T]$ be a point which satisfies~\eqref{bgpt} and such that either $t_0=0$ or $s_0=0$. If $t_0=0$, then using~\eqref{s0t0},~\eqref{liplam},~\eqref{HHconds2} and the assumption that $0 < \ep \leq \lambda T$,
\begin{multline*}\label{}
M(\omega) = \Phi(x_0,y_0,t_0,s_0,\omega) \leq u_0(x_0) - u(y_0,s_0) - \ep v^\delta\left(\frac{x_0}{\ep},\omega\,; \frac{x_0-y_0}{\alpha} \right) \\ \leq L|x_0-y_0| + L s_0 + C\ep\delta^{-1} 
\leq L^2 \alpha + C(L+1)\ep + C\ep\delta^{-1}  \leq C \lambda T.
\end{multline*}
If $s_0=0$, we also get the bound $M(\omega) \leq C\lambda T$ via a similar observation. Combining this with~\eqref{flod}, we obtain
\begin{equation*}\label{}
\sup_{x\in B_T} \sup_{0 < t \leq T} \left( u^\ep(x,t,\omega) - u(x,t)\right)  \leq C \lambda T. 
\end{equation*}
This completes the proof of~\eqref{clmone} and thus that of~\eqref{incluHH-2}.
\end{proof}

We are now ready to complete the proofs of our main results. 

\begin{proof}[{\bf Proof of Theorem~\ref{EEH}}]
The theorem is obtained in a straightforward way from the combination of Lemma~\ref{incluHH} and Proposition~\ref{ballaz}. 

We first give the proof of (i). Fix $0<\ep\leq 1$ and $T\geq 1$, set $\delta := A \ep / T\lambda^2$, where $A>0$ is the constant $C>0$ from Lemma~\ref{incluHH}, and let $0<\lambda \leq 1$. Observe that $\lambda \geq C \ep^{\frac13}$ for large enough $C>0$ implies that
\begin{equation*} \label{}
\delta  =  \frac{A\ep}{T\lambda^2} \leq \frac{A}{T\lambda^2}\left( \frac{\lambda}{C} \right)^3 \leq c\lambda,
\end{equation*}
and thus for such $\lambda$ the hypothesis of Proposition~\ref{ballaz}(i) is in force. We apply first~\eqref{incluHH-1} and then second~\eqref{dvdEEaboveR} to discover that
\begin{align*}
\Prob\left[ \inf_{x\in B_T} \inf_{0\leq t \leq T} \left( u^\ep(x,t,\cdot) - u(x,t) \right)  \leq - \lambda T \right]
& \leq \Prob \left[ \sup_{|y| \leq C\delta \lambda T^2/\ep^2  } \ \sup_{|p| \leq C} \left( -\delta v^\delta(y,\omega\,;p) - \overline H(p) \right) \geq c\lambda  \right] \\
& \leq C\left( \delta \lambda T^2 \ep^{-2} \right)^d \delta^{-5d} \exp\left( -\frac{\lambda^3}{C\delta}\right) \\
& = C T^{6d} \lambda^{9d} \ep^{-6d} \exp\left( -\frac{T\lambda^5}{C\ep}\right).
\end{align*}

We move along to the argument for (ii). Just as above, we fix $0<\ep,\lambda \leq 1$ and $T\geq 1$ and set $\delta := A \ep / T\lambda^2$, where $A>0$ is the constant $C>0$ from Lemma~\ref{incluHH}. We need check that, with this choice of $\delta$, the hypothesis \eqref{alpbangp} with $C>0$ sufficiently large implies~\eqref{lambangR}. Indeed, if $\lambda \geq C\ep^{\tfrac18}\left| \log \ep\right|^{\tfrac 3{16}}$, then
\begin{equation*} \label{}
\left| \log \delta\right| \leq C \left| \log \ep \right|
\end{equation*}
and hence 
\begin{align*}
\lambda\geq C\ep^{\tfrac18}\left| \log \ep\right|^{\tfrac 3{16}} \geq C\left(\frac{T\lambda^2 \delta}{A} \right)^{\frac18} \left| \log \ep \right|^{\frac3{16}} \geq \left( CA^{-\frac18}  \right) \lambda^{\frac14} \delta^{\frac18} \left| \log \delta \right|^{\frac3{16}}.
\end{align*}
A rearrangment produces
\begin{equation*} \label{}
\lambda \geq C^{\frac43} A^{-\frac{1}{6}}\delta^{\frac16}  \left| \log \delta \right|^{\frac14}
\end{equation*}
which implies~\eqref{lambangR}, as desired, if we take $C$ sufficiently large. Now combine~\eqref{EEFqR} and~\eqref{incluHH-2}:
\begin{align*}
\Prob\left[ \sup_{x\in B_T} \sup_{0\leq t \leq T} \left( u^\ep(x,t,\cdot) - u(x,t) \right)  \geq \lambda T \right]
& \leq \Prob \left[ \inf_{|y| \leq C\delta \lambda T^2/\ep^2  } \ \inf_{|p| \leq C} \left( -\delta v^\delta(y,\omega\,;p) - \overline H(p) \right) \leq -c\lambda  \right] \\
& \leq C\left( \delta \lambda T^2 \ep^{-2} \right)^d \delta^{-5d} \exp\left( - \frac1{C} \left(  \frac{\lambda^3}{\delta} \wedge \frac{\lambda^{\d+\theta}}{\delta^\d}\right) \right) \\
& = C T^{6\d} \lambda^{9\d} \ep^{-6\d} \exp\left( - \frac1{C} \left(  \frac{T\lambda^5}{\ep} \wedge \frac{T^\d\lambda^{3\d+\theta}}{\ep^\d}\right) \right).
\end{align*}
This completes the proof.
\end{proof}

\begin{proof}[{\bf Proof of Theorem~\ref{CRH}}]
The result follows from a combination of Lemma~\ref{incluHH} and Proposition~\ref{acpCRU}. We first prove~(i). Define, for each $\ep > 0$, 
\begin{equation*}\label{}
\lambda(\ep):= A \ep^{\frac15} |\log \ep |^{\frac15} \qquad \mbox{and} \qquad \delta (\ep):= \frac{A\ep}{T\lambda(\ep)^2},
\end{equation*}
where $A \geq 1$ will be selected below. It is straightforward to check that
\begin{equation*}\label{}
\lambda(\ep) \geq  A^{\frac43} \delta(\ep)^{\frac13} \left| \log \delta(\ep) \right|^{\frac13}.
\end{equation*}
According to~\eqref{incluHH-1}, 
\begin{align*}
\lefteqn{   \bigcap_{\eta >0} \, \bigcup_{0 < \ep \leq\eta} \left\{ \omega\in \Omega \, : \, \inf_{x\in B_T} \inf_{0<t\leq T} \left( u^\ep(x,t,\omega) - u(x,t) \right) \leq - CT\lambda(\ep) \right\}   } \qquad & \\
& \subseteq \bigcap_{\eta >0} \bigcup_{0 < \ep \leq\eta} \left\{ \omega\in \Omega \, : \,  \sup_{|y| \leq C\delta(\ep) \lambda(\ep) T^2/\ep^2  } \sup_{|p| \leq C} \left( -\delta(\ep) v^{\delta(\ep)} (y,\omega\,;p) - \overline H(p) \right) \geq \lambda(\ep)  \right\} \\
& \subseteq \bigcap_{\eta > 0} \bigcup_{0 < \ep \leq\eta} \left\{ \omega\in \Omega \, : \,  \sup_{|y| \leq CT / \ep \lambda(\ep) } \sup_{|p| \leq C} \left( -\delta(\ep) v^{\delta(\ep)} (y,\omega\,;p) - \overline H(p) \right) \geq  A \delta(\ep)^{\frac13} \left| \log \delta(\ep) \right|^{\frac13}  \right\}.
\end{align*}
If we choose $A$ large enough, then~\eqref{aboveRas2} yields that the last event is of probability zero, where we have used $N=2$ in~\eqref{aboveRas2} and the fact that $\ep\lambda(\ep) \leq \delta(\ep)^2$ for small enough $\ep$. We deduce that
\begin{equation*}\label{}
\Prob\left[ \liminf_{\ep \to 0} \inf_{x\in B_T} \inf_{0<t\leq T} \frac{u^\ep(x,t,\cdot) - u(x,t)}{\lambda(\ep)} \geq -T \right] = 1.
\end{equation*}
This completes the proof of~(i).

To prove~(ii), we define instead
\begin{equation*}\label{}
\lambda(\ep):= A \ep^{\bar a} |\log\ep|^{\bar b},
\end{equation*}
with $\delta(\ep)$ the same as above and $A\geq 1$ to be selected. With $\alpha$ and $\beta$ as defined in~\eqref{alphbeta}, and recalling that  $\bar a=\alpha(1+2\alpha)^{-1}$ and $\bar b=\beta(1+2\alpha)^{-1}$, we see that
\begin{equation*} \label{}
\alpha = \frac{\bar a}{1-2\bar a} \quad \mbox{and} \quad \beta = \frac{\bar b}{1-2\bar a}.
\end{equation*}
Thus for sufficiently large $A$, we find
\begin{equation*} \label{}
\lambda(\ep) \geq A \left( \frac{T\delta(\ep)\lambda(\ep)^2}{C} \right)^{\bar{a}}  \left| \log \ep \right|^{\bar{b} } \geq A \left( \frac{T\delta(\ep)\lambda(\ep)^2}{A} \right)^{\bar{a}}  \left| \log \delta(\ep) \right|^{\bar{b} }
\end{equation*}
and a rearrangement produces
\begin{align*}\label{}
\lambda(\ep) \geq c \left( A^{1-\bar a} \delta(\ep)^{\bar a} \left| \log \delta(\ep) \right|^{\bar b} \right)^{1/(1-2\bar a)} = c A^{(1-\bar a)/(1-2\bar a)} \delta(\ep)^\alpha \left| \log \delta(\ep) \right|^\beta \geq C \delta(\ep)^\alpha \left| \log \delta(\ep) \right|^\beta.
\end{align*}
We proceed similarly as above, using~\eqref{incluHH-2} to obtain
\begin{align*}
\lefteqn{   \bigcap_{\eta >0} \, \bigcup_{0 < \ep \leq\eta} \left\{ \omega\in \Omega \, : \, \sup_{x\in B_T} \sup_{0<t\leq T} \left( u^\ep(x,t,\omega) - u(x,t) \right) \geq C T \lambda(\ep) \right\}   } \qquad & \\
& \subseteq \bigcap_{\eta >0} \bigcup_{0 < \ep \leq\eta} \left\{ \omega\in \Omega \, : \,  \inf_{|y| \leq C\delta(\ep) \lambda(\ep) T^2/\ep^2  } \inf_{|p| \leq C} \left( -\delta(\ep) v^{\delta(\ep)} (y,\omega\,;p) - \overline H(p) \right) \leq -\lambda(\ep)  \right\} \\
& \subseteq \bigcap_{\eta > 0} \bigcup_{0 < \ep \leq\eta} \left\{ \omega\in \Omega \, : \,  \inf_{|y| \leq CT / \ep \lambda(\ep) } \inf_{|p| \leq C} \left( -\delta(\ep) v^{\delta(\ep)} (y,\omega\,;p) - \overline H(p) \right) \leq - C \delta(\ep)^{\alpha} \left| \log \delta(\ep) \right|^{\beta}  \right\}.
\end{align*}
According to~\eqref{belowRasF3}, the last set on the right has probability zero. It follows that
\begin{equation*}\label{}
\Prob\left[ \limsup_{\ep \to 0} \sup_{x\in B_T} \sup_{0<t\leq T} \frac{u^\ep(x,t,\cdot) - u(x,t)}{\lambda(\ep)} \leq T \right] = 1,
\end{equation*}
which completes the proof of~(ii).
\end{proof}

\section{Convergence rates in almost periodic environments} \label{APs}

We conclude by showing that the techniques in the previous section may be used to obtain convergence rate for the homogenization of~\eqref{HJq} in almost periodic (and in particular periodic) media. Since it is necessary to reproduce the qualitative homogenization theory from scratch, we take the opportunity to efficiently reorganize and quantify the argument.

Departing from the hypotheses in the rest of this paper, here we consider $H \in C(\Rd\times\Rd)$ satisfying, for each $K> 0$, the regularity assumption
\begin{equation}\label{regAP}
H \ \mbox{is uniformly continuous on} \  B_K \times \Rd \ \mbox{and} \ \left\{ H(\cdot,y) \, : \, y\in \Rd\right\} \ \mbox{is bounded in} \ C^{0,1}(B_K)
\end{equation}
and the coercivity condition
\begin{equation}\label{coerAP}
\lim_{|p| \to \infty} \inf_{y\in \Rd} H(p,y) = +\infty.
\end{equation}
The assumption of almost periodicity is that the family of translations of $H$ in the $y$-variable is precompact in the uniform topology of $B_K\times \Rd$. Precisely, we assume that, for all $K>0$,
\begin{equation}\label{AP}
\left\{ H(\cdot, \cdot + y) \, : \, y\in \Rd \right\} \quad \mbox{is precompact in} \ C(B_K\times \Rd). 
\end{equation}
We remark that, in this section, we make no convexity assumption on $H$, nor do we assume any analogue of~\eqref{cntl} or~\eqref{plushyp}. 

The homogenization of coercive Hamilton-Jacobi equations in almost periodic environments was proved in~\cite{I}. The key observation was an elegant proof of the following fact.

\begin{prop}[{Ishii~\cite{I}}] \label{APcorr}
Assume that $H\in C(\Rd\times\Rd)$ satisfies~\eqref{regAP}, \eqref{coerAP} and~\eqref{AP}. Then there exists $\overline H \in C(\Rd)$ such that, for every $p\in \Rd$, 
\begin{equation}\label{dvdlimAP}
\lim_{\delta\to 0} \sup_{y\in \Rd} \left| \delta v^\delta(y\,;p) + \overline H(p) \right| = 0.
\end{equation}
\end{prop}

The rate of convergence in almost periodic environments follows from Lemma~\ref{incluHH} once a rate for the limit~\eqref{dvdlimAP} is obtained. For the latter, it is necessary to quantify the almost periodicity of $H$, which leads us to introduce, for each $K > 0$ and $R> 0$,
\begin{equation*}\label{}
\rho_K(R): =  \sup_{y\in \Rd} \inf_{z\in B_R} \sup_{(p,x) \in B_K \times \Rd} \left| H(p,x+y) - H(p,x+z) \right|.
\end{equation*}
It is immediate from~\eqref{regAP} that, for each $K> 0$, $\rho_K$ is continuous and we see from its definition that it is decreasing. The assumption~\eqref{AP} is equivalent to the statement that, for each $K> 0$,
\begin{equation*}\label{mod1}
\lim_{R\to\infty} \rho_K(R) = 0. 
\end{equation*}
We next define, for each $K> 0$ and $0 < \delta < 1$, 
\begin{equation}\label{etaK}
\eta_K(\delta):=  4 \inf\left\{ s> 0\, : \, \rho_K\left(\frac{s}{K\delta}\right) \leq s \right\}.
\end{equation}
The properties of $\rho_K$ yield that $\eta_K$ is a modulus, i.e., for each $K> 0$,
\begin{equation}\label{}
\eta_K: (0,1) \to [0,\infty) \quad \mbox{is continuous, increasing and} \   \lim_{\delta \to 0} \eta_K(\delta) = 0.
\end{equation}
Observe that if $y\mapsto H(p,y)$ is $1$-periodic, then $\rho_K\!\left(\frac12\right) = 0$ for all $K> 0$ and, hence, $\eta_K(\delta) \leq \frac12\delta$. 

We also define, for each $K> 0$, the quantity
\begin{equation}\label{LK}
L=L(K):= \sup\left\{ |q| \, : \, \inf_{y\in \Rd} H(q,y) \leq \sup_{(p,y) \in B_R \times \Rd} H(p,y) \right\},
\end{equation}
which has the property (this is not difficult to check using similar arguments as in Appendix~\ref{appmappp}) that, for every $p\in \Rd$,
\begin{equation}\label{explipbnd}
|Dv^\delta(\cdot\,;p)| \leq L(|p|).
\end{equation}

We prove next a quantitative version of Proposition~\ref{APcorr}. The argument is inspired from~\cite{LPV,I}. Here we simply reorganize and quantify it.

\begin{prop} \label{alpfundy}
Assume that $H\in C(\Rd\times\Rd)$ satisfies~\eqref{regAP}, \eqref{coerAP} and~\eqref{AP}. Then there exists $\overline H \in C(\Rd)$ such that, for every $K > 0$, $\delta,\gamma \in (0,1]$ and $p\in B_K$,
\begin{equation}\label{alpfunest}
\sup_{y \Rd} \left| \delta v^\delta(y\,;p) + \overline H(p) \right| \leq \eta_{L}(\delta),
\end{equation}
where $L=L(K)$ is given by~\eqref{LK}.
\end{prop}
\begin{proof}
The result follows from two facts. The first is that, for every $\delta > 0$,
\begin{equation}\label{numero-un}
\osc_{\Rd} \delta v^\delta(\cdot\,;p) \leq \eta_L(\delta),
\end{equation}
and the second is that, for every $\delta,\gamma \in (0,1]$,
\begin{equation}\label{numero-deux}
\inf_{\Rd}  \gamma v^\gamma(\cdot\,;p) \leq \sup_{\Rd} \delta v^\delta(\cdot\,;p).
\end{equation}
Indeed, it is immediate from~\eqref{numero-un} and~\eqref{numero-deux} that, if we define 
\begin{equation*}\label{}
\overline H(p) := \liminf_{\delta \to 0} -\delta v^\delta(0\,;p),
\end{equation*}
then~\eqref{alpfunest} holds. The continuity of $\overline H$ is then immediate from Proposition~\ref{vdpdepp}.

\emph{Step 1.}
We prove \eqref{numero-un}. Fix $\hat y \in \Rd$. Let $R> 0$ be selected below, and choose $z\in \overline B_R$ such that
\begin{equation}\label{flagdown}
\sup_{(q,y)\in B_L\times \Rd} \left| H(q,y+\hat y) - H(q,y+z) \right| \leq \rho_L(R). 
\end{equation}
It follows from Proposition~\ref{appcmp}, by comparing $v^\delta(\cdot+\hat y\,;p)$ to $v^\delta(\cdot+z\,;p) \pm \rho_L(R)/\delta$, that
\begin{equation*}\label{}
\sup_{y\in\Rd} \left| \delta v^\delta (y+\hat y\,;p) - \delta v^\delta (y+z\,;p) \right| \leq \rho_L(R).
\end{equation*}
In particular, and in view of~\eqref{explipbnd}, we have
\begin{equation*}\label{}
\sup_{y\in\Rd} \left| \delta v^\delta (\hat y\,;p) - \delta v^\delta (0\,;p) \right| \leq \rho_L(R) + \delta L |z| \leq \rho_L(R) + \delta LR.
\end{equation*}
Optimizing over $R$ leads to the choice $R := \eta_L(\delta)/4\delta L$, which yields, in light of~\eqref{etaK},
\begin{equation*}\label{}
 \left| \delta v^\delta (\hat y\,;p) - \delta v^\delta (0\,;p) \right| \leq \frac12 \eta_L(\delta).
\end{equation*}
Since $\hat y \in \Rd$ was arbitrary, we obtain~\eqref{numero-un}.

\emph{Step 2.}
We give the proof of~\eqref{numero-deux}, which is essentially taken from~\cite{LPV}. Suppose, for some $\delta,\gamma\in (0,1]$, that~\eqref{numero-deux} is false. For $0 < \alpha \leq 1$ to be selected, consider the function
\begin{equation*}\label{}
w(y):= v^\delta(y\,;p) - \alpha \left( 1 + |y|^2 \right)^{\frac12}. 
\end{equation*}
Using~\eqref{regAP}, we see that, if $\alpha > 0$ is chosen suitable small, then
\begin{equation*}\label{}
H(p+Dw,y) \leq -\sup_{\Rd} \delta v^\delta(\cdot\,;p) + C\alpha < -\inf_{\Rd} \gamma v^\gamma(\cdot \,;p) \leq H(p+Dv^\gamma,y) \quad \mbox{in} \ \Rd.
\end{equation*}
The comparison principle (Proposition~\ref{comp}),~\eqref{dvdsup} and~\eqref{regAP} imply that, for every $R > 0$,
\begin{equation*}\label{}
w(0) - v^\gamma(0) \leq \max_{\partial B_R} (w-v^\gamma) \leq C \left( \frac1\gamma + \frac1\delta \right) - \alpha R.
\end{equation*}
Send $R\to +\infty$ to obtain the desired contradiction.
\end{proof}

The combination of Proposition~\ref{alpfundy} and Lemma~\ref{incluHH} yields the following convergence rate for the homogenization of~\eqref{HJq} in almost periodic media. In order to apply Lemma~\ref{incluHH}, we note that its proof did not depend in any way on the random environment or the structural assumptions, such as level-set convexity or~\eqref{cntl}, which are not in force in this section. 

\begin{prop}\label{almper}
Assume that $H\in C(\Rd\times\Rd)$ satisfies~\eqref{regAP}, \eqref{coerAP} and~\eqref{AP}. Consider the unique solutions $u^\ep,u\in C^{0,1}(\Rd\times [0,T])$ of
\begin{equation*}\label{}
u^\ep_t + H\left(Du^\ep,\frac x\ep\right) = 0 \quad \mbox{and} \quad u_t + \overline H(Du) = 0 \quad  \mbox{in} \ \Rd \times (0,T)
\end{equation*}
subject to the initial condition $u^\ep(0,t) = u(0,t) = u_0(x) \in C^{0,1}(\Rd)$, and let $K> 0$ be such that
\begin{equation*}\label{}
|u^\ep(x,t) - u^\ep(y,s)| \vee |u(x,t) - u(y,s)| \leq K\!\left( |x-y| + |t-s|\right).
\end{equation*}
Then there exists a constant $\Cl[C]{C-almper}>0$ such that, for all $T\geq 1$ and $\ep > 0$,
\begin{equation}\label{alperees}
\sup_{(x,t) \in \Rd \times [0,T] } \left| u^\ep(x,t) - u(x,t) \right| \leq \Cr{C-almper}T \left( \ep^{\frac13} + \eta_L\big(\ep^{\frac13} \big) \right),
\end{equation}
where $L=L(K)$ is given by~\eqref{LK} and the modulus $\eta_L(\cdot)$ by~\eqref{etaK}.
\end{prop}
\begin{proof}
For $\ep,\alpha > 0$, we define $\delta(\ep):= \ep^{\frac13}$ and 
\begin{equation*}\label{}
\lambda(\ep,\alpha) := \left( (C/T)^\frac12 \delta(\ep)\right) \vee \eta_L(\delta)+\alpha
\end{equation*}
and observe that~\eqref{HHconds} holds for $\alpha,\ep > 0$ small enough. An application of Lemma~\ref{incluHH} yields
\begin{equation*}\label{}
\sup_{(x,t) \in \Rd \times [0,T] } \left| u^\ep(x,t) - u(x,t) \right| \leq CT \lambda(\ep,\alpha).
\end{equation*}
Let $\alpha \to 0$ to get~\eqref{alperees}. 
\end{proof}

Observe that for a periodic Hamiltonian satisfying~\eqref{regAP} and~\eqref{coerAP}, Proposition~\ref{almper} gives a rate of convergence of $O\big(\ep^{\frac13}\big)$ for  homogenization.

\appendix

\section{Sketches of the proofs of Propositions~\ref{existMP},~\ref{e.weakner} and~\ref{vdeltas}} \label{appmappp}

Throughout this section, we assume that $H$ satisfies~\eqref{assum}. 

We begin with the following helpful lemma, which is due to the level set convexity of $H$ and is useful for checking whether~$u\in\Lip$ is a subsolution of the equation $H(Du,y,\omega) \leq \mu$ for $\mu \in \R$. A simple proof can be found in~\cite{ASo3}.
\begin{lem} \label{convtrick}
Let $\mu\in \R$, $\omega\in\Omega$ and $U\subseteq \Rd$ be open. Then $u\in \USC(U)$ is a viscosity solution of 
\begin{equation}\label{convteq}
H(Du,y,\omega) \leq \mu \quad \mbox{in} \ U
\end{equation}
if and only if $u$ is locally Lipschitz in $U$ and satisfies \eqref{convteq} almost everywhere in $U$. 
\end{lem}

Obvious analogues of Lemma~\ref{convtrick} hold for equations with zero order terms, and so forth. We leave these to the reader.

A commonly used fact in the theory of viscosity solutions is that the supremum (infimum) of a family of subsolutions (supersolutions) is a subsolution (supersolution), see~\cite{CIL}. Observe that, in light of Lemma~\ref{convtrick}, the infimum of a family of subsolutions of~\eqref{convteq} is a subsolution and, in particular, the infimum of a family of solutions of~\eqref{convtrick} is a solution. 

We next give details for some elementary facts concerning the functions $m_\mu$ defined in~\eqref{defmmu}. Most of what follows is well-known and can be found for example in~\cite{Li} or~\cite{ASo3}, but we give sketches of the arguments for completeness and the convenience of the reader. Here $\mu > \overline H_*$, where $\overline H_*$ is a critical parameter defined as the infimum of all $\mu$ for which the equation
\begin{equation} \label{eikp}
H(Du,y,\omega) = \mu
\end{equation}
admits a global subsolution $u\in C(\Rd)$ in $\Rd$. It turns out (see~\cite{ASo3}) that $\overline H_* = \min \overline H$ and the assumption \eqref{cntl} implies that $\overline H_* = \overline H(0) = 0$.

We begin by stating a comparison principle, which makes minimal assumptions on the growth of the subsolution and supersolution at infinity.
\begin{prop}[{\cite[Proposition 3.1]{ASo3}}] \label{compMP}
Let $\mu > 0$, $K\subseteq \Rd$ be compact and $u,-v\in \USC(\Rd\setminus K)$ satisfy
\begin{equation}\label{eikext}
H(Du,y,\omega) \leq \mu \leq H(Dv,y,\omega) \quad \mbox{in} \ \Rd \setminus K \quad
\mbox{and} \quad \limsup_{y\to K} \left( u(y) - v(y) \right) \leq 0
\end{equation}
and
\begin{equation*}\label{}
\liminf_{|y|\to \infty} \frac{v(y)}{|y|} \geq 0.
\end{equation*}
Then $u\leq v$ in $\Rd \setminus K$. 
\end{prop}

\begin{proof}[{\bf Proof of Proposition~\ref{existMP}}]
(i) Perron's method yields that $m_\mu(\cdot,x,\omega)$ is a solution of~\eqref{mpagan} in $\Rd\setminus \{ x \}$. Recall from~\eqref{globsub} that it is also a subsolution in $\Rd$. The uniqueness for $\mu > 0$ follows from~Proposition~\ref{compMP}. 

(ii) Assume that ~\eqref{maxm} fails and, by adding a constant to $u$, that the right side of \eqref{maxm} is negative while the left side is positive at some point $y_0\in U$. Define 
\begin{equation*}\label{}
v(y):= \begin{cases} m_\mu(y,x,\omega) & \mbox{if} \ y\in \Rd\setminus U, \\  u(y) \vee m_\mu(y,x,\omega) & \mbox{if} \ y \in U,
\end{cases}
\end{equation*}
and observe that $v(x) = 0$, $m_\mu(\cdot,x,\omega) \leq v$ in $\Rd$ and $m_\mu(y_0,x,\omega) < v(y_0)$. Moreover, it is clear from its definition that $v$ is a subsolution of~\eqref{eikp} in $\Rd$. This contradicts \eqref{defmmu}.

(iii) follows from (ii). Indeed, by Lemma~\ref{convtrick}, the map $y\mapsto m_\mu(y,x,\omega) - m_\mu(z,x,\omega)$ is a subsolution of \eqref{eikp} in $\Rd$. The maximality of $m_\mu(\cdot,z,\omega)$ (i.e., property (ii) with $U=\Rd \setminus \{ z \}$) then implies \eqref{subadd}. 

(iv) The lower bound of~\eqref{control2} follows from the observation that, if $\mu \leq K$, then~\eqref{reg} and~\eqref{cntl} imply the existence of some $c > 0$, depending on $K$, such that, for every $x\in \Rd$, the function $y\mapsto c\mu |y-x|$ is a subsolution of~\eqref{eikp} in $\Rd$. The upper bound is immediate from Proposition~\eqref{compMP} and the fact that, for large enough $C> 0$ and any $x\in \Rd$, the map $y\mapsto C|y-x|$ is a supersolution in $\Rd\setminus \{ x\}$. 

(v) The Lipschitz estimate~\eqref{lips} is immediate from~\eqref{subadd} and~\eqref{control2}.

(vi) One direction of~\eqref{dynprog2} is obvious from \eqref{subadd} and holds without restriction on $x,y\in \Rd$ and $\omega\in \Omega$. For the other, we first assume that $U$ is bounded and note that, due to Lemma~\ref{convtrick}, 
\begin{equation*}
\phi(y):= \min_{z\in \partial U} \left( m_\mu(y,z,\omega) + m_\mu(z,x,\omega)\right)
\end{equation*}
is a solution of $H(D\phi,y,\omega) \leq \mu$ in $\Rd \setminus \overline U$. For $y\in \partial U$ we may take $z=y$ in the minimum to obtain $\phi(y) \leq m_\mu(y,x,\omega)$. The maximality of $m_\mu(\cdot,x,\omega)$ yields $\phi\leq m_\mu(\cdot,x,\omega)$ in $\Rd\setminus U$, which is the other side of~\eqref{dynprog2}. If $U$ is not bounded, then we approximate $U$ by bounded sets and use~\eqref{control2}, which implies that points $z\in \partial U$ which are very far away from $x$ and $y$ are irrelevant.

(vii) follows from Lemma~\ref{convtrick} and \eqref{defmmu}. Indeed, it is immediate from the definitions that $m_\mu(y,x,\omega) = n_\mu(x,y,\omega)$, where $n_\mu$ is the function corresponding to $m_\mu$ for the Hamiltonian $H(-p,y,\omega)$. Thus~\eqref{flipx2} holds and we can apply Lemma~\ref{convtrick} to obtain~\eqref{flipx}. 

(viii) follows from~\eqref{reg} and (ii). We observe that, as a function of $y$, the left side of~\eqref{strinc} is a subsolution of~\eqref{eikp}, provided that we take $c> 0$ small enough. Both sides of~\eqref{strinc} vanish at $y=x$, and so the result follows from the maximality of $m_\mu(\cdot,x,\omega)$. 
\end{proof}

\begin{proof}[{\bf Proof of Proposition~\ref{e.weakner}}]
That $m_\mu(\cdot,K,\omega)$ and $m_\mu(K,\cdot,\omega)$ satisfy~\eqref{mmuKeqs} is a fact which is immediate from~\eqref{mpagan} and~\eqref{flipx2} and the fact that, in light of Lemma~\ref{convtrick}, the infimum of a family of solutions is a solution. Uniqueness follows from Proposition~\ref{compMP}.

It is easy to check from the standard Perron argument, using again the fact that a minimum of solutions is a solution, that the right side of~\eqref{Krep1}, as a function of $z$, is a solution of the first equation in~\eqref{mmuKeqs}. Clearly it vanishes on $K$ by taking $x=z$ in the minimum and the zero function in the supremum. Therefore it must be equal to $m_\mu(\cdot,K,\omega)$ by the uniqueness of the latter. This proves~\eqref{Krep1}, and~\eqref{Krep2} is obtained similarly.
\end{proof}

We conclude with a sketch the proof of~Proposition~\ref{vdeltas}. Recall that~$v^\delta$ is defined in~\eqref{dvdform}.

\begin{proof}[{\bf Proof of Proposition~\ref{vdeltas}}]
(i) The boundedness of $v^\delta$ was proved in~\eqref{dvdsup2}. That \eqref{dvdform} gives a solution of \eqref{amp} in $\BUC(\Rd)$ follows from the classical Perron argument adapted to viscosity solutions (see~\cite{CIL}). The uniqueness of $v^\delta$ is immediate from Proposition~\ref{appcmp}.
 
(ii) According to \eqref{dvdsup2}, $v^\delta(\cdot,\omega\,;p)$ is a solution of the inequality
\begin{equation} \label{bleechers}
H(p+Dv^\delta,y,\omega) \leq \esssup_{z\in \Rd} H(p,z,\omega).
\end{equation}
According to Lemma~\ref{convtrick}, $v^\delta(\cdot,\omega\,;p)$ is locally Lipschitz, hence differentiable almost everywhere with $|Dv^\delta|\in L^\infty_{\mathrm{loc}}(\Rd)$ and satisfies~\eqref{bleechers} in the almost everywhere sense. This implies that, for Lebesgue-almost every~$x\in\Rd$,
\begin{align*} \label{}
Dv^\delta(x,\omega\,;p) & \in \left\{ q\in \Rd\,:\,  \inf_{y\in\Rd} H(p+q,y,\omega) \leq \sup_{y\in \Rd} H(p,y,\omega) \right\} \\
& = \left\{ q-p\,:\, q\in \Rd \ \ \mbox{and}\ \ \inf_{y\in\Rd} H(q,y,\omega) \leq \sup_{y\in \Rd} H(p,y,\omega) \right\}.
\end{align*}
Hence for Lebesgue-almost every $x\in \Rd$,
\begin{equation*} \label{}
\left| Dv^\delta(x,\omega\,;p) \right| \leq \sup\left\{ |q-p| \,:\, q\in \Rd \ \ \mbox{and} \  \ \inf_{y\in\Rd} H(q,y,\omega) \leq \sup_{y\in \Rd} H(p,y,\omega) \right\} \leq K_p. 
\end{equation*}
It follows that $v^\delta(\cdot,\omega\,;p)$ is Lipschitz with constant $K_p$. 

(iii) The dependence of $\delta v^\delta$ on $p$ can be controlled using the comparison principle together with~\eqref{reg} and \eqref{vdlip2}. The argument is routine, so we merely sketch it. One inserts $v^\delta(\cdot,\omega\,;q)$ into \eqref{amp}, adds or subtracts a constant until the resulting function is a supersolution or subsolution, and applies Proposition~\ref{appcmp}. The estimate produced by this argument is~\eqref{vdlip2}. We remark that, due to \eqref{reg} and~\eqref{coer}, the right side of \eqref{vdpdepp} is  controlled by $C|p- q|$ for a constant $C> 0$ depending on an upper bound for $|p| \vee |q|$. 

(iv) The dependence of $\delta v^\delta$ on $\delta$ is also controlled with a simple comparison argument and the help of \eqref{reg} and \eqref{vdlip2}. Set $\lambda := \delta/\eta$ and $w(y):= \lambda v^\delta(y,\omega\,;p)$ and check that
\begin{equation*}\label{}
\eta w + H(p+Dw,y,\omega) \leq \delta v^\delta + H(p+\lambda Dv^\delta,y,\omega) \leq C(1-\lambda) \quad\mbox{in} \ \Rd,
\end{equation*}
where $C=\Pi_p$ defined in~\eqref{pip}. An application of Proposition~\ref{appcmp} yields
\begin{equation}\label{}
\delta v^\delta(\cdot,\omega\,;p) = \eta w \leq \eta v^\eta(\cdot,\omega\,;p) - C(1-\lambda),
\end{equation}
which is half of~\eqref{dvddepd}. The other inequality is obtained via a similar argument.
\end{proof}

\subsection*{Acknowledgements}
The first author was partially supported by NSF Grant DMS-1004645, the second author by the French National Research Agency ANR-12-BS01-0008-01 and the third author by NSF Grant DMS-0901802. 

\bibliographystyle{plain}
\bibliography{HJrates}

\end{document}